\documentclass[reqno, twoside, a4paper]{amsart}

\usepackage{amsmath}
\usepackage{amssymb}
\usepackage{amsthm}
\usepackage{graphicx}
\usepackage{subfig}
\usepackage{enumerate}
\usepackage[all]{xy}
\usepackage{tabu}
\usepackage{float}
\usepackage{paralist}
\usepackage{multicol}
\usepackage{booktabs}
\usepackage{tabulary}
\usepackage{hyperref}

\usepackage{comment}
\usepackage{xspace}
\usepackage{mathtools}

\usepackage{tikz}
\usetikzlibrary{patterns,decorations.pathreplacing,arrows,decorations.pathmorphing,backgrounds,positioning,fit,petri,shapes}

\tikzset{bullet/.style={
shape = circle,fill = black, inner sep = 0pt, outer sep = 0pt, minimum size = 0.35em, line width = 0pt, draw=black!100}}

\tikzset{circle/.style={
shape = circle,fill = none, inner sep = 0pt, outer sep = 0pt, minimum size = 0.35em, line width = 1pt, draw=black!100}}

\tikzset{rectangle/.style={
shape = rectangle,fill = white, inner sep = 0pt, outer sep = 0pt, minimum size = 0.35em, line width = 0pt, draw=black!100}}

\tikzset{empty/.style={
shape = circle,fill = white, inner sep = 0pt, outer sep = 0pt, minimum size = 0.35em, line width = 0pt, draw=white!100}}

\tikzset{xmark/.style={
shape = cross out,fill = white, inner sep = 0pt, outer sep = 0pt, minimum size = 0em, line width = 0pt, draw=white!100}}

\tikzset{longrectangle/.style={
inner sep = 1em,
rectangle,
minimum size=1em,
very thick,
draw=black!100, 
}}

\tikzset{label distance=-0.15em}

\tikzset{font=\scriptsize}

\usepackage{tikz-cd}
\usepackage{tkz-graph}

\newtheorem{theorem}{Theorem}[section]
\newtheorem{lemma}[theorem]{Lemma}
\newtheorem{proposition}[theorem]{Proposition}

\theoremstyle{definition}

\newtheorem{definition}[theorem]{Definition}
\newtheorem{remark}[theorem]{Remark}

\newtheorem{example}[theorem]{Example}

\numberwithin{equation}{section}

\DeclareMathOperator{\Spec}{Spec}
\DeclareMathOperator{\Proj}{Proj}

\DeclareMathOperator{\Def}{Def}
\DeclareMathOperator{\DefQG}{Def^{\mathrm{QG}}}

\newcommand{\plist}[1]{\noindent \ensuremath $\mathbf{#1}$}

\begin{document}

\title[Invariants of deformations of quotient surface singularities]{Invariants of deformations of quotient surface singularities}

\author[B. Han]{Byoungcheon Han}

\address{Department of Mathematics, Chungnam National University, Daejeon 34134, Korea}

\email{bchan@cnu.ac.kr}

\author[J. Jeon]{Jaekwan Jeon}

\address{Department of Mathematics, Chungnam National University, Daejeon 34134, Korea}

\email{jk-jeon@cnu.ac.kr}

\author[D. Shin]{Dongsoo Shin}

\address{Department of Mathematics, Chungnam National University, Daejeon 34134, Korea}

\email{dsshin@cnu.ac.kr}



\begin{abstract}
We find all $P$-resolutions of quotient surface singularities (especially, tetrahedral, octahedral, and icosahedral singularities) together with their dual graphs, which reproduces Jan Steven's list [Manuscripta Math. 1993] of the numbers of $P$-resolutions of each singularities. We then compute the dimensions and Milnor numbers of the corresponding irreducible components of the reduced base spaces of versal deformations of each singularities. Furthermore we realize Milnor fibers as complements of certain divisors (depending only on the singularities) in rational surfaces via the minimal model program for 3-folds. Then we compare Milnor fibers with minimal symplectic fillings, where the latter are classified by Bhupal and Ono [Nagoya Math. J. 2012]. As an application, we show that there are 6 pairs of entries in the list of Bhupal and Ono [Nagoya Math. J. 2012] such that two entries in each pairs represent diffeomorphic minimal symplectic fillings.
\end{abstract}

\maketitle

\section{Introduction}

Let $(X_0,0)$ be a quotient surface singularity and let $\Def(X_0)$ be the reduced base space of a versal deformation of $X_0$.

Koll\'{a}r and Shepherd-Barron~\cite{KSB-1988} prove that there is a one-one correspondence between irreducible components of $\Def(X_0)$ and certain partial resolutions (called \emph{$P$-resolutions}; Definition~\ref{definition:P-resolution}) of $X_0$. Roughly speaking, if $X \to \Delta$ is a smoothing of $X_0$ over a small disk $\Delta$, then there is a $P$-resolution $Y_0 \to X_0$ and a smoothing $Y \to \Delta$ of $Y_0$ which is induced by $\mathbb{Q}$-Gorenstein smoothings of singularities of $Y_0$ such that $Y \to \Delta$ blows down to $X \to \Delta$. In this way they have shown how to find irreducible components of $\Def(X_0)$ and how to compute their dimensions.

Meanwhile, Stevens~(\cite{Stevens-1991, Stevens-1993}) determine all $P$-resolutions of each quotient surface singularities. For cyclic quotient singularities, Stevens~\cite{Stevens-1991} finds out all $P$-resolutions by an inductive procedure using certain continued fractions which represent zero. On the other hand, PPSU~\cite[p.45]{PPSU-2015} provides a direct algorithm recovering $P$-resolutions from the continued fractions via the minimal model program for 3-folds. The study of dihedral singularities is reduced to the cyclic case; Stevens~\cite[\S7]{Stevens-1991}. Finally, Stevens~\cite{Stevens-1993} determines all $P$-resolutions of the remained cases (i.e., tetrahedral, octahedral, icosahedral singularities; \emph{TOI-singularities} in short) and he gives the number of $P$-resolutions of each singularities in \cite[Table~1]{Stevens-1993}; cf. PPSU~\cite[Remark~6.11]{PPSU-2015}.

One of the main results of this paper is reproducing the above list of Stevens~\cite[Table~1]{Stevens-1993} together with their dual graphs.

\begin{theorem}\label{theorem:list}
We find all $P$-resolutions of tetrahedral, octahedral, and icosahedral singularities in Section~\ref{section:list}. And, for each $P$-resolutions, we compute the dimensions and Milnor numbers of the corresponding irreducible components of $\Def(X_0)$.
\end{theorem}

Here, the \emph{Milnor number} of an irreducible component of $\Def(X_0)$ is the second Betti number of a \emph{Milnor fiber} of a smoothing $X \to \Delta$ of $X_0$ contained in the component, where, roughly speaking, a Milnor fiber is a general fiber of a smoothing $X \to \Delta$. Note that every component of $\Def(X_0)$ contains a smoothing because quotient surface singularities are rational singularities; Artin~\cite{Artin-1974}. In Section~\ref{subsection:P-resolution-vs-invariant} we briefly recall how to compute dimensions and Milnor numbers from $P$-resolutions.

\subsection{Milnor fibers as complements of the compactifying divisors}

Using the list in Theorem~\ref{theorem:list}, we realize Milnor fibers of each irreducible components as complements of certain divisors (depending only on the given singularities, called, \emph{compactifying divisors} (See below) embedded in rational surfaces.

We briefly recall how to realize Milnor fibers as complements of compactifying divisors; for details, see PPSU~\cite{PPSU-2015} for example. Let $(X_0,0)$ be a non-cyclic quotient singularity. Then $X_0$ admits a good $\mathbb{C}^{\ast}$-action. That is, if $X_0=\Spec(A)$, then $A$ is a graded ring with non-negative weights. So there is a \emph{singular natural compactification} $\overline{X}_0 = \Proj(A[t])$ where the weight of $t$ is $1$; Pinkham~\cite{Pinkham-1977}. The \emph{singular compactifying divisor} of $X_0$ is the complement $\overline{E}_{\infty}=\overline{X}_0-X_0$. It has been known that there are cyclic quotient singularities on $\overline{E}_{\infty}$ in $\overline{X}_0$.

According to Pinkham~\cite{Pinkham-1977}, every deformation $X \to \Delta$ of $X_0$ can be lifted to a deformation $\overline{X} \to \Delta$ of $\overline{X}_0$ that is locally trivial near $\overline{E}_{\infty}$. So if $X_t$ ($t \neq 0$) is a general fiber of a smoothing $X \to \Delta$ of $X_0$, that is, if $X_t$ is a Milnor fiber of $X \to \Delta$, then there is a deformation $\overline{X} \to \Delta$ of $\overline{X}_0$ that is a lifting of $X \to \Delta$ such that
\begin{equation*}
X_t = \overline{X}_t - \overline{E}_{\infty}.
\end{equation*}
Let us take a simultaneous resolution $\widehat{X} \to \overline{X}$ of $\overline{X} \to \Delta$ along the cyclic quotient singularities on each $\overline{E}_{\infty} \subset \overline{X}_t$. We call the induced deformation $\widehat{X} \to \Delta$ by the \emph{natural compactification}. Let $\widehat{E}_{\infty}$ is the proper transform of $\overline{E}_{\infty}$, which is called the \emph{compactifying divisor} of $X_0$. Then
\begin{equation*}
X_t = \widehat{X}_t - \widehat{E}_{\infty}.
\end{equation*}

Since the complement $X_t (=\widehat{X}_t - \widehat{E}_{\infty})$ is Stein as it is a Milnor fiber, $\widehat{E}_{\infty}$ supports an ample divisor in $\widehat{X}_t$. Furthermore it has been known that $\widehat{X}_t$ is a rational surface; see Theorem~\ref{theorem:rational}. Therefore every $(-1)$-curve in $\widehat{X}_t$ should intersect $\widehat{E}_{\infty}$. Conversely, if the compactifying divisor $\widehat{E}_{\infty}$ is embedded in a rational surface $Z$ supporting an ample divisor, then there is a smoothing $X \to \Delta$ of $X_0$ such that $X_t = Z - \widehat{E}_{\infty}$; Pinkham~\cite[Theorem~6.7]{Pinkham-1978}. That is, once one knows how $(-1)$-curves in $\widehat{X}_t$ intersect $\widehat{E}_{\infty}$, one can realize the Milnor fiber $X_t$ as a complement.

\begin{theorem}
For each $P$-resolutions in Theorem~\ref{theorem:list}, we compute how $(-1)$-curves in $\widehat{X}_t$ intersect $\widehat{E}_{\infty}$.
\end{theorem}

The main ingredient of the computations is the method developed in PPSU~\cite{PPSU-2015} using the minimal model program for 3-folds. We briefly review in Section~\ref{section:MMP}.

\subsection{Milnor fibers and symplectic fillings}

Finally, we compare Milnor fibers associated to $P$-resolutions in Theorem~\ref{theorem:list} with minimal symplectic fillings of the links of quotient surface singularities, where the latter are classified by Bhupal-Ono~\cite{Bhupal-Ono-2012}.

We briefly recall some basics on symplectic fillings of quotient surface singularities; for details, refer Lisca~\cite{Lisca-2008} and Bhupal-Ono~\cite{Bhupal-Ono-2012}. Let $X_0$ is a quotient surface singularity. Suppose that $(X_0,0) \subset (\mathbb{C}^N,0)$. A \emph{link} $L$ of $X_0$ is a boundary of $B_{\epsilon} \cap X_0$, where $B_{\epsilon}$ is a very small ball centered at the origin $0 \in \mathbb{C}^N$ with radius $\epsilon > 0$. A \emph{symplectic filling} of a link $L$ of $X_0$ (or a symplectic filling of $X_0$, for simplicity) is a symplectic $4$-manifold $W$ with $\partial W = L$ satisfying a certain compatible condition on the contact structure of $\partial W$ coming from the symplectic structure on $W$ and the Milnor fillable contact structure of $L$. Note that a Milnor fiber of a smoothing of $X_0$ is a typical example of a minimal symplectic filling of $X_0$.

Lisca~\cite{Lisca-2008} classifies minimal symplectic fillings of lens spaces, where lens spaces can be realized as links of cyclic quotient singularities. He also shows that any minimal symplectic fillings of a cyclic quotient singularity are diffeomorphic to its Milnor fibers. Then Bhupal-Ono~\cite{Bhupal-Ono-2012} classifies minimal symplectic fillings of the remained quotient surface singularities. For dihedral singularities, the classification is reduced to that of cyclic quotient singularities; PPSU~\cite{PPSU-2015}. Bhupal-Ono~\cite{Bhupal-Ono-2012} shows that any minimal symplectic fillings of TOI-singularities are symplectic deformation equivalent to the complements of the compactifying divisors embedded in iterated blow-ups of  $\mathbb{CP}^2$ or $\mathbb{CP}^1 \times \mathbb{CP}^1$. Then they classifies all possible embeddings of the compactifying divisors into iterated blow-ups of  $\mathbb{CP}^2$ or $\mathbb{CP}^1 \times \mathbb{CP}^1$ for each singularities; cf. PPSU~\cite[Remark 4.13]{PPSU-2015}. We briefly summarize their classification in Section~\ref{subsection:classification-symplectic-filling}.

Similar to Lisca~\cite{Lisca-2008}, PPSU~\cite{PPSU-2015} shows that any minimal symplectic fillings of TOI-singularities are diffeomorphic to their Milnor fibers, respectively. Therefore there is a one-one correspondence (up to diffeomorphism type) between $P$-resolutions of quotient surface singularities and their minimal symplectic fillings.

it is natural to compare two classifications of $P$-resolutions and minimal symplectic fillings of each singularities.

\begin{theorem}
\label{theorem:complement}
For each $P$-resolutions in Theorem~\ref{theorem:list}, we realize them as complements of the compactifying divisors and compare them with minimal symplectic fillings classified by Bhupal-Ono~\cite[\S5]{Bhupal-Ono-2012}.
\end{theorem}

\subsection{Applications}

As an application, we reduce the list of Bhupal-Ono~\cite{Bhupal-Ono-2012} of all possible minimal symplectic fillings.

\begin{proposition}
There are 6 pairs of entries in the list of Bhupal--Ono~\cite[\S5]{Bhupal-Ono-2012} such that the minimal symplectic fillings associated to entries in each pairs are diffeomorphic:

\begin{multicols}{2}

\begin{itemize}
\item \#7 and \#8
\item \#129 and \#209
\item \#212 and \#213
\end{itemize}

\columnbreak

\begin{itemize}
\item \#10 and \#11
\item \#132 and \#211
\item \#135 and \#214
\end{itemize}
\end{multicols}
\end{proposition}

\begin{proof}
From the list of $P$-resolutions in Section~\ref{section:list}, we find that there are 6 pairs of $P$-resolutions which are symmetrical to each other:

For $T_{6(5-2)+1}$ singularity,

\begin{multicols}{2}
\begin{itemize}
\item $T_{6(5-2)+1}[3] = BO~\#7$

\begin{tikzpicture}[scale=0.5]
\node[bullet] (00) at (0,0) [label=below:{$-2$}] {};
\node[rectangle] (10) at (1,0) [label=below:{$-2$}] {};

\node[rectangle] (20) at (2,0) [label=below:{$-5$}] {};

\node[bullet] (30) at (3,0) [label=below:{$-2$}] {};
\node[rectangle] (40) at (4,0) [label=below:{$-2$}] {};

\node[bullet] (21) at (2,1) [label=left:{$-2$}] {};

\draw [-] (00)--(10);
\draw [-] (10)--(20);
\draw [-] (20)--(30);
\draw [-] (30)--(40);
\draw [-] (20)--(21);
\end{tikzpicture}
\end{itemize}

\columnbreak

\begin{itemize}
\item $T_{6(5-2)+1}[4] = BO~\#8$

\begin{tikzpicture}[scale=0.5]
\node[rectangle] (00) at (0,0) [label=below:{$-2$}] {};
\node[bullet] (10) at (1,0) [label=below:{$-2$}] {};

\node[rectangle] (20) at (2,0) [label=below:{$-5$}] {};

\node[rectangle] (30) at (3,0) [label=below:{$-2$}] {};
\node[bullet] (40) at (4,0) [label=below:{$-2$}] {};

\node[bullet] (21) at (2,1) [label=left:{$-2$}] {};

\draw [-] (00)--(10);
\draw [-] (10)--(20);
\draw [-] (20)--(30);
\draw [-] (30)--(40);
\draw [-] (20)--(21);
\end{tikzpicture}
\end{itemize}
\end{multicols}

For $T_{6(6-2)+1}$ singularity,

\begin{multicols}{2}
\begin{itemize}
\item $T_{6(6-2)+1}[2] = BO~\#10$

\begin{tikzpicture}[scale=0.5]
\node[rectangle] (00) at (0,0) [label=below:{$-2$}] {};
\node[rectangle] (10) at (1,0) [label=below:{$-2$}] {};

\node[rectangle] (20) at (2,0) [label=below:{$-6$}] {};

\node[bullet] (30) at (3,0) [label=below:{$-2$}] {};
\node[rectangle] (40) at (4,0) [label=below:{$-2$}] {};

\node[bullet] (21) at (2,1) [label=left:{$-2$}] {};

\draw [-] (00)--(10);
\draw [-] (10)--(20);
\draw [-] (20)--(30);
\draw [-] (30)--(40);
\draw [-] (20)--(21);
\end{tikzpicture}
\end{itemize}

\columnbreak

\begin{itemize}
\item $T_{6(6-2)+1}[3] = BO~\#11$

\begin{tikzpicture}[scale=0.5]
\node[rectangle] (00) at (0,0) [label=below:{$-2$}] {};
\node[bullet] (10) at (1,0) [label=below:{$-2$}] {};

\node[rectangle] (20) at (2,0) [label=below:{$-6$}] {};

\node[rectangle] (30) at (3,0) [label=below:{$-2$}] {};
\node[rectangle] (40) at (4,0) [label=below:{$-2$}] {};

\node[bullet] (21) at (2,1) [label=left:{$-2$}] {};

\draw [-] (00)--(10);
\draw [-] (10)--(20);
\draw [-] (20)--(30);
\draw [-] (30)--(40);
\draw [-] (20)--(21);
\end{tikzpicture}
\end{itemize}
\end{multicols}

For $T_{6(3-2)+5}$ singularity,

\begin{multicols}{2}
\begin{itemize}
\item $T_{6(3-2)+5}[2] = BO~\#209$

\begin{tikzpicture}[scale=0.5]
\node[rectangle] (10) at (1,0) [label=below:{$-3$}] {};
\node[rectangle] (20) at (2,0) [label=below:{$-3$}] {};
\node[bullet] (30) at (3,0) [label=below:{$-3$}] {};

\node[bullet] (21) at (2,1) [label=left:{$-2$}] {};

\draw [-] (10)--(20);
\draw [-] (20)--(30);
\draw [-] (20)--(21);
\end{tikzpicture}
\end{itemize}

\columnbreak

\begin{itemize}
\item $T_{6(3-2)+5}[3] = BO~\#129$

\begin{tikzpicture}[scale=0.5]
\node[bullet] (10) at (1,0) [label=below:{$-3$}] {};
\node[rectangle] (20) at (2,0) [label=below:{$-3$}] {};
\node[rectangle] (30) at (3,0) [label=below:{$-3$}] {};

\node[bullet] (21) at (2,1) [label=left:{$-2$}] {};

\draw [-] (10)--(20);
\draw [-] (20)--(30);
\draw [-] (20)--(21);
\end{tikzpicture}
\end{itemize}
\end{multicols}

For $T_{6(4-2)+5}$ singularity,

\begin{multicols}{2}
\begin{itemize}
\item $T_{6(4-2)+5}[3] = BO~\#211$

\begin{tikzpicture}[scale=0.5]
\node[rectangle] (00) at (0,0) [label=below:{$-4$}] {};
\node[bullet] (10) at (1,0) [label=below:{$-1$}] {};
\node[rectangle] (20) at (2,0) [label=below:{$-5$}] {};
\node[bullet] (30) at (3,0) [label=below:{$-3$}] {};

\node[rectangle] (21) at (2,1) [label=left:{$-2$}] {};

\draw [-] (00)--(10);
\draw [-] (10)--(20);
\draw [-] (20)--(30);
\draw [-] (20)--(21);
\end{tikzpicture}
\end{itemize}

\columnbreak

\begin{itemize}
\item $T_{6(4-2)+5}[5] = BO~\#132$

\begin{tikzpicture}[scale=0.5]
\node[bullet] (10) at (1,0) [label=below:{$-3$}] {};
\node[rectangle] (20) at (2,0) [label=below:{$-5$}] {};
\node[bullet] (30) at (3,0) [label=below:{$-1$}] {};
\node[rectangle] (40) at (4,0) [label=below:{$-4$}] {};

\node[rectangle] (21) at (2,1) [label=left:{$-2$}] {};

\draw [-] (10)--(20);
\draw [-] (20)--(30);
\draw [-] (30)--(40);
\draw [-] (20)--(21);
\end{tikzpicture}
\end{itemize}
\end{multicols}
and
\begin{multicols}{2}
\begin{itemize}
\item $T_{6(4-2)+5}[4] = BO~\#213$

\begin{tikzpicture}[scale=0.5]
\node[rectangle] (00) at (0,0) [label=below:{$-4$}] {};
\node[bullet] (10) at (1,0) [label=below:{$-1$}] {};
\node[rectangle] (20) at (2,0) [label=below:{$-5$}] {};
\node[rectangle] (30) at (3,0) [label=below:{$-3$}] {};

\node[rectangle] (21) at (2,1) [label=left:{$-2$}] {};

\draw [-] (00)--(10);
\draw [-] (10)--(20);
\draw [-] (20)--(30);
\draw [-] (20)--(21);
\end{tikzpicture}
\end{itemize}

\columnbreak

\begin{itemize}
\item $T_{6(4-2)+5}[6] = BO~\#212$

\begin{tikzpicture}[scale=0.5]
\node[rectangle] (10) at (1,0) [label=below:{$-3$}] {};
\node[rectangle] (20) at (2,0) [label=below:{$-5$}] {};
\node[bullet] (30) at (3,0) [label=below:{$-1$}] {};
\node[rectangle] (40) at (4,0) [label=below:{$-4$}] {};

\node[rectangle] (21) at (2,1) [label=left:{$-2$}] {};

\draw [-] (10)--(20);
\draw [-] (20)--(30);
\draw [-] (30)--(40);
\draw [-] (20)--(21);
\end{tikzpicture}
\end{itemize}
\end{multicols}

For $T_{6(5-2)+5}$ singularity,

\begin{multicols}{2}
\begin{itemize}
\item $T_{6(5-2)+5}[3] = BO~\#214$

\begin{tikzpicture}[scale=0.5]
\node[rectangle] (10) at (1,0) [label=below:{$-3$}] {};
\node[rectangle] (20) at (2,0) [label=below:{$-5$}] {};
\node[bullet] (30) at (3,0) [label=below:{$-3$}] {};

\node[rectangle] (21) at (2,1) [label=left:{$-2$}] {};

\draw [-] (10)--(20);
\draw [-] (20)--(30);
\draw [-] (20)--(21);
\end{tikzpicture}
\end{itemize}

\columnbreak

\begin{itemize}
\item $T_{6(5-2)+5}[4] = BO~\#135$

\begin{tikzpicture}[scale=0.5]
\node[bullet] (10) at (1,0) [label=below:{$-3$}] {};
\node[rectangle] (20) at (2,0) [label=below:{$-5$}] {};
\node[rectangle] (30) at (3,0) [label=below:{$-3$}] {};

\node[rectangle] (21) at (2,1) [label=left:{$-2$}] {};

\draw [-] (10)--(20);
\draw [-] (20)--(30);
\draw [-] (20)--(21);
\end{tikzpicture}
\end{itemize}
\end{multicols}

Since two $P$-resolutions symmetrical to each other give us diffeomorphic Milnor fibers, the corresponding minimal symplectic fillings are  diffeomorphic to each other.
\end{proof}

\begin{remark}
PPSU~\cite[Theorem~5.5]{PPSU-2015} shows that the above are the only pairs consisting of diffeomorphic entries.
\end{remark}

\subsection*{Acknowledgements}

This is a part of M.S. Theses of BC and JK presented at Department of Mathematics, Chungnam National University, Daejeon, Korea in 2017. BC and JK were supported by Basic Science Research Program through the National Research Foundation of Korea funded by the Ministry of Education (NRF-2015R1D1A1A01060476). DS was supported by Basic Science Research Program through the National Research Foundation of Korea funded by the Ministry of Education (2013R1A1A2010613).

\section{Generalities on quotient surface singularities}
\label{section:generalities}

We briefly recall some basics on quotient surface singularities for fixing notations.

According to Riemenschneider~\cite{Riemenschneider-1981}, quotient surface singularities are classified by five types; cyclic, dihedral, tetrahedral, octahedral, icosahedral singularities. Especially, TOI-singularities are denoted by $T_{m}$, $O_{m}$, $I_{m}$, respectively. For the dual graphs of the minimal resolutions of non-cyclic quotient surface singularities, refer Bhupal-Ono~\cite{Bhupal-Ono-2012} for example. Instead we divide the dual graphs of the minimal resolutions of non-cyclic quotient surface singularities into three types as in Bhupal-Ono~\cite{Bhupal-Ono-2012}; Figure~\ref{figure:non-cyclic-minimal-resolution}. The dual graphs have one node representing a rational curve with self-intersection number $-b$ and three arms representing the minimal resolutions of cyclic quotient singularities.

\begin{figure}[t]
\centering
\subfloat [dihedral] {%
\begin{tikzpicture}[scale=0.75]
 \node[bullet] (00) at (0,0) [label=below:{$-2$}] {};
 \node[bullet] (10) at (1,0) [label=below:{$-b$}] {};
 \node[bullet] (20) at (2,0) [label=below:{$-b_1$}] {};

\node[empty] (250) at (2.5,0) [] {};
\node[empty] (30) at (3,0) [] {};

 \node[bullet] (350) at (3.5,0) [label=below:{$-b_{r}$}] {};

 \node[bullet] (11) at (1,1) [label=left:{$-2$}] {};

\draw [-] (00)--(10);
\draw [-] (10)--(20);
\draw [-] (20)--(250);
\draw [dotted] (20)--(350);
\draw [-] (30)--(350);

\draw [-] (10)--(11);
\end{tikzpicture}}
\\
\subfloat [type $(3,2)$] {%
\begin{tikzpicture}[scale=0.75]
 \node[bullet] (-10) at (-1,0) [label=below:{$-2$}] {};
 \node[bullet] (00) at (0,0) [label=below:{$-2$}] {};
 \node[bullet] (10) at (1,0) [label=below:{$-b$}] {};
 \node[bullet] (20) at (2,0) [label=below:{$-b_1$}] {};

\node[empty] (250) at (2.5,0) [] {};
\node[empty] (30) at (3,0) [] {};

 \node[bullet] (350) at (3.5,0) [label=below:{$-b_r$}] {};

 \node[bullet] (11) at (1,1) [label=left:{$-2$}] {};

\draw [-] (-10)--(00);
\draw [-] (00)--(10);
\draw [-] (10)--(20);
\draw [-] (20)--(250);
\draw [dotted] (20)--(350);
\draw [-] (30)--(350);

\draw [-] (10)--(11);
\end{tikzpicture}}
\qquad \qquad
\subfloat [type $(3,1)$] {%
\begin{tikzpicture}[scale=0.75]
 \node[bullet] (00) at (0,0) [label=below:{$-3$}] {};
 \node[bullet] (10) at (1,0) [label=below:{$-b$}] {};
 \node[bullet] (20) at (2,0) [label=below:{$-b_1$}] {};

\node[empty] (250) at (2.5,0) [] {};
\node[empty] (30) at (3,0) [] {};

 \node[bullet] (350) at (3.5,0) [label=below:{$-b_r$}] {};

 \node[bullet] (11) at (1,1) [label=left:{$-2$}] {};

\draw [-] (00)--(10);
\draw [-] (10)--(20);
\draw [-] (20)--(250);
\draw [dotted] (20)--(350);
\draw [-] (30)--(350);

\draw [-] (10)--(11);
\end{tikzpicture}}

\caption{The dual graphs of the minimal resolutions of non-cyclic quotient  singularities}
\label{figure:non-cyclic-minimal-resolution}
\end{figure}
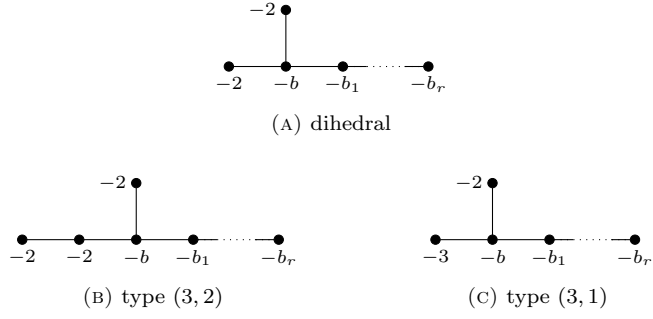

\subsection{Compactifying divisors}
\label{subsection:compactifying-divisor}

Let $X_0$ be a non-cyclic quotient singularity. Suppose the dual graphs of the minimal resolution of $X_0$ has three arms representing the minimal resolutions of cyclic quotient singularities of type $\frac{1}{n_i}(1,q_i)$ ($i=1,2,3$) whose dual graph is given as in Figure~\ref{figure:non-cyclic-minimal-resolution}. That is, we have
\begin{equation*}
\frac{n_i}{q_i} = [b_{i1}, \dotsc, b_{ir}],
\end{equation*}
where $[c_1, \dotsc, c_t]$ ($c_i \ge 2$) represents Hirzebruch-Jung continued fraction, i.e.,
\begin{equation*}
[c_1, \dotsc, c_t] := c_1 - \cfrac{1}{c_2-\cfrac{1}{\ddots - \cfrac{1}{c_t}}}.
\end{equation*}
Then the singular compactifying divisor $\overline{E}_{\infty}$ is a singular rational curve with three cyclic quotient singularities of dual type $\frac{1}{n_i}(1, n_i-q_i)$ ($i=1,2,3$); Pinkham~\cite{Pinkham-1977}. Therefore the dual graphs of the compactifying divisors $\widehat{E}_{\infty}$ of non-cyclic quotient surface singularities are given as in Figure~\ref{figure:dual-graph-of-compactifying-divisor}. Let $\widetilde{X}_0$ be the smooth surface obtained by resolving all singularities on the natural compactification $\overline{X}_0$. According to Pinkham~\cite{Pinkham-1977}, $\widetilde{X}_0$ has the configurations of rational curves whose dual graph is given as in Figure~\ref{figure:non-cyclic-widetilde(X)}, where
\begin{equation*}
\frac{n_i}{n_i-q_i} = [a_{i1}, \dotsc, a_{ie}].
\end{equation*}

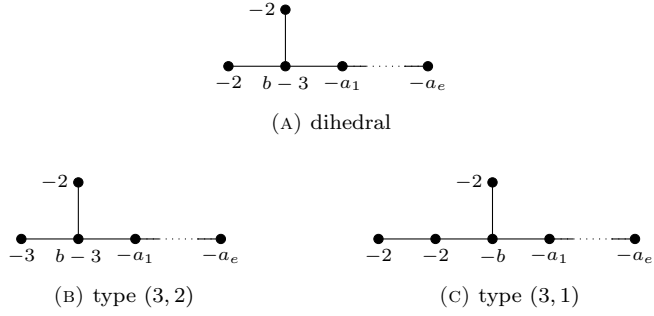
\begin{figure}[t]
\centering
\subfloat [dihedral] {%
\begin{tikzpicture}[scale=0.75]
 \node[bullet] (00) at (0,0) [label=below:{$-2$}] {};
 \node[bullet] (10) at (1,0) [label=below:{$b-3$}] {};
 \node[bullet] (20) at (2,0) [label=below:{$-a_1$}] {};

\node[empty] (250) at (2.5,0) [] {};
\node[empty] (30) at (3,0) [] {};

 \node[bullet] (350) at (3.5,0) [label=below:{$-a_e$}] {};

 \node[bullet] (11) at (1,1) [label=left:{$-2$}] {};

\draw [-] (00)--(10);
\draw [-] (10)--(20);
\draw [-] (20)--(250);
\draw [dotted] (20)--(350);
\draw [-] (30)--(350);

\draw [-] (10)--(11);
\end{tikzpicture}}
\\
\subfloat [type $(3,2)$] {%
\begin{tikzpicture}[scale=0.75]
 \node[bullet] (00) at (0,0) [label=below:{$-3$}] {};
 \node[bullet] (10) at (1,0) [label=below:{$b-3$}] {};
 \node[bullet] (20) at (2,0) [label=below:{$-a_1$}] {};

\node[empty] (250) at (2.5,0) [] {};
\node[empty] (30) at (3,0) [] {};

 \node[bullet] (350) at (3.5,0) [label=below:{$-a_e$}] {};

 \node[bullet] (11) at (1,1) [label=left:{$-2$}] {};

\draw [-] (00)--(10);
\draw [-] (10)--(20);
\draw [-] (20)--(250);
\draw [dotted] (20)--(350);
\draw [-] (30)--(350);

\draw [-] (10)--(11);
\end{tikzpicture}}
\qquad \qquad
\subfloat [type $(3,1)$] {%
\begin{tikzpicture}[scale=0.75]
 \node[bullet] (-10) at (-1,0) [label=below:{$-2$}] {};
 \node[bullet] (00) at (0,0) [label=below:{$-2$}] {};
 \node[bullet] (10) at (1,0) [label=below:{$-b$}] {};
 \node[bullet] (20) at (2,0) [label=below:{$-a_1$}] {};

\node[empty] (250) at (2.5,0) [] {};
\node[empty] (30) at (3,0) [] {};

 \node[bullet] (350) at (3.5,0) [label=below:{$-a_e$}] {};

 \node[bullet] (11) at (1,1) [label=left:{$-2$}] {};

\draw [-] (-10)--(00);
\draw [-] (00)--(10);
\draw [-] (10)--(20);
\draw [-] (20)--(250);
\draw [dotted] (20)--(350);
\draw [-] (30)--(350);

\draw [-] (10)--(11);
\end{tikzpicture}}

\caption{The dual graphs of the compactifying divisors of non-cyclic quotient  singularities}
\label{figure:dual-graph-of-compactifying-divisor}
\end{figure}

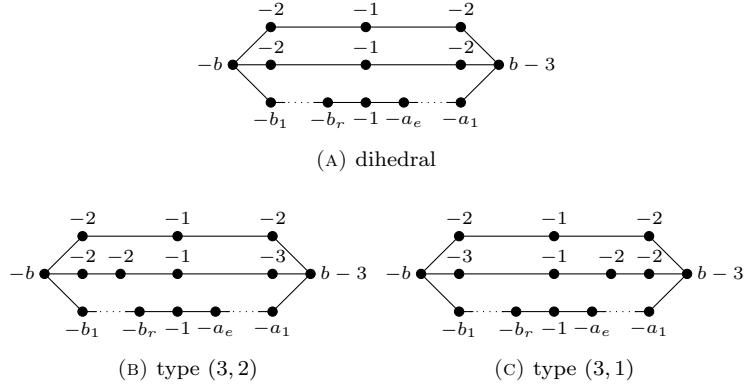
\begin{figure}[t]
\centering

\subfloat [dihedral] {%
\begin{tikzpicture}[scale=0.5]
 \node[bullet] (00) at (0,0) [label=left:{$-b$}] {};
 \node[bullet] (70) at (7,0) [label=right:{$b-3$}] {};

 \node[bullet] (11) at (1,1) [label=above:{$-2$}] {};

 \node[bullet] (351) at (3.5,1) [label=above:{$-1$}] {};

 \node[bullet] (61) at (6,1) [label=above:{$-2$}] {};

\draw [-] (00)--(11)--(351)--(61)--(70);

 \node[bullet] (10) at (1,0) [label=above:{$-2$}] {};

 \node[bullet] (350) at (3.5,0) [label=above:{$-1$}] {};

 \node[bullet] (60) at (6,0) [label=above:{$-2$}] {};

\draw [-] (00)--(10)--(350)--(60)--(70);

 \node[bullet] (1-1) at (1,-1) [label=below:{$-b_1$}] {};
\node[empty] (15-1) at (1.5,-1) [] {};
\node[empty] (2-1) at (2,-1) [] {};
 \node[bullet] (25-1) at (2.5,-1) [label=below:{$-b_r$}] {};

\draw [-] (00)--(1-1);
\draw [-] (1-1)--(15-1);
\draw [dotted] (1-1)--(25-1);
\draw [-] (2-1)--(25-1);

 \node[bullet] (35-1) at (3.5,-1) [label=below:{$-1$}] {};
 \node[bullet] (45-1) at (4.5,-1) [label=below:{$-a_e$}] {};
\node[empty] (5-1) at (5,-1) [] {};
\node[empty] (55-1) at (5.5,-1) [] {};
 \node[bullet] (6-1) at (6,-1) [label=below:{$-a_1$}] {};

\draw [-] (25-1)--(35-1);
\draw [-] (35-1)--(45-1);
\draw [-] (45-1)--(5-1);
\draw [dotted] (45-1)--(6-1);
\draw [-] (55-1)--(6-1);
\draw [-] (6-1)--(70);
\end{tikzpicture}}
\\
\subfloat [type $(3,2)$] {%
\begin{tikzpicture}[scale=0.5]
 \node[bullet] (00) at (0,0) [label=left:{$-b$}] {};
 \node[bullet] (70) at (7,0) [label=right:{$b-3$}] {};

 \node[bullet] (11) at (1,1) [label=above:{$-2$}] {};
 \node[bullet] (351) at (3.5,1) [label=above:{$-1$}] {};
 \node[bullet] (61) at (6,1) [label=above:{$-2$}] {};

\draw [-] (00)--(11)--(351)--(61)--(70);

 \node[bullet] (10) at (1,0) [label=above:{$-2$}] {};
 \node[bullet] (20) at (2,0) [label=above:{$-2$}] {};
 \node[bullet] (350) at (3.5,0) [label=above:{$-1$}] {};
 \node[bullet] (60) at (6,0) [label=above:{$-3$}] {};

\draw [-] (00)--(10)--(20)--(350)--(60)--(70);

 \node[bullet] (1-1) at (1,-1) [label=below:{$-b_1$}] {};
\node[empty] (15-1) at (1.5,-1) [] {};
\node[empty] (2-1) at (2,-1) [] {};
 \node[bullet] (25-1) at (2.5,-1) [label=below:{$-b_r$}] {};

\draw [-] (00)--(1-1);
\draw [-] (1-1)--(15-1);
\draw [dotted] (1-1)--(25-1);
\draw [-] (2-1)--(25-1);

 \node[bullet] (35-1) at (3.5,-1) [label=below:{$-1$}] {};
 \node[bullet] (45-1) at (4.5,-1) [label=below:{$-a_e$}] {};
\node[empty] (5-1) at (5,-1) [] {};
\node[empty] (55-1) at (5.5,-1) [] {};
 \node[bullet] (6-1) at (6,-1) [label=below:{$-a_1$}] {};

\draw [-] (25-1)--(35-1);
\draw [-] (35-1)--(45-1);
\draw [-] (45-1)--(5-1);
\draw [dotted] (45-1)--(6-1);
\draw [-] (55-1)--(6-1);
\draw [-] (6-1)--(70);
\end{tikzpicture}}
\subfloat [type $(3,1)$] {%
\begin{tikzpicture}[scale=0.5]
 \node[bullet] (00) at (0,0) [label=left:{$-b$}] {};
 \node[bullet] (70) at (7,0) [label=right:{$b-3$}] {};

 \node[bullet] (11) at (1,1) [label=above:{$-2$}] {};
 \node[bullet] (351) at (3.5,1) [label=above:{$-1$}] {};
 \node[bullet] (61) at (6,1) [label=above:{$-2$}] {};

\draw [-] (00)--(11)--(351)--(61)--(70);

 \node[bullet] (10) at (1,0) [label=above:{$-3$}] {};
 \node[bullet] (350) at (3.5,0) [label=above:{$-1$}] {};
 \node[bullet] (50) at (5,0) [label=above:{$-2$}] {};
 \node[bullet] (60) at (6,0) [label=above:{$-2$}] {};

\draw [-] (00)--(10)--(350)--(50)--(60)--(70);

 \node[bullet] (1-1) at (1,-1) [label=below:{$-b_1$}] {};
\node[empty] (15-1) at (1.5,-1) [] {};
\node[empty] (2-1) at (2,-1) [] {};
 \node[bullet] (25-1) at (2.5,-1) [label=below:{$-b_r$}] {};

\draw [-] (00)--(1-1);
\draw [-] (1-1)--(15-1);
\draw [dotted] (1-1)--(25-1);
\draw [-] (2-1)--(25-1);

 \node[bullet] (35-1) at (3.5,-1) [label=below:{$-1$}] {};
 \node[bullet] (45-1) at (4.5,-1) [label=below:{$-a_e$}] {};
\node[empty] (5-1) at (5,-1) [] {};
\node[empty] (55-1) at (5.5,-1) [] {};
 \node[bullet] (6-1) at (6,-1) [label=below:{$-a_1$}] {};

\draw [-] (25-1)--(35-1);
\draw [-] (35-1)--(45-1);
\draw [-] (45-1)--(5-1);
\draw [dotted] (45-1)--(6-1);
\draw [-] (55-1)--(6-1);
\draw [-] (6-1)--(70);
\end{tikzpicture}}

\caption{The dual graph of $\widetilde{X}_0$ for non-cyclic quotient singularities.}

\label{figure:non-cyclic-widetilde(X)}
\end{figure}

\subsection{$P$-resolutions and invariants of irreducible components}
\label{subsection:P-resolution-vs-invariant}

A \emph{singularity of class $T$} is a cyclic quotient singularity of type $\dfrac{1}{dn^2}(1, dna-1)$ with $d, a \ge 1$, $n \ge 2 $, $(n,a)=1$. A \emph{Wahl singularity} is a singularity of class $T$ with $d=1$. The dual graphs of singularities of class $T$ can be described inductively as follows:

\begin{enumerate}[(i)]
\item The singularities
\begin{tikzpicture}
 \node[bullet] (10) at (1,0) [label=above:{$-4$}] {};
\end{tikzpicture}
and
\begin{tikzpicture}[scale=0.5]
 \node[bullet] (10) at (1,0) [label=above:{$-3$}] {};
 \node[bullet] (20) at (2,0) [label=above:{$-2$}] {};

\node[empty] (250) at (2.5,0) [] {};
\node[empty] (30) at (3,0) [] {};

 \node[bullet] (350) at (3.5,0) [label=above:{$-2$}] {};
 \node[bullet] (450) at (4.5,0) [label=above:{$-3$}] {};

\draw [-] (10)--(20);
\draw [-] (20)--(250);
\draw [dotted] (20)--(350);
\draw [-] (30)--(350);
\draw [-] (350)--(450);
\end{tikzpicture}
are of class $T$

\item If the singularity
\begin{tikzpicture}[scale=0.5]
 \node[bullet] (20) at (2,0) [label=above:{$-b_1$}] {};

\node[empty] (250) at (2.5,0) [] {};
\node[empty] (30) at (3,0) [] {};

 \node[bullet] (350) at (3.5,0) [label=above:{$-b_r$}] {};

\draw [-] (20)--(250);
\draw [dotted] (20)--(350);
\draw [-] (30)--(350);
\end{tikzpicture}
is of class $T$, then so are
\begin{equation*}
\begin{tikzpicture}
 \node[bullet] (10) at (1,0) [label=above:{$-2$}] {};
 \node[bullet] (20) at (2,0) [label=above:{$-b_1$}] {};

\node[empty] (250) at (2.5,0) [] {};
\node[empty] (30) at (3,0) [] {};

 \node[bullet] (350) at (3.5,0) [label=above:{$-b_{r-1}$}] {};
 \node[bullet] (450) at (4.5,0) [label=above:{$-b_r-1$}] {};

\draw [-] (10)--(20);
\draw [-] (20)--(250);
\draw [dotted] (20)--(350);
\draw [-] (30)--(350);
\draw [-] (350)--(450);
\end{tikzpicture}
\text{~and~}
\begin{tikzpicture}
 \node[bullet] (10) at (1,0) [label=above:{$-b_1-1$}] {};
 \node[bullet] (20) at (2,0) [label=above:{$-b_2$}] {};

\node[empty] (250) at (2.5,0) [] {};
\node[empty] (30) at (3,0) [] {};

 \node[bullet] (350) at (3.5,0) [label=above:{$-b_r$}] {};
 \node[bullet] (450) at (4.5,0) [label=above:{$-2$}] {};

\draw [-] (10)--(20);
\draw [-] (20)--(250);
\draw [dotted] (20)--(350);
\draw [-] (30)--(350);
\draw [-] (350)--(450);
\end{tikzpicture}
\end{equation*}

\item Every singularity of class $T$ that is not a rational double point can be obtained by starting with one of the singularities described in (i) and iterating the steps described in (ii).
\end{enumerate}

\begin{definition}[{KSB~\cite[Definition~3.8]{KSB-1988}}, cf. Stevens~\cite{Stevens-1993}]
\label{definition:P-resolution}
A \emph{$P$-resolution} $f \colon Y_0 \rightarrow X_0$ of a quotient singularity $X_0$ is a modification such that $Y_0$ has at most rational double points or singularities of class $T$ as singularities, and $K_{Y_0} \cdot E_i >0$ for all exceptional divisors $E_i$ of $f$.
\end{definition}

\begin{remark}[cf.~{Stevens~\cite[p.8]{Stevens-1993}}]
Let $\widetilde{Y}_0 \to Y_0$ be the minimal resolutions of $Y_0$. Let $W_j$ be a neighborhood of a singularity of class $T$ $P_j$ (if any) on $Y_0$. Then the intersection condition $K_{Y_0} \cdot E_i > 0$ is equivalent to the following conditions: (i) Every $(-1)$-curve of $\widetilde{Y}_0$ intersects two exceptional curve $F_1$ and $F_2$ in $\widetilde{Y}_0$ of singularities $P_1$ and $P_2$ on $Y_0$. (ii) The sum of coefficients $c_1$ and $c_2$ of $F_1$ and $F_2$ in the canonical divisors $K_{W_1}$ and $K_{W_2}$ is less than $-1$.
\end{remark}

Stevens~\cite{Stevens-1991, Stevens-1993} provides an algorithm for finding all $P$-resolutions of a given non-cyclic quotient surface singularity and he presents the number of $P$-resolutions for each singularities.

\begin{proposition}[{Stevens~\cite[Table~1]{Stevens-1993}, cf.~PPSU~\cite[Remark~6.11]{PPSU-2015}}]
The number of $P$-resolutions of each TOI-singularities is given in Table~\ref{table:P-resolution}.
\end{proposition}

\begin{table}
\begin{tabular}{l|cccccccc}
\toprule
Type & $n < 4$ & $n=4$ & $n=5$ & $n=6$ & $n=7$ & $n=8$ & $n=9$ & $n > 9$ \tabularnewline
\midrule
$T_{6(n-2)+1}$ & 1 & 2 & 4 & 3 & 1 & 1 & 1 & 1 \tabularnewline
$T_{6(n-3)+3}$ & 1 & 2 & 4 & 6 & 2 & 1 & 1 & 1 \tabularnewline
$T_{6(n-4)+5}$ &   & 2 & 4 & 6 & 4 & 1 & 1 & 1 \tabularnewline
\midrule
$O_{12(n-2)+1}$ & 1 & 2 & 4 & 3 & 2 & 1 & 1 & 1 \tabularnewline
$O_{12(n-3)+5}$ & 1 & 2 & 4 & 6 & 3 & 2 & 1 & 1 \tabularnewline
$O_{12(n-4)+7}$ &   & 2 & 4 & 4 & 5 & 3 & 2 & 2 \tabularnewline
$O_{12(n-5)+11}$&   &   & 3 & 5 & 4 & 4 & 2 & 2 \tabularnewline
\midrule
$I_{30(n-2)+1}$ & 1 & 2 & 4 & 3 & 2 & 2 & 1 & 1 \tabularnewline
$I_{30(n-3)+7}$ & 1 & 2 & 4 & 7 & 4 & 1 & 1 & 1 \tabularnewline
$I_{30(n-3)+11}$& 1 & 2 & 4 & 6 & 3 & 3 & 2 & 1 \tabularnewline
$I_{30(n-3)+13}$& 1 & 2 & 5 & 6 & 2 & 1 & 1 & 1 \tabularnewline
$I_{30(n-4)+17}$&   & 2 & 4 & 7 & 7 & 2 & 1 & 1 \tabularnewline
$I_{30(n-5)+19}$&   &   & 2 & 2 & 2 & 3 & 2 & 1 \tabularnewline
$I_{30(n-4)+23}$&   & 2 & 5 & 8 & 5 & 1 & 1 & 1 \tabularnewline
$I_{30(n-6)+29}$&   &   &   & 3 & 2 & 3 & 3 & 1 \tabularnewline
\bottomrule
\end{tabular}
\caption{The number of $P$-resolutions; {Stevens~\cite[Table~1]{Stevens-1993}, cf.~PPSU~\cite[Remark~6.11]{PPSU-2015}}}
\label{table:P-resolution}
\end{table}

In Section~\ref{section:list} we present all $P$-resolutions with their dual graphs of minimal resolutions.

Below are examples of $P$-resolutions of TOI-singularities. We present the dual graphs of the minimal resolutions of $P$-resolutions, where a connected linear chain of vertices decorated by a rectangle $\square$ denotes curves on the minimal resolution of a $P$-resolution which are contracted to a rational double point or a singularity of class $T$ on the $P$-resolution.

\begin{example}[$T_{6(5-2)+1}$-singularity]
\label{example:T_6(5-2)+1}
$T_{6(5-2)+1}$-singularity has four $P$-resolutions:

\begin{center}
\begin{tikzpicture}[scale=0.5]
\node[empty] (-21) at (-2,1) [] {[1]};

\node[rectangle] (-20) at (-2,0) [label=below:{$-2$}] {};
\node[rectangle] (-10) at (-1,0) [label=below:{$-2$}] {};
\node[bullet] (00) at (0,0) [label=below:{$-5$}] {};
\node[rectangle] (10) at (1,0) [label=below:{$-2$}] {};
\node[rectangle] (20) at (2,0) [label=below:{$-2$}] {};

\node[rectangle] (01) at (0,1) [label=left:{$-2$}] {};

\draw [-] (00)--(-10)--(-20);
\draw [-] (00)--(01);
\draw [-] (00)--(10)--(20);
\end{tikzpicture}
\begin{tikzpicture}[scale=0.5]
\node[empty] (-21) at (-2,1) [] {[2]};

\node[rectangle] (-20) at (-2,0) [label=below:{$-2$}] {};
\node[bullet] (-10) at (-1,0) [label=below:{$-2$}] {};
\node[rectangle] (00) at (0,0) [label=below:{$-5$}] {};
\node[bullet] (10) at (1,0) [label=below:{$-2$}] {};
\node[rectangle] (20) at (2,0) [label=below:{$-2$}] {};

\node[rectangle] (01) at (0,1) [label=left:{$-2$}] {};

\draw [-] (00)--(-10)--(-20);
\draw [-] (00)--(01);
\draw [-] (00)--(10)--(20);
\end{tikzpicture}
\begin{tikzpicture}[scale=0.5]
\node[empty] (-21) at (-2,1) [] {[3]};

\node[bullet] (-20) at (-2,0) [label=below:{$-2$}] {};
\node[rectangle] (-10) at (-1,0) [label=below:{$-2$}] {};
\node[rectangle] (00) at (0,0) [label=below:{$-5$}] {};
\node[bullet] (10) at (1,0) [label=below:{$-2$}] {};
\node[rectangle] (20) at (2,0) [label=below:{$-2$}] {};

\node[bullet] (01) at (0,1) [label=left:{$-2$}] {};

\draw [-] (00)--(-10)--(-20);
\draw [-] (00)--(01);
\draw [-] (00)--(10)--(20);
\end{tikzpicture}
\begin{tikzpicture}[scale=0.5]
\node[empty] (-21) at (-2,1) [] {[4]};

\node[rectangle] (-20) at (-2,0) [label=below:{$-2$}] {};
\node[bullet] (-10) at (-1,0) [label=below:{$-2$}] {};
\node[rectangle] (00) at (0,0) [label=below:{$-5$}] {};
\node[rectangle] (10) at (1,0) [label=below:{$-2$}] {};
\node[bullet] (20) at (2,0) [label=below:{$-2$}] {};

\node[bullet] (01) at (0,1) [label=left:{$-2$}] {};

\draw [-] (00)--(-10)--(-20);
\draw [-] (00)--(01);
\draw [-] (00)--(10)--(20);
\end{tikzpicture}
\end{center}

\end{example}

\begin{example}[$O_{12(3-2)+7}$-singularity]
\label{example:O_12(3-2)+7}
There are four $P$-resolutions for $O_{12(3-2)+7}$-singularity:

\begin{center}
\begin{tikzpicture}[scale=0.5]
\node[empty] (-21) at (-2,1) [] {[1]};

\node[rectangle] (-20) at (-2,0) [label=below:{$-2$}] {};
\node[rectangle] (-10) at (-1,0) [label=below:{$-2$}] {};
\node[bullet] (00) at (0,0) [label=below:{$-3$}] {};
\node[bullet] (10) at (1,0) [label=below:{$-4$}] {};

\node[rectangle] (01) at (0,1) [label=left:{$-2$}] {};

\draw [-] (00)--(-10)--(-20);
\draw [-] (00)--(01);
\draw [-] (00)--(10);
\end{tikzpicture}
\begin{tikzpicture}[scale=0.5]
\node[empty] (-21) at (-2,1) [] {[2]};

\node[rectangle] (-20) at (-2,0) [label=below:{$-2$}] {};
\node[rectangle] (-10) at (-1,0) [label=below:{$-2$}] {};
\node[bullet] (00) at (0,0) [label=below:{$-3$}] {};
\node[rectangle] (10) at (1,0) [label=below:{$-4$}] {};

\node[rectangle] (01) at (0,1) [label=left:{$-2$}] {};

\draw [-] (00)--(-10)--(-20);
\draw [-] (00)--(01);
\draw [-] (00)--(10);
\end{tikzpicture}
\begin{tikzpicture}[scale=0.5]
\node[empty] (-21) at (-2,1) [] {[3]};

\node[rectangle] (-20) at (-2,0) [label=below:{$-2$}] {};
\node[bullet] (-10) at (-1,0) [label=below:{$-2$}] {};
\node[rectangle] (00) at (0,0) [label=below:{$-3$}] {};
\node[rectangle] (10) at (1,0) [label=below:{$-4$}] {};

\node[rectangle] (01) at (0,1) [label=left:{$-2$}] {};

\draw [-] (00)--(-10)--(-20);
\draw [-] (00)--(01);
\draw [-] (00)--(10);
\end{tikzpicture}
\begin{tikzpicture}[scale=0.5]
\node[empty] (-21) at (-2,1) [] {[4]};

\node[bullet] (-20) at (-2,0) [label=below:{$-2$}] {};
\node[rectangle] (-10) at (-1,0) [label=below:{$-2$}] {};
\node[rectangle] (00) at (0,0) [label=below:{$-3$}] {};
\node[rectangle] (10) at (1,0) [label=below:{$-4$}] {};

\node[bullet] (01) at (0,1) [label=left:{$-2$}] {};

\draw [-] (00)--(-10)--(-20);
\draw [-] (00)--(01);
\draw [-] (00)--(10);
\end{tikzpicture}
\end{center}
\end{example}

\begin{example}[$I_{30(4-2)+17}$-singularity]
\label{example:I_30(4-2)+17}
It has seven $P$-resolutions:

\begin{center}
\begin{tikzpicture}[scale=0.5]
\node[empty] (-21) at (-2,1) [] {[1]};

\node[bullet] (-10) at (-1,0) [label=below:{$-3$}] {};
\node[bullet] (00) at (0,0) [label=below:{$-4$}] {};
\node[rectangle] (10) at (1,0) [label=below:{$-2$}] {};
\node[bullet] (20) at (2,0) [label=below:{$-3$}] {};

\node[rectangle] (01) at (0,1) [label=left:{$-2$}] {};

\draw [-] (00)--(-10);
\draw [-] (00)--(01);
\draw [-] (00)--(10)--(20);
\end{tikzpicture}
\begin{tikzpicture}[scale=0.5]
\node[empty] (-21) at (-2,1) [] {[2]};

\node[bullet] (-10) at (-1,0) [label=below:{$-3$}] {};
\node[rectangle] (00) at (0,0) [label=below:{$-4$}] {};
\node[bullet] (10) at (1,0) [label=below:{$-2$}] {};
\node[bullet] (20) at (2,0) [label=below:{$-3$}] {};

\node[bullet] (01) at (0,1) [label=left:{$-2$}] {};

\draw [-] (00)--(-10);
\draw [-] (00)--(01);
\draw [-] (00)--(10)--(20);
\end{tikzpicture}
\begin{tikzpicture}[scale=0.5]
\node[empty] (-21) at (-2,1) [] {[3]};

\node[rectangle] (-20) at (-2,0) [label=below:{$-4$}] {};
\node[bullet] (-10) at (-1,0) [label=below:{$-1$}] {};
\node[rectangle] (00) at (0,0) [label=below:{$-5$}] {};
\node[bullet] (10) at (1,0) [label=below:{$-2$}] {};
\node[bullet] (20) at (2,0) [label=below:{$-3$}] {};

\node[rectangle] (01) at (0,1) [label=left:{$-2$}] {};

\draw [-] (00)--(-10)--(-20);
\draw [-] (00)--(01);
\draw [-] (00)--(10)--(20);
\end{tikzpicture}
\begin{tikzpicture}[scale=0.5]
\node[empty] (-21) at (-2,1) [] {[4]};

\node[rectangle] (-20) at (-2,0) [label=below:{$-4$}] {};
\node[bullet] (-10) at (-1,0) [label=below:{$-1$}] {};
\node[rectangle] (00) at (0,0) [label=below:{$-5$}] {};
\node[rectangle] (10) at (1,0) [label=below:{$-2$}] {};
\node[bullet] (20) at (2,0) [label=below:{$-3$}] {};

\node[bullet] (01) at (0,1) [label=left:{$-2$}] {};

\draw [-] (00)--(-10)--(-20);
\draw [-] (00)--(01);
\draw [-] (00)--(10)--(20);
\end{tikzpicture}
\begin{tikzpicture}[scale=0.5]
\node[empty] (-21) at (-2,1) [] {[5]};

\node[bullet] (-10) at (-1,0) [label=below:{$-3$}] {};
\node[rectangle] (00) at (0,0) [label=below:{$-5$}] {};
\node[bullet] (10) at (1,0) [label=below:{$-1$}] {};
\node[rectangle] (20) at (2,0) [label=below:{$-3$}] {};
\node[rectangle] (30) at (3,0) [label=below:{$-3$}] {};

\node[rectangle] (01) at (0,1) [label=left:{$-2$}] {};

\draw [-] (00)--(-10);
\draw [-] (00)--(01);
\draw [-] (00)--(10)--(20)--(30);
\end{tikzpicture}
\begin{tikzpicture}[scale=0.5]
\node[empty] (-21) at (-2,1) [] {[6]};

\node[rectangle] (-10) at (-1,0) [label=below:{$-3$}] {};
\node[rectangle] (00) at (0,0) [label=below:{$-5$}] {};
\node[bullet] (10) at (1,0) [label=below:{$-1$}] {};
\node[rectangle] (20) at (2,0) [label=below:{$-3$}] {};
\node[rectangle] (30) at (3,0) [label=below:{$-3$}] {};

\node[rectangle] (01) at (0,1) [label=left:{$-2$}] {};

\draw [-] (00)--(-10);
\draw [-] (00)--(01);
\draw [-] (00)--(10)--(20)--(30);
\end{tikzpicture}
\begin{tikzpicture}[scale=0.5]
\node[empty] (-21) at (-2,1) [] {[7]};

\node[rectangle] (-20) at (-2,0) [label=below:{$-4$}] {};
\node[bullet] (-10) at (-1,0) [label=below:{$-1$}] {};
\node[rectangle] (00) at (0,0) [label=below:{$-5$}] {};
\node[rectangle] (10) at (1,0) [label=below:{$-3$}] {};
\node[bullet] (20) at (2,0) [label=below:{$-1$}] {};
\node[rectangle] (30) at (3,0) [label=below:{$-4$}] {};

\node[rectangle] (01) at (0,1) [label=left:{$-2$}] {};

\draw [-] (00)--(-10)--(-20);
\draw [-] (00)--(01);
\draw [-] (00)--(10)--(20)--(30);
\end{tikzpicture}

\end{center}
\end{example}

Let $f \colon Y \to X$ be a $P$-resolution. We have an induced map $F \colon \Def(Y) \to \Def(X)$ of deformation spaces by Wahl~\cite{Wahl-1976}, which is called as \emph{blowing-down deformations}. On the other hand, there is an irreducible subspace $\DefQG(Y) \subset \Def(Y)$ that corresponds to the $\mathbb{Q}$-Gorenstein deformations of singularities in $Y$.

\begin{proposition}[{KSB~\cite[Theorem~3.9]{KSB-1988}}]
\label{proposition:KSB-P-resolution-vs-components}
Let $X$ be a quotient surface singularity. Then

\begin{enumerate}
\item If $f \colon Y \to X$ is a $P$-resolution, then $F(\DefQG(Y))$ is an irreducible component of $\Def(X)$.

\item If $f_1 \colon Y_1 \to X$ and $f_2 \colon Y_2 \to X$ are two $P$-resolutions of $X$ that are not isomorphic over $X$, and if $F_1$ and $F_2$ are the corresponding maps of deformation spaces, then $F_1(\DefQG(Y_1)) \neq F_2(\DefQG(Y_2))$.

\item Every component of $\Def(X)$ arises in this way.
\end{enumerate}
\end{proposition}

Since Milnor fibers are invariants of irreducible components of $\Def(X_0)$, there is a one-to-one correspondence between Milnor fibers and $P$-resolutions of $(X,0)$. So many invariants of Milnor fibers or the irreducible components of $\Def(X_0)$ can be computed using the corresponding $P$-resolutions.

\subsubsection{Dimensions of irreducible components}

For each $P$-resolutions in Theorem~\ref{theorem:list}, we compute the dimensions of the corresponding irreducible components of $\Def(X_0)$ in Section~\ref{section:list} using the following theorem:

\begin{theorem}[{KSB~\cite[Corollary 3.20]{KSB-1988}}]
Let $X_0$ be a quotient surface singularity, $f \colon Y_0 \to X_0$ be a $P$-resolution of $X_0$, $Q_i$ be its singularities on $Y_0$, and $B$ be the irreducible component of $\Def(X_0)$ corresponding to $Y_0$. Then $\dim{B} = \sum_{i=1}^r \dim{\Def(Y_0, Q_i)}+d$, where $d$ is the dimension of the space $D$ of locally trivial deformations of $Y_0$.
\end{theorem}
Here $\dim{\Def(Y_0, Q_i)}$ and $d$ can be computed as follows:

\begin{lemma}[{KSB~\cite[Lemma 3.21]{KSB-1988}}]
If $Q$ is a singularity of class $T$ of type $\frac{1}{dna^2}(1, dna-1)$, then $\dim{\Def(Q)}=d$.
\end{lemma}

\begin{lemma}[{Riemenschneider~\cite{Riemenschneider-1974}}; cf.~{KSB~\cite[Remark~3.23]{KSB-1988}}]
For a cyclic quotient surface singularity whose dual graph of the minimal resolution is given by
\begin{tikzpicture}[scale=0.5]
 \node[bullet] (20) at (2,0) [label=above:{$-b_1$}] {};

\node[empty] (250) at (2.5,0) [] {};
\node[empty] (30) at (3,0) [] {};

 \node[bullet] (350) at (3.5,0) [label=above:{$-b_r$}] {};

\draw [-] (20)--(250);
\draw [dotted] (20)--(350);
\draw [-] (30)--(350);
\end{tikzpicture}
the Artin component has dimension $\sum(b_i-1)$.
\end{lemma}

\begin{example}[Continued from Example~\ref{example:O_12(3-2)+7}]
Let $X_0$ be the $O_{12(3-2)+7}$-singularity and let $B_1, \dotsc, B_4$ be the irreducible component of $\Def(X_0)$ corresponding to the $P$-resolutions $O_{12(3-2)+7}[1], \dotsc, O_{12(3-2)+7}[4]$. Then
\begin{align*}
\dim{B_1} &= (2-1) \times 3 + (3-1) + (4-1) = 8, \\
\dim{B_2} &= (2-1) \times 3 + (3-1) + 1 = 6, \\
\dim{B_3} &= (2-1) \times 2 + 2 = 4, \\
\dim{B_4} &= (2-1) \times 2 + 2 = 4.
\end{align*}
\end{example}

\subsubsection{Milnor numbers}

Let $X_0$ be a quotient surface singularity and let $M$ be a Milnor fiber of a smoothing $\pi \colon X \to \Delta$ of $X_0$. The \emph{Milnor number} of a smoothing $\pi$ is the second betti number of $M$.

For each $P$-resolutions in Theorem~\ref{theorem:list}, we compute their Milnor numbers in Section~\ref{section:list}. We explain briefly how to compute Milnor numbers using the corresponding $P$-resolution.

Let $f \colon Y_0 \to X_0$ be the $P$-resolution corresponding to $\pi \colon X \to \Delta$. According to Proposition~\ref{proposition:KSB-P-resolution-vs-components}, there is a smoothing $\phi \colon Y \to \Delta$ induced from $Q$-Gorenstein smoothings of singularities of $Y_0$ which blows down to $\pi \colon X \to \Delta$. Therefore two general fiber $Y_t$ and $X_t$ ($t \neq 0$) should be diffeomorphic, which implies that the Milnor number of $\pi$ can be computed by that of $\phi$.

On the other hand, a Milnor fiber $Y_t$ of $\phi \colon Y \to \Delta$ is just a rationally blowdown $4$-manifold from $Y_0$, that is, one cuts out neighborhoods of singularities of $Y_0$ and pastes back the corresponding Milnor fibers to the singularities. Hence we have:

\begin{lemma}
Let $f \colon Y_0 \to X_0$ be the $P$-resolution corresponding to a smoothing $\pi \colon X \to \Delta$ of $X_0$. Suppose $Q_i$ ($i=1,\dotsc,n$) are singularities of class $T$ (not rational double points) on $Y_0$ of type $\frac{1}{d_i n_i^2}(1,d_in_ia_i-1)$. Let $\widetilde{Y}_0$ be the minimal resolution of all singularities of $Y_0$. Then the Milnor number $\mu_{\pi}$ of $\pi$ is given by the following formula:
\begin{equation*}
\mu_{\pi} =
\begin{aligned}[t]
&\text{the number of irreducible divisors of $\widetilde{Y}_0$}\\
&\text{that are not contained in the minimal resolutions of $Q_i$'s}\\
&+ \sum_{i=1}^{n} (d_i-1)
\end{aligned}
\end{equation*}
\end{lemma}

\begin{example}[Continued from Example~\ref{example:O_12(3-2)+7}]
Let $X_0$ be the $O_{12(3-2)+7}$-singularity and let $\mu_1, \dotsc, \mu_4$ be the Milnor numbers corresponding to the $P$-resolutions. Then
\begin{align*}
\mu_1 &= 5, \\
\mu_2 &= 4 + (1-1) = 4, \\
\mu_3 &= 2 + (2-1) = 3, \\
\mu_4 &= 2 + (2-1) = 3.
\end{align*}
\end{example}

\subsection{$M$-resolutions}

There is another one-to-one correspondence between the components of $\Def(X)$ and certain partial resolutions of $X$, the so-called \emph{$M$-resolutions}. We apply the semi-stable minimal model program discussed in the next section to the smoothing $\widehat{Z} \to \Delta$ of the natural compactification $\widehat{Z}_0$ of the $M$-resolution $Z_0 \to X_0$ in order to identify the compactified Milnor fiber $\widehat{Z}_t$ as a rational complex surface. For details on $M$-resolutions, see Behnke--Christophersen~\cite{Behnke-Christophersen-1994} and PPSU~\cite[\S 6.2]{PPSU-2015}.

\begin{definition}[{Behnke--Christophersen~\cite[p.882]{Behnke-Christophersen-1994}}]
\label{definition:M-resolution}
An \emph{$M$-resolution} of a quotient surface singularity $X_0$ is a partial resolution $g \colon Z_0 \to X_0$ such that
\begin{enumerate}
\item $Z_0$ has only Wahl singularities.

\item $K_{Z_0}$ is nef relative to $g$, i.e., $K_{Z_0} \cdot E \ge 0$ for all $g$-exceptional curves $E$.
\end{enumerate}
\end{definition}

\begin{theorem}[{Behnke--Christophersen~\cite[3.1.4, 3.3.2, 3.4]{Behnke-Christophersen-1994}}]
Let $(X_0,0)$ be a quotient surface singularity. Then

\begin{enumerate}
\item Each $P$-resolution $Y_0 \to X_0$ is dominated by a unique $M$-resolution $Z_0 \to X_0$, i.e., there is a surjection $h \colon Z_0 \to Y_0$, with the property that $K_{Z_0} = g^{\ast}{K_{Y_0}}$.

\item There is a surjective map $\DefQG(Z_0) \to \DefQG(Y_0)$ induced by blowing down deformations.

\item There is a one-to-one correspondence between the components of $\Def(X_0)$ and $M$-resolutions of $X_0$.
\end{enumerate}
\end{theorem}

Refer PPSU~\cite[\S6.2]{PPSU-2015} for constructing the $M$-resolution corresponding to a given $P$-resolution. Here we give an example.

\begin{example}[Continued from Example~\ref{example:O_12(3-2)+7}]
\label{example:M-resolution-O_12(3-2)+7}
Let $X_0$ be the $O_{12(3-2)+7}$-singularity. The $M$-resolutions of $X_0$ are as follows:

\begin{center}
\begin{tikzpicture}[scale=0.5]
\node[empty] (-21) at (-2,1) [] {[1]};

\node[rectangle] (-20) at (-2,0) [label=below:{$-2$}] {};
\node[rectangle] (-10) at (-1,0) [label=below:{$-2$}] {};
\node[bullet] (00) at (0,0) [label=below:{$-3$}] {};
\node[bullet] (10) at (1,0) [label=below:{$-4$}] {};

\node[rectangle] (01) at (0,1) [label=left:{$-2$}] {};

\draw [-] (00)--(-10)--(-20);
\draw [-] (00)--(01);
\draw [-] (00)--(10);
\end{tikzpicture}
\begin{tikzpicture}[scale=0.5]
\node[empty] (-21) at (-2,1) [] {[2]};

\node[rectangle] (-20) at (-2,0) [label=below:{$-2$}] {};
\node[rectangle] (-10) at (-1,0) [label=below:{$-2$}] {};
\node[bullet] (00) at (0,0) [label=below:{$-3$}] {};
\node[rectangle] (10) at (1,0) [label=below:{$-4$}] {};

\node[rectangle] (01) at (0,1) [label=left:{$-2$}] {};

\draw [-] (00)--(-10)--(-20);
\draw [-] (00)--(01);
\draw [-] (00)--(10);
\end{tikzpicture}
\begin{tikzpicture}[scale=0.5]
\node[empty] (-21) at (-2,1) [] {[3]};

\node[rectangle] (-20) at (-2,0) [label=below:{$-2$}] {};
\node[bullet] (-10) at (-1,0) [label=below:{$-2$}] {};
\node[rectangle] (00) at (0,0) [label=below:{$-5$}] {};
\node[bullet] (10) at (1,0) [label=below:{$-1$}] {};
\node[rectangle] (20) at (2,0) [label=below:{$-2$}] {};
\node[rectangle] (30) at (3,0) [label=below:{$-5$}] {};

\node[rectangle] (01) at (0,1) [label=left:{$-2$}] {};

\draw [-] (00)--(-10)--(-20);
\draw [-] (00)--(01);
\draw [-] (00)--(10)--(20)--(30);
\end{tikzpicture}
\begin{tikzpicture}[scale=0.5]
\node[empty] (-21) at (-2,1) [] {[4]};

\node[bullet] (-20) at (-2,0) [label=below:{$-2$}] {};
\node[rectangle] (-10) at (-1,0) [label=below:{$-2$}] {};
\node[rectangle] (00) at (0,0) [label=below:{$-5$}] {};
\node[bullet] (10) at (1,0) [label=below:{$-1$}] {};
\node[rectangle] (20) at (2,0) [label=below:{$-2$}] {};
\node[rectangle] (30) at (3,0) [label=below:{$-5$}] {};

\node[bullet] (01) at (0,1) [label=left:{$-2$}] {};

\draw [-] (00)--(-10)--(-20);
\draw [-] (00)--(01);
\draw [-] (00)--(10)--(20)--(30);
\end{tikzpicture}
\end{center}
\end{example}

\subsection{Natural compactifications of partial resolutions}

Let $X_0$ be a non-cyclic quotient surface singularity. Let $\overline{X}_0$ and $\widehat{X}_0$ be the singular natural compactification and the natural compactification, respectively. Then one can extend the deformations $X \to \Delta$ of $X_0$, $\overline{X} \to \Delta$ of $\overline{X}_0$, and $\widehat{X} \to \Delta$ of $\widehat{X}_0$ to deformations of certain partial resolutions of $X$, $\overline{X}$, and $\widehat{X}$, respectively. We briefly recall how to construct them. For details, refer PPSU~\cite[\S7]{PPSU-2015}.

Let $f \colon Z_0 \to X$ be the $M$-resolution (or the $P$-resolution) of $X$ corresponding to the smoothing $X \to \Delta$. We take a partial resolution $\overline{f} \colon \overline{Z}_0 \to \overline{X}_0$ of $\overline{X}_0$ corresponding to $f$, that is, $\overline{f}|_{Z_0}=f$. Let $\widehat{Y}_0$ be the minimal resolution of the cyclic quotient surface singularities on $\overline{E}_{\infty} \subset \overline{Z}_0$, which is called the \emph{singular natural compactification} of $Z_0$.

On can show that the smoothing $Z \to \Delta$ extends to a deformation $\overline{Z} \to \Delta$ of $\overline{Z}_0$ which is again a locally trivial deformation near $\overline{E}_{\infty}$. Let $\widehat{Z} \to \Delta$ be the simultaneous resolution of the cyclic quotient surface singularities along $\overline{E}_{\infty}$ in each fiber of $\overline{Z} \to \Delta$. Then the deformations $\mathcal{Y}$, $\overline{\mathcal{Y}}$, and $\widehat{\smash[b]{\mathcal{Y}}}$ blow down to the deformations $\mathcal{X}$, $\overline{\mathcal{X}}$, and $\widehat{\smash[b]{\mathcal{X}}}$, which is called the \emph{natural compactification} of $Z_0$ as before. Then it is clear that
\begin{equation*}
\widehat{X}_t \cong \widehat{Z}_t
\end{equation*}
for $t \neq 0$, which are called the \emph{compactified Milnor fiber} of the smoothing $X \to \Delta$.

\section{Semi-stable minimal model program}
\label{section:MMP}

Let $X_0$ be a non-cyclic quotient surface singularity and let $X \to \Delta$ be a smoothing. Let $Z_0 \to X_0$ be the corresponding $M$-resolution. We will apply the semi-stable minimal model program to the smoothing $\widehat{Z} \to \Delta$ of the natural compactification $\widehat{Z}_0$ in order to identify the compactified Milnor fiber $\widehat{X}_t (= \widehat{Z}_t)$ as a rational complex surface; Proposition~\ref{proposition:smoothable-by-flips}. Hence we can identify the Milnor fiber $X_t=\widehat{X}_t - \widehat{E}_{\infty}$ as a complement the compactifying divisor $\widehat{E}_{\infty}$ embedded in a rational surface.

Two operations in the semi-stable minimal model that we apply to the smoothing $\widehat{Z} \to \Delta$ are the so-called \emph{Iitaka-Kodaira divisorial contractions} and \emph{(usual) flips}. Here we explain only how the operations modify the smoothing $\widehat{Z} \to \Delta$. We refer Koll\'ar-Mori~\cite{Kollar-Mori-1992}, HTU~\cite{Hacking-Tevelev-Urzua-2013}, Urz\'ua~\cite{Urzua-2013}, PPSU~\cite{PPSU-2015} for theoretical details.

\subsection{Iitaka-Kodaira divisorial contractions}

If there is a $(-1)$-curve in $\widehat{Z}_0$ which does not passing through any singularities on $\widehat{Z}_0$, one may apply \emph{(Iitaka-Kodaira) divisorial contractions} to $\widehat{Z} \to \Delta$. There are corresponding $(-1)$-curves to the $(-1)$-curve in $\widehat{Z}_0$ on each fibers $\widehat{Z}_t$ ($t \in \Delta$) of the smoothing $\widehat{Z} \to \Delta$. Applying a divisorial contraction to the $(-1)$-curves, we have a morphism $\widehat{Z} \to \widehat{Z}'$ that is induced by blowing down the $(-1)$-curves on each fibers $\widehat{Z}_t \to \widehat{Z}_t'$ ($t \in \Delta$). So we have a new smoothing $\widehat{Z}' \to \Delta$ of the new central fiber $Z_0'$.

\subsection{Usual flips}

The typical situation that we encounter with the so-called \emph{(usual) flip} is that there is a (singular) curve $C \subset \widehat{Z}_0$ passing through only one Wahl singularity $P$ such that the dual graph of the minimal resolution $f \colon \widehat{W}_0 \to \widehat{Z}_0$ of $P \in \widehat{Z}_0$ is given as follows
\begin{equation}
\label{equation:usual-flip}
\begin{tikzpicture}
 \node[bullet] (10) at (1,0) [label=below:{$-b_1$},label=above:{$E_1$}] {};
 \node[bullet] (20) at (2,0) [label=below:{$-b_2$},label=above:{$E_2$}] {};

\node[empty] (250) at (2.5,0) [] {};
\node[empty] (30) at (3,0) [] {};

 \node[bullet] (350) at (3.5,0) [label=below:{$-b_{r-1}$},label=above:{$E_{r-1}$}] {};
 \node[bullet] (450) at (4.5,0) [label=below:{$-b_r$},label=above:{$E_r$}] {};
 \node[bullet] (550) at (5.5,0) [label=below:{$-1$},label=above:{$\overline{C}$}] {};

\draw [-] (10)--(20);
\draw [-] (20)--(250);
\draw [dotted] (20)--(350);
\draw [-] (30)--(350);
\draw [-] (350)--(450)--(550);
\end{tikzpicture}
\end{equation}
where $E_i$ ($i=1,\dotsc,r$) are the exceptional divisors of $f$ and $\overline{C}$ is the proper image of $C$.

We then \emph{flip} $(C \subset \widehat{Z})$ so that we have a birational map
\begin{equation*}
(C \subset \widehat{Z}) \dashrightarrow (C^+ \subset \widehat{Z}^+)
\end{equation*}
where the pair $(C^+ \subset \widehat{Z}^+)$ of a new curve $C^+ \subset \widehat{Z}_0^+$ (inside a new (singular) surface $\widehat{Z}_0^+$) and a new smoothing $\widehat{Z}^+ \to \Delta$ of $\widehat{Z}_0^+$ is obtained as follows: Suppose that $i$ is the largest index satisfying $b_i \ge 3$ and $b_j=2$ for all $i < j \le  r$. We blow down $(-1)$-curves in $\widehat{W}_0$ repeatedly starting from the $(-1)$-curve $\overline{C}$ until $E_{i+1}$ so that we get a blown down surface $\widehat{W}_0^+$; then, we contract the linear chain consisting of the images of $E_2, \dotsc, E_i$ in $\widehat{W}_0^+$ to obtain a new singular surface $\widehat{Z}_0^+$; then the image of $E_1$ in $\widehat{Z}_0^+$ is the curve $C^+$ with the Wahl singularity $P^+ \in \widehat{Z}_0^+$ of type $\frac{1}{m^2}(1, ma-1)$ where $\frac{m^2}{ma-1}=[b_2, \dotsc, b_i-1]$. And $\widehat{Z}^+ \to \Delta$ is a $\mathbb{Q}$-Gorenstein smoothing of singularities of $\widehat{Z}_0^+$ including $Q$. In case $b_i=2$ for all $i \ge 2$, $C^+$ is the image of the blown down $E_1$ with $C^+ \cdot C^+ = -b_1+1$.

\subsection{Identifying Milnor fibers}

We will apply the above semi-stable minimal model program to the smoothing $\widehat{Z} \to \Delta$ of the natural compactification $\widehat{Z}_0$ of the $M$-resolution $Z_0 \to X_0$ corresponding to the given smoothing $X \to \Delta$.

\begin{proposition}[{PPSU~\cite[Theorem~9.4]{PPSU-2015}}]
\label{proposition:smoothable-by-flips}
One can run a sequence of Iitaka-Kodaira divisorial contractions and usual flips to $\widehat{Z} \to \Delta$ until we obtain a deformation $\widehat{W} \to \Delta$ whose central fiber $W_0$ is smooth.
\end{proposition}

Since $\widehat{W} \to \Delta$ is a deformation with only smooth fibers, the central fiber $\widehat{W}_0$ is diffeomorphic to a general fiber $\widehat{W}_t$. Then one can get the data of positions of $(-1)$-curves in $W_t$ by comparing $W_0$ and $W_t$ because of the following proposition, which says how usual flips break down curves after running MMP.

\begin{proposition}[{PPSU~\cite[Proposition~8.9]{PPSU-2015}}]
\label{proposition:degenerationCurves}
Suppose that we are in the flipping case $(C \subset \widehat{Z})$.
Suppose that $\Gamma_0$ intersects transversally $C$ at one point, and that $F \colon C \subset \widehat{Z}_0 \subset \widehat{Z} \to 0 \in \Delta$ is a usual flip. Let $F \colon C^+ \subset \widehat{Z}_0^+ \subset \widehat{Z}^+ \to 0 \in \Delta$ be the flip, where $C^+=E_1$, and let $\Gamma^+ \subset \widehat{Z}^+$ be the proper transform of $\Gamma$. Then $\Gamma^+_0 = \Gamma_0 + C^+$.
\end{proposition}

Note that a usual flip changes only the central fiber. So, in order to track down how a general fiber is changed during the above MMP process, we need to take care of only divisorial contractions, which are just blow-downs of $(-1)$-curves on a general fiber. Hence:

\begin{proposition}[{PPSU~\cite[Corollary~9.5]{PPSU-2015}}]
A general fiber $\widehat{Z}_t$ ($t \neq 0$) of the smoothing $\widehat{Z} \to \Delta$ is obtained by a sequence of blowing-ups of a general fiber $\widehat{W}_t$ of the smoothing $\widehat{W} \to \Delta$ in Proposition~\ref{proposition:smoothable-by-flips}.
\end{proposition}

Finally one can get the data of intersections of $(-1)$-curves with the compactifying divisor $\widehat{E}_{\infty}$ in $\widehat{Z}_t$ by tracking the blow-downs $\widehat{Z}_t \to \widehat{W}_t$ given by flips and divisorial contractions. We will given some examples in detail in the next subsection.

\begin{definition}
For a given $P$-resolution $Z_0$, we call by the \emph{MMP $(-1)$-data} of $Z_0$ the data of intersections of $(-1)$-curves with the compactifying divisor $\widehat{E}_{\infty}$ in $\widehat{Z}_t$ obtained by applying the above flips and divisorial contractions.
\end{definition}

For each $P$-resolutions in Theorem~\ref{theorem:list}, we make a list of the MMP $(-1)$-data in Section~\ref{section:list}.

\subsection{Example}

We give some examples of running MMP to get the MMP $(-1)$-data from $M$-resolutions of quotient surface singularities.

\subsubsection{$T_{6(5-2)+1}$-singularity}

Let $X_0$ be the $T_{6(5-2)+1}$-singularity in Example~\ref{example:T_6(5-2)+1} and let $Y_0$ be the $P$-resolution $T_{6(5-2)+1}[4]$. Then its MMP $(-1)$-data is as follows, where the red curves are $(-1)$-curves.

\begin{center}
\includegraphics{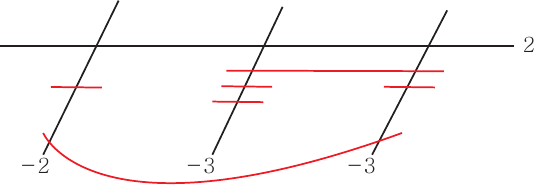}
\end{center}

We now briefly explain the procedures. Let $\widehat{Y}_0$ be the natural compactification and let $\widehat{Y} \to \Delta$ be the corresponding smoothing  of $\widehat{Y}_0$. The dual graphs of the central fiber $\widehat{Y}_0$ and a general fiber $\widehat{Y}_t$ of $\widehat{Y} \to \Delta$ is given as follows:

\begin{multicols}{2}
\begin{center}
\begin{tikzpicture}
    \node[empty] (-31) at (-3,1) [label=above:$\widehat{Y}_0$] {};

    \node[rectangle] (-30) at (-3,0) [label=left:$-5$]{};
    \node[bullet] (-20) at (-2,0) [label=below:$-2$] {};
    \node[bullet] (-21) at (-2,1) [label=below:$-2$] {};
    \node[rectangle] (-2-1) at (-2,-1) [label=below:$-2$] {};
    \node[bullet] (-1-1) at (-1,-1) [label=below:$-2$] {};
    \node[bullet] (0-1) at (0,-1) [label=below:$-1$, label=above:$E_3$] {};
    \node[bullet] (1-1) at (1,-1) [label=below:$-3$, label=above:$C$] {};
    \node[bullet] (-10) at (-1,0) [label=below:$-2$] {};
    \node[bullet] (00) at (0,0) [label=below:$-1$, label=above:$E_2$] {};
    \node[bullet] (10) at (1,0) [label=below:$-3$, label=above:$B$] {};
    \node[bullet] (01) at (0,1) [label=below:$-1$, label=above:$E_1$] {};
    \node[bullet] (11) at (1,1) [label=below:$-2$, label=above:$A$] {};
    \node[bullet] (20) at (2,0) [label=right:$2$] {};

    \draw[-] (-30)--(-21);
    \draw[-] (-21)--(01);
    \draw[-] (01)--(11);
    \draw[-] (11)--(20);
    \draw[-] (-30)--(-20);
    \draw[-] (-20)--(-10);
    \draw[-] (-10)--(00);
    \draw[-] (00)--(10);
    \draw[-] (10)--(20);
    \draw[-] (20)--(20);
    \draw[-] (-30)--(-2-1);
    \draw[-] (-2-1)--(-1-1);
    \draw[-] (-1-1)--(0-1);
    \draw[-] (0-1)--(1-1);
    \draw[-] (1-1)--(20);
\end{tikzpicture}
\end{center}

\columnbreak

\begin{center}
\begin{tikzpicture}
    \node[empty] (01) at (0,1) [label=above:$\widehat{Y}_t$] {};

    \node[bullet] (1-1) at (1,-1) [label=below:$-3$, label=above:$C$] {};
    \node[bullet] (10) at (1,0) [label=below:$-3$, label=above:$B$] {};
    \node[bullet] (11) at (1,1) [label=below:$-2$, label=above:$A$] {};
    \node[bullet] (20) at (2,0) [label=right:$2$] {};

    \draw[-] (11)--(20);
    \draw[-] (10)--(20);
    \draw[-] (1-1)--(20);
\end{tikzpicture}
\end{center}
\end{multicols}

Since three $(-1)$-curves $E_1$, $E_1$, $E_3$ do not pass through the singularity, they survives during the smoothing. Hence there are three $(-1)$-curves (denoted again by $E_i$'s) on a general fiber $\widehat{Y}_t$ with the same intersection data of $E_i$'s. Hence:

\begin{multicols}{2}
\begin{center}
\begin{tikzpicture}
    \node[empty] (-31) at (-3,1) [label=above:$\widehat{Y}_0$] {};

    \node[rectangle] (-30) at (-3,0) [label=left:$-5$]{};
    \node[bullet] (-20) at (-2,0) [label=below:$-2$] {};
    \node[bullet] (-21) at (-2,1) [label=below:$-2$] {};
    \node[rectangle] (-2-1) at (-2,-1) [label=below:$-2$] {};
    \node[bullet] (-1-1) at (-1,-1) [label=below:$-2$] {};
    \node[bullet] (0-1) at (0,-1) [label=below:$-1$, label=above:$E_3$] {};
    \node[bullet] (1-1) at (1,-1) [label=below:$-3$, label=above:$C$] {};
    \node[bullet] (-10) at (-1,0) [label=below:$-2$] {};
    \node[bullet] (00) at (0,0) [label=below:$-1$, label=above:$E_2$] {};
    \node[bullet] (10) at (1,0) [label=below:$-3$, label=above:$B$] {};
    \node[bullet] (01) at (0,1) [label=below:$-1$, label=above:$E_1$] {};
    \node[bullet] (11) at (1,1) [label=below:$-2$, label=above:$A$] {};
    \node[bullet] (20) at (2,0) [label=right:$2$] {};

    \draw[-] (-30)--(-21);
    \draw[-] (-21)--(01);
    \draw[-] (01)--(11);
    \draw[-] (11)--(20);
    \draw[-] (-30)--(-20);
    \draw[-] (-20)--(-10);
    \draw[-] (-10)--(00);
    \draw[-] (00)--(10);
    \draw[-] (10)--(20);
    \draw[-] (20)--(20);
    \draw[-] (-30)--(-2-1);
    \draw[-] (-2-1)--(-1-1);
    \draw[-] (-1-1)--(0-1);
    \draw[-] (0-1)--(1-1);
    \draw[-] (1-1)--(20);
\end{tikzpicture}
\end{center}

\columnbreak

\begin{center}
\begin{tikzpicture}
    \node[empty] (-11) at (-1,1) [label=above:$\widehat{Y}_t$] {};

    \node[bullet] (1-1) at (1,-1) [label=below:$-3$, label=above:$C$] {};
    \node[bullet] (10) at (1,0) [label=below:$-3$, label=above:$B$] {};
    \node[bullet] (11) at (1,1) [label=below:$-2$, label=above:$A$] {};
    \node[bullet] (20) at (2,0) [label=right:$2$] {};

    \node[bullet] (0-1) at (0,-1) [label=below:$-1$, label=above:$E_3$] {};
    \node[bullet] (00) at (0,0) [label=below:$-1$, label=above:$E_2$] {};
    \node[bullet] (01) at (0,1) [label=below:$-1$, label=above:$E_1$] {};

    \draw[-] (01)--(11)--(20);
    \draw[-] (00)--(10)--(20);
    \draw[-] (0-1)--(1-1)--(20);
\end{tikzpicture}
\end{center}
\end{multicols}

We now apply divisorial contractions to $E_1$, $E_2$, $E_3$. Then:

\begin{multicols}{2}

\begin{center}
\begin{tikzpicture}
    \node[empty] (-31) at (-3,1) [label=above:$\widehat{Y}_0$] {};

    \node[rectangle] (-30) at (-3,0) [label=left:$-5$]{};
    \node[bullet] (-20) at (-2,0) [label=below:$-2$] {};
    \node[bullet] (-21) at (-2,1) [label=below:$-1$] {};
    \node[rectangle] (-2-1) at (-2,-1) [label=below:$-2$] {};
    \node[bullet] (-1-1) at (-1,-1) [label=below:$-1$, label=above:$F$] {};
    \node[xmark] (0-1) at (0,-1) [] {};
    \node[bullet] (1-1) at (1,-1) [label=below:$-2$, label=above:$C$] {};
    \node[bullet] (-10) at (-1,0) [label=below:$-1$] {};
    \node[xmark] (00) at (0,0) [] {};
    \node[bullet] (10) at (1,0) [label=below:$-2$, label=above:$B$] {};
    \node[xmark] (01) at (0,1) [] {};
    \node[bullet] (11) at (1,1) [label=below:$-1$, label=above:$A$] {};
    \node[bullet] (20) at (2,0) [label=right:2] {};

    \draw[-] (-30)--(-21);
    \draw[-] (-21)--(01);
    \draw[-] (01)--(11);
    \draw[-] (11)--(20);
    \draw[-] (-30)--(-20);
    \draw[-] (-20)--(-10);
    \draw[-] (-10)--(00);
    \draw[-] (00)--(10);
    \draw[-] (10)--(20);
    \draw[-] (20)--(20);
    \draw[-] (-30)--(-2-1);
    \draw[-] (-2-1)--(-1-1);
    \draw[-] (-1-1)--(0-1);
    \draw[-] (0-1)--(1-1);
    \draw[-] (1-1)--(20);
\end{tikzpicture}
\end{center}

\columnbreak

\begin{center}
\begin{tikzpicture}
    \node[empty] (01) at (0,1) [label=above:$\widehat{Y}_t$] {};

    \node[bullet] (1-1) at (1,-1) [label=below:$-2$, label=above:$C$] {};
    \node[bullet] (10) at (1,0) [label=below:$-2$, label=above:$B$] {};
    \node[bullet] (11) at (1,1) [label=below:$-1$, label=above:$A$] {};
    \node[bullet] (20) at (2,0) [label=right:$2$] {};

    \draw[-] (11)--(20);
    \draw[-] (10)--(20);
    \draw[-] (1-1)--(20);
\end{tikzpicture}
\end{center}
\end{multicols}

We apply a usual flip to $F$. Then we have:

\begin{multicols}{2}

\begin{center}
\begin{tikzpicture}
    \node[empty] (-31) at (-3,1) [label=above:$\widehat{Y}_0$] {};

    \node[bullet] (-30) at (-3,0) [label=left:$-4$, label=above:$F^+$]{};
    \node[bullet] (-20) at (-2,0) [label=below:$-2$] {};
    \node[bullet] (-21) at (-2,1) [label=below:$-1$, label=above:$E_4$] {};
    \node[xmark] (-2-1) at (-2,-1) [] {};
    \node[xmark] (-1-1) at (-1,-1) [] {};
    \node[xmark] (0-1) at (0,-1) [] {};
    \node[bullet] (1-1) at (1,-1) [label=below:$0$, label=above:$C$] {};
    \node[bullet] (-10) at (-1,0) [label=below:$-1$, label=above:$E_5$] {};
    \node[xmark] (00) at (0,0) [] {};
    \node[bullet] (10) at (1,0) [label=below:$-2$, label=above:$B$] {};
    \node[xmark] (01) at (0,1) [] {};
    \node[bullet] (11) at (1,1) [label=below:$-1$, label=above:$A$] {};
    \node[bullet] (20) at (2,0) [label=right:2] {};

    \draw[-] (-30)--(-21);
    \draw[-] (-21)--(01);
    \draw[-] (01)--(11);
    \draw[-] (11)--(20);
    \draw[-] (-30)--(-20);
    \draw[-] (-20)--(-10);
    \draw[-] (-10)--(00);
    \draw[-] (00)--(10);
    \draw[-] (10)--(20);
    \draw[-] (20)--(20);
    \draw[-] (-30)--(-2-1);
    \draw[-] (-2-1)--(-1-1);
    \draw[-] (-1-1)--(0-1);
    \draw[-] (0-1)--(1-1);
    \draw[-] (1-1)--(20);
\end{tikzpicture}
\end{center}

\columnbreak

\begin{center}
\begin{tikzpicture}
    \node[empty] (01) at (0,1) [label=above:$\widehat{Y}_t$] {};

    \node[bullet] (1-1) at (1,-1) [label=below:$-2$, label=above:$C$] {};
    \node[bullet] (10) at (1,0) [label=below:$-2$, label=above:$B$] {};
    \node[bullet] (11) at (1,1) [label=below:$-1$, label=above:$A$] {};
    \node[bullet] (20) at (2,0) [label=right:$2$] {};

    \draw[-] (11)--(20);
    \draw[-] (10)--(20);
    \draw[-] (1-1)--(20);
\end{tikzpicture}
\end{center}

\end{multicols}

According to Proposition~\ref{proposition:degenerationCurves}, the curve $C$ in a general fiber $\widehat{Y}_t$ degenerates to two curves $C \cup F^+$ in the central fiber $\widehat{Y}_0$. Since there are two $(-1)$-curves $E_4$ and $E_5$ in $\widehat{Y}_0$, we have also two $(-1)$-curves, denoted again by $E_4$ and $E_5$, in $\widehat{Y}_t$:

\begin{multicols}{2}

\begin{center}
\begin{tikzpicture}
    \node[empty] (-31) at (-3,1) [label=above:$\widehat{Y}_0$] {};

    \node[bullet] (-30) at (-3,0) [label=left:$-4$, label=above:$F^+$]{};
    \node[bullet] (-20) at (-2,0) [label=below:$-2$] {};
    \node[bullet] (-21) at (-2,1) [label=below:$-1$, label=above:$E_4$] {};
    \node[xmark] (-2-1) at (-2,-1) [] {};
    \node[xmark] (-1-1) at (-1,-1) [] {};
    \node[xmark] (0-1) at (0,-1) [] {};
    \node[bullet] (1-1) at (1,-1) [label=below:$0$, label=above:$C$] {};
    \node[bullet] (-10) at (-1,0) [label=below:$-1$, label=above:$E_5$] {};
    \node[xmark] (00) at (0,0) [] {};
    \node[bullet] (10) at (1,0) [label=below:$-2$, label=above:$B$] {};
    \node[xmark] (01) at (0,1) [] {};
    \node[bullet] (11) at (1,1) [label=below:$-1$, label=above:$A$] {};
    \node[bullet] (20) at (2,0) [label=right:2] {};

    \draw[-] (-30)--(-21);
    \draw[-] (-21)--(01);
    \draw[-] (01)--(11);
    \draw[-] (11)--(20);
    \draw[-] (-30)--(-20);
    \draw[-] (-20)--(-10);
    \draw[-] (-10)--(00);
    \draw[-] (00)--(10);
    \draw[-] (10)--(20);
    \draw[-] (20)--(20);
    \draw[-] (-30)--(-2-1);
    \draw[-] (-2-1)--(-1-1);
    \draw[-] (-1-1)--(0-1);
    \draw[-] (0-1)--(1-1);
    \draw[-] (1-1)--(20);
\end{tikzpicture}
\end{center}

\columnbreak

\begin{center}
\begin{tikzpicture}
    \node[empty] (-11) at (-1,1) [label=above:$\widehat{Y}_t$] {};

    \node[bullet] (1-1) at (1,-1) [label=below:$-2$, label=above:$C$] {};
    \node[bullet] (10) at (1,0) [label=below:$-2$, label=above:$B$] {};
    \node[bullet] (11) at (1,1) [label=below:$-1$, label=above:$A$] {};
    \node[bullet] (20) at (2,0) [label=right:$2$] {};

    \node[bullet] (-10) at (-1,0) [label=below:$-1$, label=above:$E_4$] {};
    \node[bullet] (00) at (0,0) [label=below:$-1$, label=above:$E_5$] {};

    \draw[-] (11)--(20);
    \draw[-] (00)--(10)--(20);
    \draw[-] (1-1)--(20);
    \draw[-] (-10)--(11);
    \draw[-] (-10)--(1-1);
\end{tikzpicture}
\end{center}

\end{multicols}

Finally we apply again a divisorial contraction to $E_4$ and $E_5$. Then we have a new $(-1)$-curve $E_6$ in $\widehat{Y}_0$ which will be a new $(-1)$-curve (denoted again by $E_6$) in $\widehat{Y}_t$:

\begin{multicols}{2}

\begin{center}
\begin{tikzpicture}
    \node[bullet] (-30) at (-3,0) [label=left:$-3$, label=above:$F^+$]{};
    \node[bullet] (-20) at (-2,0) [label=below:$-1$, label=above:$E_6$] {};
    \node[xmark] (-21) at (-2,1) [] {};
    \node[xmark] (-2-1) at (-2,-1) [] {};
    \node[xmark] (-1-1) at (-1,-1) [] {};
    \node[xmark] (0-1) at (0,-1) [] {};
    \node[bullet] (1-1) at (1,-1) [label=below:$0$, label=above:$C$] {};
    \node[xmark] (-10) at (-1,0) [] {};
    \node[xmark] (00) at (0,0) [] {};
    \node[bullet] (10) at (1,0) [label=below:$-1$, label=above:$B$] {};
    \node[xmark] (01) at (0,1) [] {};
    \node[bullet] (11) at (1,1) [label=below:$0$, label=above:$A$] {};
    \node[bullet] (20) at (2,0) [label=right:2] {};

    \draw[-] (-30)--(-21);
    \draw[-] (-21)--(01);
    \draw[-] (01)--(11);
    \draw[-] (11)--(20);
    \draw[-] (-30)--(-20);
    \draw[-] (-20)--(-10);
    \draw[-] (-10)--(00);
    \draw[-] (00)--(10);
    \draw[-] (10)--(20);
    \draw[-] (20)--(20);
    \draw[-] (-30)--(-2-1);
    \draw[-] (-2-1)--(-1-1);
    \draw[-] (-1-1)--(0-1);
    \draw[-] (0-1)--(1-1);
    \draw[-] (1-1)--(20);
\end{tikzpicture}
\end{center}

\columnbreak

\begin{center}
\begin{tikzpicture}
    \node[empty] (-11) at (-1,1) [label=above:$\widehat{Y}_t$] {};

    \node[bullet] (1-1) at (1,-1) [label=below:$-1$, label=above:$C$] {};
    \node[bullet] (10) at (1,0) [label=below:$-1$, label=above:$B$] {};
    \node[bullet] (11) at (1,1) [label=below:$0$, label=above:$A$] {};
    \node[bullet] (20) at (2,0) [label=right:$2$] {};

    \node[bullet] (00) at (0,0) [label=below:$-1$, label=above:$E_6$] {};

    \draw[-] (11)--(20);
    \draw[-] (10)--(20);
    \draw[-] (1-1)--(20);
    \draw[-] (11) .. controls +(left:3cm) and +(left:3cm) .. (1-1);

    \draw[-] (00)--(10);
    \draw[-] (00)--(1-1);
\end{tikzpicture}
\end{center}

\end{multicols}

In a similar way, one can compute MMP $(-1)$-data for other $P$-resolutions of $X_0$. We summarize:

\begin{example}[Continued from Example~\ref{example:T_6(5-2)+1}]
\label{example:MMP-(-1)-data-T_6(5-2)+1}
The MMP $(-1)$-data for the singularity of $T_{6(5-2)+1}$-singularity are as follows:

\begin{center}
[1] \includegraphics[width=0.25\textwidth]{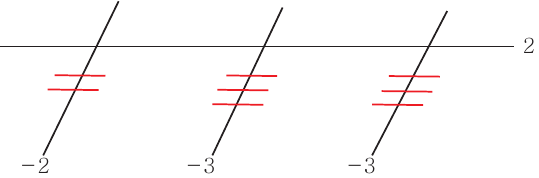} \qquad \qquad
[2] \includegraphics[width=0.25\textwidth]{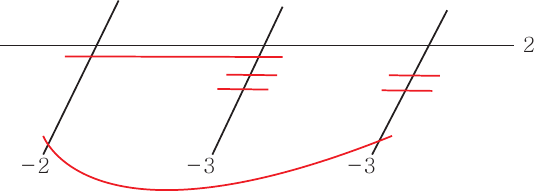}

[3] \includegraphics[width=0.25\textwidth]{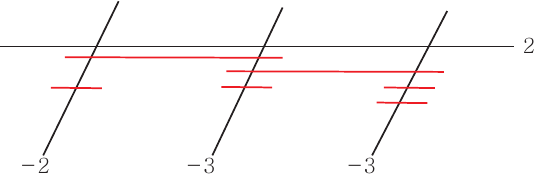} \qquad \qquad
[4] \includegraphics[width=0.25\textwidth]{T-1-5-4}
\end{center}
\end{example}

\subsubsection{$O_{12(3-2)+7}$-singularity}

Return to Example~\ref{example:O_12(3-2)+7}. Let $X_0$ be the $O_{12(3-2)+7}$-singularity and let $Y_0$ be the $P$-resolution $O_{12(3-2)+7}[3]$. In order to compute MMP $(-1)$-data for $Y_0$, we should use the corresponding $M$-resolution described in Example~\ref{example:M-resolution-O_12(3-2)+7}.
\begin{center}
\includegraphics[scale=0.75]{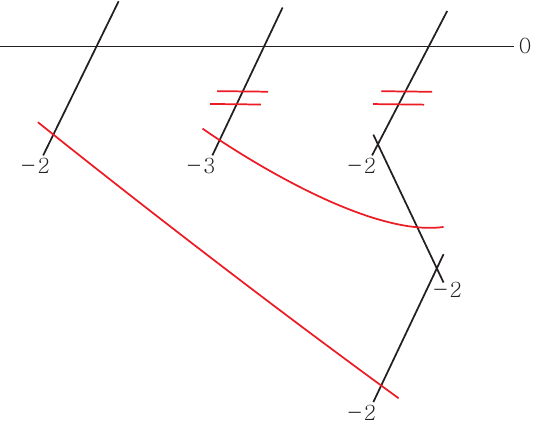}
\end{center}

The detailed procedures are as follows, where we draw the dual graphs of central fibers $\widehat{Y}_0$:

\begin{multicols}{2}
\begin{center}
\begin{tikzpicture}[scale=0.65]
    \node[empty] (-31) at (-3,1) [label=above:{Step 1}] {};

    \node[rectangle] (-30) at (-3,0) [label=left: \tiny $-3$]{};
    \node[bullet] (-20) at (-2,0) [label=below: \tiny $-2$] {};
    \node[rectangle] (-21) at (-2,1) [label=below: \tiny $-2$] {};
    \node[rectangle] (-2-1) at (-2,-1) [label=below: \tiny $-4$] {};
    \node[bullet] (0-1) at (0,-1) [label=below: \tiny $-1$] {};
    \node[bullet] (3-1) at (3,-1) [label=below: \tiny $-2$, label=above: \tiny $c_1$] {};
    \node[bullet] (-10) at (-1,0) [label=below: \tiny $-2$] {};
    \node[bullet] (00) at (0,0) [label=below: \tiny $-1$] {};
    \node[bullet] (40) at (4,0) [label=below: \tiny $0$] {};
    \node[bullet] (01) at (0,1) [label=below: \tiny $-1$] {};
    \node[bullet] (31) at (3,1) [label=below: \tiny $-2$, label=above: \tiny $a_1$] {};
    \node[bullet] (30) at (3,0) [label=below: \tiny $-3$,  label=above: \tiny $b_1$] {};
    \node[bullet] (1-1) at (1,-1) [label=below: \tiny $-2$, label=above: \tiny $c_3$] {};
    \node[bullet] (2-1) at (2,-1) [label=below: \tiny $-2$, label=above: \tiny $c_2$] {};

    \draw[-] (-30)--(-20);
    \draw[-] (-30)--(-21);
    \draw[-] (-30)--(-2-1);
    \draw[-] (-2-1)--(0-1);
    \draw[-] (0-1)--(1-1);
    \draw[-] (1-1)--(2-1);
    \draw[-] (2-1)--(3-1);
    \draw[-] (3-1)--(40);
    \draw[-] (-20)--(-10);
    \draw[-] (-10)--(00);
    \draw[-] (00)--(30);
    \draw[-] (30)--(40);
    \draw[-] (-21)--(01);
    \draw[-] (01)--(31);
    \draw[-] (31)--(40);
\end{tikzpicture}
\end{center}
\columnbreak
\begin{center}
\begin{tikzpicture}[scale=0.65]
    \node[empty] (-41) at (-4,1) [label=above:{Step 2}] {};

    \node[rectangle] (-40) at (-4,0) [label=left: \tiny $-5$]{};
    \node[rectangle] (-31) at (-3,1) [label=below: \tiny $-2$] {};
    \node[bullet] (-30) at (-3,0) [label=below: \tiny $-2$] {};
    \node[bullet] (-3-1) at (-3,-1) [label=below: \tiny $-1$] {};
    \node[bullet] (-20) at (-2,0) [label=below: \tiny $-2$] {};
    \node[rectangle] (-2-1) at (-2,-1) [label=below: \tiny $-2$] {};
    \node[rectangle] (-1-1) at (-1,-1) [label=below: \tiny $-5$] {};
    \node[bullet] (01) at (0,1) [label=below: \tiny $-1$, label=above: \tiny $F_1$] {};
    \node[bullet] (00) at (0,0) [label=below: \tiny $-1$] {};
    \node[bullet] (0-1) at (0,-1) [label=below: \tiny $-1$, label=above: \tiny $F_2$] {};
    \node[bullet] (1-1) at (1,-1) [label=below: \tiny $-2$, label=above: \tiny $c_3$] {};
    \node[bullet] (2-1) at (2,-1) [label=below: \tiny $-2$, label=above: \tiny $c_2$] {};
    \node[bullet] (31) at (3,1) [label=below: \tiny $-2$,  label=above: \tiny $a_1$] {};
    \node[bullet] (30) at (3,0) [label=below: \tiny $-3$,  label=above: \tiny $b_1$] {};
    \node[bullet] (3-1) at (3,-1) [label=below: \tiny $-2$, label=above: \tiny $c_1$] {};
    \node[bullet] (40) at (4,0) [label=below: \tiny $0$] {};

    \draw[-] (-40)--(-31);
    \draw[-] (-31)--(01);
    \draw[-] (01)--(31);
    \draw[-] (31)--(40);

    \draw[-] (-40)--(-30);
    \draw[-] (-30)--(-20);
    \draw[-] (-20)--(00);
    \draw[-] (00)--(30);
    \draw[-] (30)--(40);

    \draw[-] (-40)--(-3-1);
    \draw[-] (-3-1)--(-2-1);
    \draw[-] (-2-1)--(-1-1);
    \draw[-] (-1-1)--(0-1);
    \draw[-] (0-1)--(1-1);
    \draw[-] (1-1)--(2-1);
    \draw[-] (2-1)--(3-1);
    \draw[-] (3-1)--(40);
\end{tikzpicture}
\end{center}
\end{multicols}

\begin{multicols}{2}

\begin{center}
\begin{tikzpicture}[scale=0.65]
    \node[empty] (-41) at (-4,1) [label=above:{Step 3}] {};

    \node[bullet] (-40) at (-4,0) [label=left: \tiny $-4$, label=above: \tiny $F_1^+$]{};
    \node[xmark] (-31) at (-3,1) [] {};
    \node[bullet] (-30) at (-3,0) [label=below: \tiny $-2$] {};
    \node[bullet] (-3-1) at (-3,-1) [label=below: \tiny $-1$] {};
    \node[bullet] (-20) at (-2,0) [label=below: \tiny $-2$] {};
    \node[bullet] (-2-1) at (-2,-1) [label=below: \tiny $-2$, label=above: \tiny $F_2^+$] {};
    \node[rectangle] (-1-1) at (-1,-1) [label=below: \tiny $-4$] {};
    \node[xmark] (01) at (0,1) [] {};
    \node[bullet] (00) at (0,0) [label=below: \tiny $-1$] {};
    \node[xmark] (0-1) at (0,-1) [] {};
    \node[bullet] (1-1) at (1,-1) [label=below: \tiny $-1$, label=above: \tiny $F_3$] {};
    \node[bullet] (2-1) at (2,-1) [label=below: \tiny $-2$, label=above: \tiny $c_2$] {};
    \node[bullet] (31) at (3,1) [label=below: \tiny $0$, label=above: \tiny $a_1$] {};
    \node[bullet] (30) at (3,0) [label=below: \tiny $-3$, label=above: \tiny $b_1$] {};
    \node[bullet] (3-1) at (3,-1) [label=below: \tiny $-2$, label=above: \tiny $c_1$] {};
    \node[bullet] (40) at (4,0) [label=below: \tiny $0$] {};

    \draw[-] (-40)--(-31);
    \draw[-] (-31)--(01);
    \draw[-] (01)--(31);
    \draw[-] (31)--(40);

    \draw[-] (-40)--(-30);
    \draw[-] (-30)--(-20);
    \draw[-] (-20)--(00);
    \draw[-] (00)--(30);
    \draw[-] (30)--(40);

    \draw[-] (-40)--(-3-1);
    \draw[-] (-3-1)--(-2-1);
    \draw[-] (-2-1)--(-1-1);
    \draw[-] (-1-1)--(0-1);
    \draw[-] (0-1)--(1-1);
    \draw[-] (1-1)--(2-1);
    \draw[-] (2-1)--(3-1);
    \draw[-] (3-1)--(40);
\end{tikzpicture}
\end{center}

\columnbreak

\begin{center}
\begin{tikzpicture}[scale=0.65]
    \node[empty] (-41) at (-4,1) [label=above:{Step 4}] {};

    \node[bullet] (-40) at (-4,0) [label=left:\tiny $-4$, label=above: \tiny $F_1^+$]{};
    \node[xmark] (-31) at (-3,1) [] {};
    \node[bullet] (-30) at (-3,0) [label=below: \tiny $-2$] {};
    \node[bullet] (-3-1) at (-3,-1) [label=below: \tiny $-1$, label=above: \tiny $E_2$] {};
    \node[bullet] (-20) at (-2,0) [label=below: \tiny $-2$] {};
    \node[bullet] (-2-1) at (-2,-1) [label=below: \tiny $-2$, label=above: \tiny $F_2^+$] {};
    \node[bullet] (-1-1) at (-1,-1) [label=below: \tiny $-3$, label=above: \tiny $F_3^+$] {};
    \node[xmark] (01) at (0,1) [] {};
    \node[bullet] (00) at (0,0) [label=below: \tiny $-1$, label=above: \tiny $E_1$] {};
    \node[xmark] (0-1) at (0,-1) [] {};
    \node[xmark] (1-1) at (1,-1) [] {};
    \node[bullet] (2-1) at (2,-1) [label=below: \tiny $-1$, label=above left: \tiny $E_3$, label=above right: \tiny $c_2$] {};
    \node[bullet] (31) at (3,1) [label=below: \tiny $0$, label=above: \tiny $a_1$] {};
    \node[bullet] (30) at (3,0) [label=below: \tiny $-3$, label=above: \tiny $b_1$] {};
    \node[bullet] (3-1) at (3,-1) [label=below: \tiny $-2$, label=above: \tiny $c_1$] {};
    \node[bullet] (40) at (4,0) [label=below: \tiny $1$] {};

    \draw[-] (-40)--(-31);
    \draw[-] (-31)--(01);
    \draw[-] (01)--(31);
    \draw[-] (31)--(40);

    \draw[-] (-40)--(-30);
    \draw[-] (-30)--(-20);
    \draw[-] (-20)--(00);
    \draw[-] (00)--(30);
    \draw[-] (30)--(40);

    \draw[-] (-40)--(-3-1);
    \draw[-] (-3-1)--(-2-1);
    \draw[-] (-2-1)--(-1-1);
    \draw[-] (-1-1)--(0-1);
    \draw[-] (0-1)--(1-1);
    \draw[-] (1-1)--(2-1);
    \draw[-] (2-1)--(3-1);
    \draw[-] (3-1)--(40);
\end{tikzpicture}
\end{center}
\end{multicols}

\begin{multicols}{2}
\begin{center}
\begin{tikzpicture}[scale=0.65]
    \node[empty] (-41) at (-4,1) [label=above:{Step 5}] {};

    \node[bullet] (-40) at (-4,0) [label=left: \tiny $-2$, label=above: \tiny $F_1^+$]{};
    \node[xmark] (-31) at (-3,1) [] {};
    \node[bullet] (-30) at (-3,0) [label=below: \tiny $-2$] {};
    \node[xmark] (-3-1) at (-3,-1) [] {};
    \node[bullet] (-20) at (-2,0) [label=below: \tiny $-1$, label=above: \tiny $E_4$] {};
    \node[bullet] (-2-1) at (-2,-1) [label=below: \tiny $-1$,label=above: \tiny $F_2^+$] {};
    \node[bullet] (-1-1) at (-1,-1) [label=below: \tiny $-2$, label=above: \tiny $F_3^+$] {};
    \node[xmark] (01) at (0,1) [] {};
    \node[xmark] (00) at (0,0) [] {};
    \node[xmark] (0-1) at (0,-1) [] {};
    \node[xmark] (1-1) at (1,-1) [] {};
    \node[xmark] (2-1) at (2,-1) [] {};
    \node[xmark] (31) at (3,1) [] {};
    \node[bullet] (30) at (3,0) [label=below: \tiny $-2$, label=above: \tiny $b_1$] {};
    \node[bullet] (3-1) at (3,-1) [label=below: \tiny $-1$, label=above: \tiny $c_1$] {};
    \node[bullet] (40) at (4,0) [label=below: \tiny $1$] {};

    \draw[-] (-40)--(-31);
    \draw[-] (-31)--(01);
    \draw[-] (01)--(31);
    \draw[-] (31)--(40);

    \draw[-] (-40)--(-30);
    \draw[-] (-30)--(-20);
    \draw[-] (-20)--(00);
    \draw[-] (00)--(30);
    \draw[-] (30)--(40);

    \draw[-] (-40)--(-3-1);
    \draw[-] (-3-1)--(-2-1);
    \draw[-] (-2-1)--(-1-1);
    \draw[-] (-1-1)--(0-1);
    \draw[-] (0-1)--(1-1);
    \draw[-] (1-1)--(2-1);
    \draw[-] (2-1)--(3-1);
    \draw[-] (3-1)--(40);
\end{tikzpicture}
\end{center}

\columnbreak

\begin{center}
\begin{tikzpicture}[scale=0.65]
    \node[empty] (-41) at (-4,1) [label=above:{Step 6}] {};

    \node[bullet] (-40) at (-4,0) [label=left: \tiny $-2$, label=above: \tiny $F_1^+$]{};
    \node[xmark] (-31) at (-3,1) [] {};
    \node[bullet] (-30) at (-3,0) [label=below: \tiny $-1$, label=above: \tiny $E_5$] {};
    \node[xmark] (-3-1) at (-3,-1) [] {};
    \node[xmark] (-20) at (-2,0) [] {};
    \node[bullet] (-2-1) at (-2,-1) [label=below: \tiny $-1$,label=above: \tiny $F_2^+$] {};
    \node[bullet] (-1-1) at (-1,-1) [label=below: \tiny $-2$, label=above: \tiny $F_3^+$] {};
    \node[xmark] (01) at (0,1) [] {};
    \node[xmark] (00) at (0,0) [] {};
    \node[xmark] (0-1) at (0,-1) [] {};
    \node[xmark] (1-1) at (1,-1) [] {};
    \node[xmark] (2-1) at (2,-1) [] {};
    \node[xmark] (31) at (3,1) [] {};
    \node[bullet] (30) at (3,0) [label=below: \tiny $-1$, label=above: \tiny $b_1$] {};
    \node[bullet] (3-1) at (3,-1) [label=below: \tiny $-1$, label=above: \tiny $c_1$] {};
    \node[bullet] (40) at (4,0) [label=below: \tiny $1$] {};

    \draw[-] (-40)--(-31);
    \draw[-] (-31)--(01);
    \draw[-] (01)--(31);
    \draw[-] (31)--(40);

    \draw[-] (-40)--(-30);
    \draw[-] (-30)--(-20);
    \draw[-] (-20)--(00);
    \draw[-] (00)--(30);
    \draw[-] (30)--(40);

    \draw[-] (-40)--(-3-1);
    \draw[-] (-3-1)--(-2-1);
    \draw[-] (-2-1)--(-1-1);
    \draw[-] (-1-1)--(0-1);
    \draw[-] (0-1)--(1-1);
    \draw[-] (1-1)--(2-1);
    \draw[-] (2-1)--(3-1);
    \draw[-] (3-1)--(40);
\end{tikzpicture}
\end{center}
\end{multicols}

\subsubsection{$I_{30(4-2)+17}$-singularity}

Let $X_0$ be the $I_{30(4-2)+17}$-singularity in Example~\ref{example:I_30(4-2)+17}.

\begin{example}[Continued from Example~\ref{example:I_30(4-2)+17}]
\label{example:MMP-(-1)-data-I_30(4-2)+17}

The MMP $(-1)$-data for $X_0$ are as follows:

\begin{center}
[1] \includegraphics[width=0.3\textwidth]{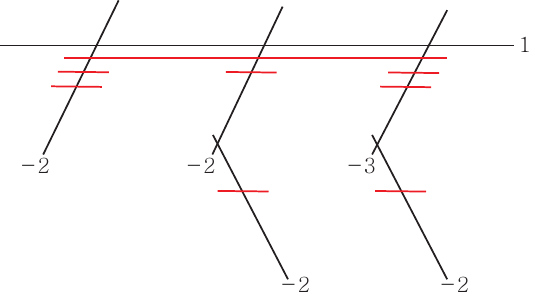} \quad
[2] \includegraphics[width=0.3\textwidth]{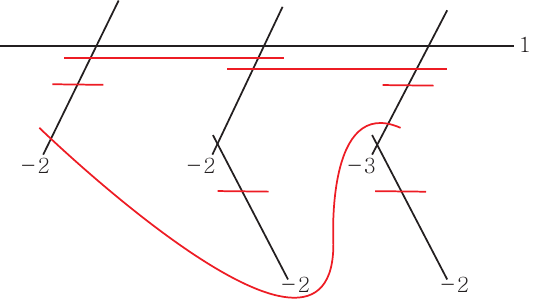}

[3] \includegraphics[width=0.3\textwidth]{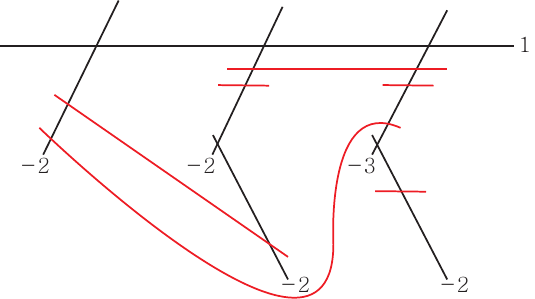} \quad
[4] \includegraphics[width=0.3\textwidth]{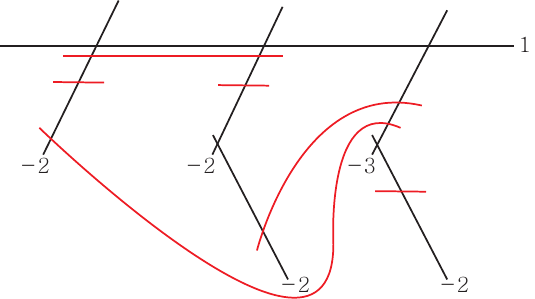}

[5] \includegraphics[width=0.3\textwidth]{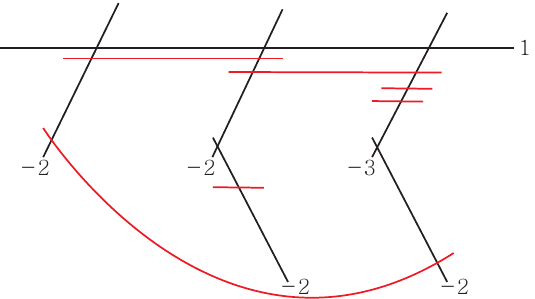} \quad
[6] \includegraphics[width=0.3\textwidth]{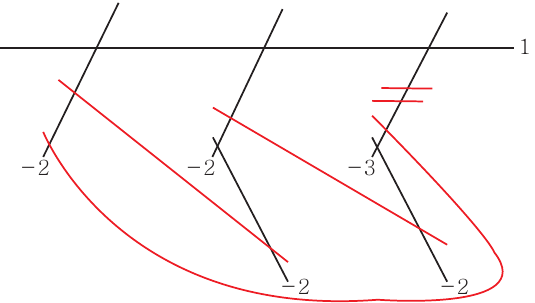}

[7] \includegraphics[width=0.3\textwidth]{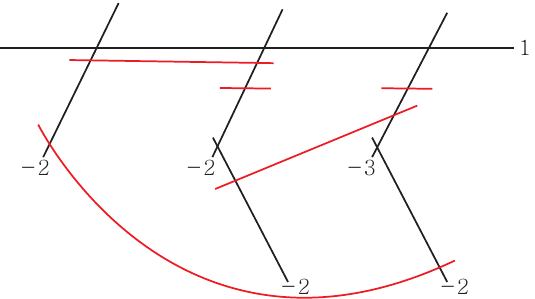}
\end{center}
\end{example}

\section{Milnor fibers and Symplectic Fillings}
\label{section:Milnor-vs-filling}

In this section we explain how to match a given $P$-resolution of a tetrahedral/octahedral/icosahedral singularity to a minimal symplectic fillings in the list of Bhupal-Ono~\cite{Bhupal-Ono-2012}.

In the previous section, Milnor fibers (hence, $P$-resolutions) of non-cyclic quotient surface singularities are identified as complements of compactifying divisors embedded in rational surfaces. In a similar way, Bhupal-Ono~\cite{Bhupal-Ono-2012} classifies minimal symplectic fillings of the links of non-cyclic quotient surface singularities up to symplectic deformation equivalence:

\begin{theorem}[{Bhupal-Ono~\cite[Theorem~1.1]{Bhupal-Ono-2012}}]
\label{theorem:rational}
A symplectic filling of the link of a quotient surface singularity is symplectic deformation equivalent to the complement of a certain divisor in an iterated blow-up of  $\mathbb{CP}^2$ or $\mathbb{CP}^1 \times \mathbb{CP}^1$.
\end{theorem}

The strategy of Bhupal-Ono~\cite{Bhupal-Ono-2012} is similar to the method in the previous section. At first one compactifies a minimal symplectic filling $W$ of a quotient surface singularity $X_0$ as a symplectic $4$-manifold $\widehat{W}$ by gluing a regular neighborhood $\nu(\widehat{E}_{\infty})$ of the compactifying divisor along the boundary, that is,
\begin{equation*}
\widehat{W} = W \cup_{\partial W} \nu(\widehat{E}_{\infty}).
\end{equation*}
Bhupal-Ono~\cite{Bhupal-Ono-2012} then shows that $\widehat{W}$ is a rational $4$-manifold. Therefore Bhupal-Ono~\cite{Bhupal-Ono-2012} shows that any minimal symplectic fillings of non-cyclic quotient surface singularities are complements of the compactifying divisors in rational $4$-manifolds as Milnor fibers are.

\subsection{Classification of minimal symplectic fillings}
\label{subsection:classification-symplectic-filling}

However classifications and presentations in Bhupal-Ono~\cite{Bhupal-Ono-2012} of all possible embeddings of compactifying divisors in rational $4$-manifolds are different from that of Milnor fibers in Section~\ref{section:MMP}. In order to deal with the problem of symplectic deformation equivalence, Bhupal-Ono~\cite{Bhupal-Ono-2012} transforms the embedded compactifying divisors in $\widehat{W}$ into certain configurations containing cuspidal curves by applying sequences of blow-ups and blow-downs as in Figures~\ref{figure:sequence-TOI-32}, \ref{figure:sequence-TOI-31}.

We briefly explain the sequences of blow-ups and blow-downs in Figures~\ref{figure:sequence-TOI-32}, \ref{figure:sequence-TOI-31}.  For details, refer Bhupal--Ono~\cite{Bhupal-Ono-2012}. At first we blow up successively (if necessary) the intersection point of the central curve of $\widehat{E}_{\infty}$ and the third branch until the central curve has became to ($-1$)-curve. Then we blow down (or blow up) as described in Figure~\ref{figure:sequence-TOI-32} and \ref{figure:sequence-TOI-31} so that we get a rational 4-manifold $Z_2$ with a cuspidal curve $C$ with $C \cdot C > 0$ and a linear chain of 2-spheres $C_1, \dotsc, C_k$ (plus some extra 2-spheres intersecting $C$ at the cusp). Let $\widetilde{E}_{\infty} \subset Z_2$ be the proper transform of $\widehat{E}_{\infty} \subset Z$. Since the blow-ups and blow-downs occur only on $E_{\infty}$ and its proper transforms, we have
\begin{equation*}
W \cong Z - \nu(\widehat{E}_{\infty}) \cong Z_2 - \nu(\widetilde{E}_{\infty}).
\end{equation*}

Since $W$ is minimal, every $(-1)$-curve in $Z_2$ should intersect $\widetilde{E}_{\infty}$. Let $Z_1$ be the rational 4-manifold obtained by contracting $(-1)$-curves not intersecting $C$. Then Bhupal-Ono~\cite{Bhupal-Ono-2012} shows that $Z_1$ can be obtained by a sequence of blowing-ups at $p$ (including infinitely near points over $p$) in a cuspidal curve $C$ in $\mathbb{CP}^2$ or $\mathbb{CP}^1 \times \mathbb{CP}^1$ and (if necessary) some points on $C$ as shown in  Figure~\ref{figure:sequence-TOI-32} and \ref{figure:sequence-TOI-31}.

\begin{figure}[tp]
\centering
\includegraphics{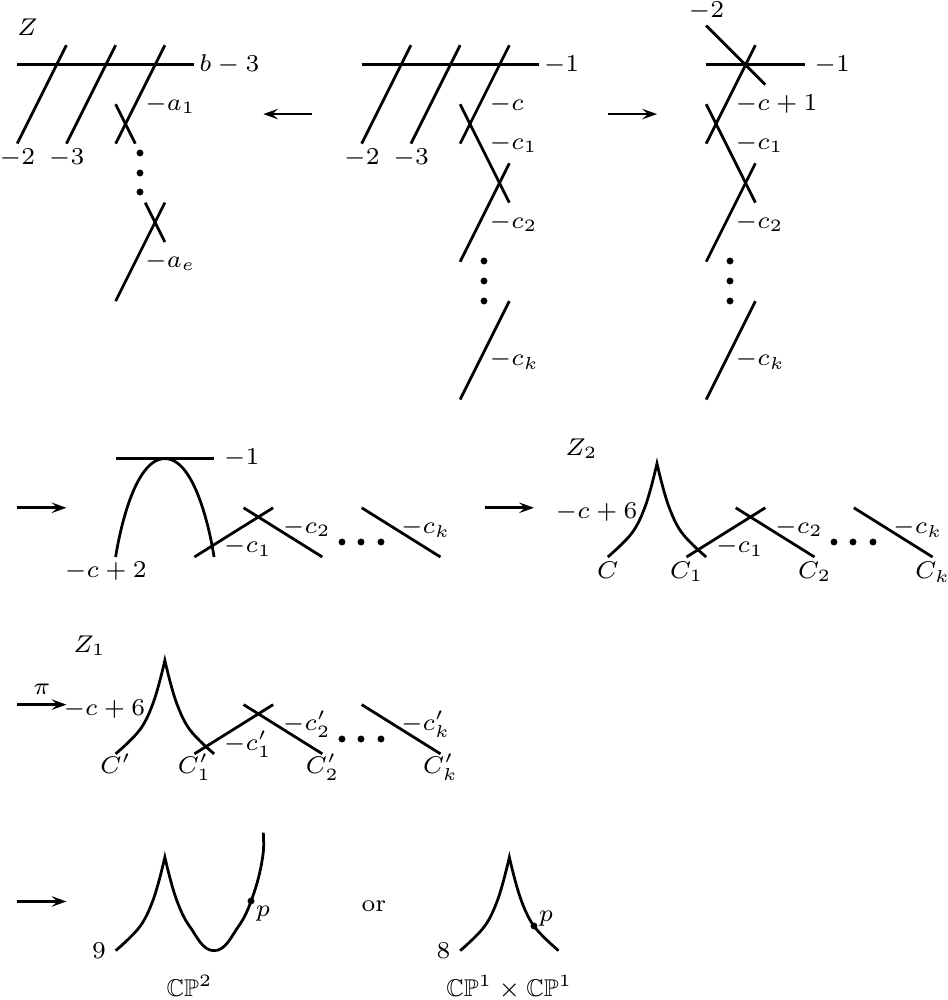}
\caption{For tetrahedral, octahedral, or icosahedral singularities of type $(3,2)$; Bhupal--Ono~\cite[Figure~3]{Bhupal-Ono-2012}, PPSU~\cite[Figure~6]{PPSU-2015}}
\label{figure:sequence-TOI-32}
\end{figure}

\begin{figure}[tp]
\centering
\includegraphics{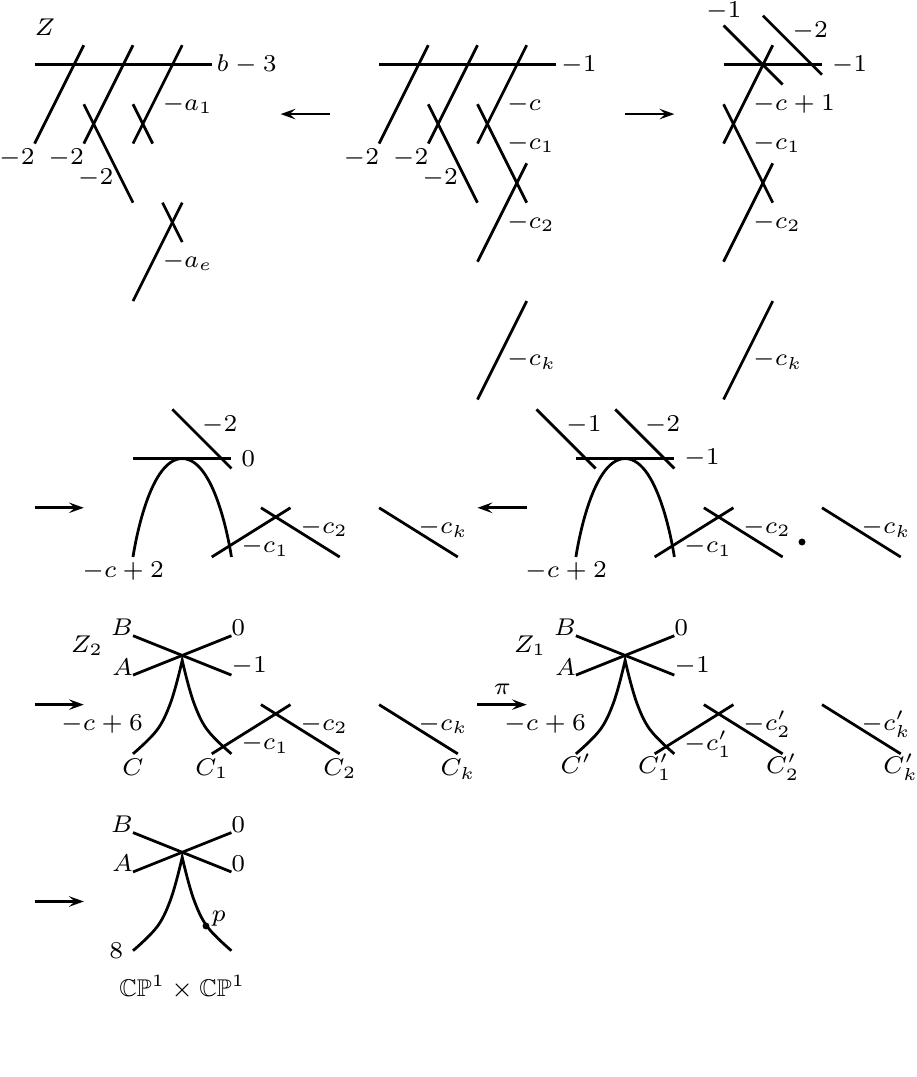}
\caption{For tetrahedral, octahedral, or icosahedral singularities of type $(3,1)$; Bhupal--Ono~\cite[Figure~5, 6, 10]{Bhupal-Ono-2012}, PPSU~\cite[Figure~7]{PPSU-2015}}
\label{figure:sequence-TOI-31}
\end{figure}

Furthermore Bhupal-Ono~\cite[Section~5]{Bhupal-Ono-2012} classifies all possible ways to obtain the rational $4$-manifold $Z_2$ from $\mathbb{CP}^2$ or $\mathbb{CP}^1 \times \mathbb{CP}^1$ by presenting all possible configurations of $(-1)$-curves intersecting only one of $C_1, \dotsc, C_k$.

\begin{definition}
We call the configurations of $(-1)$-curves in the list of Bhupal-Ono~\cite[Section~5]{Bhupal-Ono-2012} as the \emph{BO $(-1)$-data} of the corresponding minimal symplectic fillings.
\end{definition}

We give two examples of BO $(-1)$-data which will be used in the next subsections.

\begin{example}[cf.~Example~\ref{example:T_6(5-2)+1}]
\label{example:BO-(-1)-data-T_6(5-2)+1}
Let $X_0$ be the tetrahedral singularity $T_{6(5-2)+1}$, which is of type $(3,1)$. Its BO $(-1)$-data are given as follows:

\begin{enumerate}
\item[\#5.] $(T_{6(5-2)+1}; 5,-2,-2,-4; 3 \times C_3), \mathbb{CP}^2$,
\item[\#6.] $(T_{6(5-2)+1}; 5,-2,-2,-4; 1 \times C_1, 2 \times C_3), \mathbb{CP}^2$
\item[\#7.] $(T_{6(5-2)+1}; 5,-2,-2,-4; 1 \times C_1, 2 \times C_3), \mathbb{CP}^1 \times \mathbb{CP}^1$,
\item[\#8.] $(T_{6(5-2)+1}; 5,-2,-2,-4; 1 \times C_2, 1 \times C_3),\mathbb{CP}^2$,
\end{enumerate}
\end{example}

\begin{remark}[Notation for type $(3,2)$]
Here we use slightly different notation $(TOI_m; C \cdot C, -c_1, \dotsc ,-c_k; a_1 \times C_{i_1}, \dotsc , a_l \times X_{i_l})$ from that of Bhupal-Ono~\cite{Bhupal-Ono-2012} for clarity in order to denote the symplectic filling of the link of $T_m$, $O_m$, or $I_m$ of type $(3,2)$ given as the complement of a regular neighborhood of the compactifying divisor $K = C \cup  C_1 \cup \dotsb \cup C_k$ in $Z_2$ in Figure~\ref{figure:sequence-TOI-32}. The number $\# nn$ is the number of the model in Bhupal--Ono~\cite{Bhupal-Ono-2012}, $a_i \times C_i$ means that there are $a_i$ distinct $(-1)$-curves intersecting only $C_i$ in $Z_2$.
\end{remark}

\begin{example}[cf.~Example~\ref{example:I_30(4-2)+17}]
Let $X_0$ be the icosahedral singularity $I_{30(4-2)+17}$, which is of type $(3,1)$. The BO $(-1)$-data are divided into two cases:

Case I

\begin{enumerate}
\item[\#179.] $(I_{30(4-2)+17}; 5,-2,-4,-2; 2 \times C_2, 1 \times C_3)$
\item[\#180.] $(I_{30(4-2)+17}; 5,-2,-4,-2; 1 \times C_1, 2 \times C_2)$
\item[\#181.] $(I_{30(4-2)+17}; 5,-2,-4,-2; 1 \times C_1, 1 \times C_2, 1 \times C_3)$
\end{enumerate}

Case II

\begin{enumerate}
\item[\#238.] $(I_{30(4-2)+17}; 5,-2,-4,-2; 1, 2; 1 \times C_2, 1 \times C_3)$
\item[\#239.] $(I_{30(4-2)+17}; 5,-2,-4,-2; 1, 3; 2 \times C_2)$
\item[\#240.] $(I_{30(4-2)+17}; 5,-2,-4,-2; 2, 2; 1 \times C_1, 1 \times C_3)$
\item[\#241.] $(I_{30(4-2)+17}; 5,-2,-4,-2; 2, 3; 1 \times C_1, 1 \times C_2)$
\end{enumerate}
\end{example}

\begin{remark}[Notation for type $(3,1)$]
\label{remark:CaseI-vs-CaseII}
Let us briefly recall how Bhupal-Ono~\cite{Bhupal-Ono-2012} divides the BO $(-1)$-data for TOI-singularities of type $(3,1)$ into Case I and Case II. According to Bhupal-Ono~\cite[Lemma~4.6]{Bhupal-Ono-2012} there exists a unique $(-1)$-curve $E$ in $Z_2 \setminus (C \cup A)$ (cf.~Figure~\ref{figure:sequence-TOI-31}) such that $E \cdot B = 1$; furthermore $E$ can intersect at most one of the $C_i$ and $E \cdot C_i = 1$ if $E$ intersects $C_i$. So there are two cases:
\begin{itemize}
\item Case I: $E \cdot C_i=0$ for all $i$,

\item Case II: $E \cdot C_i = 1$ for some $i$.
\end{itemize}
In Case II, Bhupal-Ono~\cite[Lemma~4.9]{Bhupal-Ono-2012} also shows that there exists a $(-1)$-curve $F$ intersecting $C$ and some $C_j$ in $Z_2$. So we use the similar notation $(TOI_m; C \cdot C, -c_1, \dotsc ,-c_k; i, j; a_1 \times i_1, \dotsc , a_l \times i_l)$ to denote the Case II, where the numbers $i$ and $j$ denote the existence of $(-1)$-curves intersecting $B$ and $C_i$ and $C$ and $C_j$, respectively.
\end{remark}

\subsection{The correspondence between Milnor fibers and minimal symplectic fillings}
\label{subsection:Milnor-vs-filling}

In the previous section, we show that every Milnor fiber $M$ of a quotient surface singularity $X_0$ can be realized as complement of the compactifying divisor $\widehat{E}_{\infty}$ embedded in a rational surface $V$. And the embedding of $\widehat{E}_{\infty}$ in $V$ is completely determined by the $(-1)$-curves intersecting $\widehat{E}_{\infty}$. Therefore, after applying the sequence of blowing-ups and blowing-downs described in Figures~\ref{figure:sequence-TOI-32}, \ref{figure:sequence-TOI-31}, one may identify a given Milnor fiber to a minimal symplectic filling in the list of Bhupal-Ono~\cite[\S5]{Bhupal-Ono-2012}.

\begin{example}[Continued from Example~\ref{example:MMP-(-1)-data-T_6(5-2)+1}]

Let $X_0$ be the $T_{6(5-2)+1}$-singularity and let $Y_0$ be the $P$-resolution $T_{6(5-2)+1}[4]$. We apply the sequence of blowing-ups and blowing-downs in Figures~\ref{figure:sequence-TOI-32} to the MMP $(-1)$-data for $Y_0$ and we track down the $(-1)$-curves of the MMP $(-1)$-data. Then we have the following picture:
\\
\centering
\includegraphics[width=0.3\textwidth]{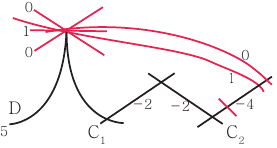}
\end{example}

As we see from the above example, one doesn't obtain, in general, a complete BO $(-1)$-data as given in Bhupal-Ono~\cite[\S5]{Bhupal-Ono-2012} only by the sequence of blow-ups and blow-downs to a MMP $(-1)$-data because some $C_i$'s in $Z_2$ are not appeared in the compactifying divisor.

\begin{definition}
The \emph{partial $(-1)$-data} corresponding to a given MMP $(-1)$-data is a part of BO $(-1)$-data obtained from the MMP $(-1)$-data by applying the sequence of blowing-ups and blowing-downs described in Figures~\ref{figure:sequence-TOI-32}, \ref{figure:sequence-TOI-31}.
\end{definition}

\begin{remark}[Notation for partial $(-1)$-data]
We use the notation $(TOI_m[n]; a_1 \times C_{j_1}, \dotsc , a_i \times C_{j_l})$ for denoting a partial $(-1)$-data for a $P$-resolution of number $n$ of $T_m$, $O_m$, or $I_m$ singularity, where $C_{j_l}$'s are the proper transforms of the curves that appear in the compactifying divisor after the sequence of blowing-ups and blowing-downs and $a_l$ implies that there are $a_i$ $(-1)$-curves.
\end{remark}

\begin{example}[Continued from Example~\ref{example:MMP-(-1)-data-T_6(5-2)+1}]

The partial $(-1)$-data for the tetrahedral singularity $T_{6(5-2)+1}$ are as follows:

\begin{multicols}{2}
\begin{itemize}
\item $(T_{6(5-2)+1}[1]; 3 \times C_3)$
\item $(T_{6(5-2)+1}[3]; 2 \times C_3)$
\end{itemize}

\columnbreak

\begin{itemize}
\item $(T_{6(5-2)+1}[2]; 2 \times C_3)$
\item $(T_{6(5-2)+1}[4]; 1 \times C_3)$
\end{itemize}
\end{multicols}

\end{example}

\begin{example}[Continued from Example~\ref{example:MMP-(-1)-data-I_30(4-2)+17}]

The following are the partial $(-1)$-data for the icosahedral singularity $I_{30(4-2)+17}$:

\begin{multicols}{2}
\begin{itemize}
\item $(I_{30(4-2)+17}[1]; 2 \times C_2, 1 \times C_3)$
\item $(I_{30(4-2)+17}[3]; 1 \times C_2, 1 \times C_3)$
\item $(I_{30(4-2)+17}[5]; 2 \times C_2)$
\item $(I_{30(4-2)+17}[7]; 1 \times C_2)$
\end{itemize}

\columnbreak

\begin{itemize}
\item $(I_{30(4-2)+17}[2]; 1 \times C_2, 1 \times C_3)$
\item $(I_{30(4-2)+17}[4]; 1 \times C_3)$
\item $(I_{30(4-2)+17}[6]; 2 \times C_2)$
\end{itemize}
\end{multicols}

\end{example}

In most cases, one can identify a given Milnor fiber with a minimal symplectic fillings by using the partial $(-1)$-data obtained from MMP. For example, the partial $(-1)$-data for $Y_0=T_{6(5-2)+1}[4]$ has only one $(-1)$-curve intersecting $C_3$. Then, by comparing with the BO $(-1)$-data for $X_0$ in Example~\ref{example:BO-(-1)-data-T_6(5-2)+1}, we can conclude that the $P$-resolution $Y_0$ corresponds to the minimal symplectic filling of number \#8.

However there are some cases that we cannot identify given $P$-resolutions with BO $(-1)$-data only by the partial $(-1)$-data of the $P$-resolutions because there exist two $P$-resolutions with the same partial $(-1)$-data. For example, the below are such pairs:
\[(T_{6(5-2)+1}[2], T_{6(5-2)+1}[3]), (I_{30(4-2)+17}[2],I_{30(4-2)+17}[2])\]

We divide pairs of $P$-resolutions with the same $(-1)$-data into two cases:

\begin{enumerate}[{Case} A.]
\item There are 16 pairs of entries in the list of Bhupal-Ono~\cite[\S5]{Bhupal-Ono-2012} such that two entries in each pair correspond to the same singularity with identical BO $(-1)$-data. Bhupal-Ono~\cite{Bhupal-Ono-2012} shows that, in each pairs, one entry has a BO $(-1)$-data of a minimal symplectic filling obtained by a sequence of blowups from $\mathbb{CP}^2$ and the other entry from $\mathbb{CP}^1 \times \mathbb{C}^1$. For example, $(T_{6(5-2)+1}[2], T_{6(5-2)+1}[3])$ is in Case A.

\item There are 19 pairs of $P$-resolutions of the same singularities (except Case A) with the same partial $(-1)$-data. One can check one by one that the singularities are of type $(3,1)$ and one of the corresponding BO $(-1)$-data in each pairs is of Case I and the other of Case II, where Case I and Case II are divided as in Remark~\ref{remark:CaseI-vs-CaseII}. For example, $(I_{30(4-2)+17}[2],I_{30(4-2)+17}[2])$ is in Case B.
\end{enumerate}

In the next two subsections, we identify $P$-resolutions (i.e., the MMP $(-1)$-data) in the above two cases with BO $(-1)$-data.

\subsection{Case A: Identical BO $(-1)$-data}

In the list of Bhupal-Ono~\cite[\S5]{Bhupal-Ono-2012}, there are 16 pairs of entries such that each pair consists of two entries corresponding to the same singularity with the same BO $(-1)$-data. But one of the entries in each pairs includes a BO $(-1)$-data of a minimal symplectic filling obtained from $\mathbb{CP}^2$ by a sequence of blow-ups and the other from $\mathbb{CP}^1 \times \mathbb{CP}^1$. Therefore we have 16 pairs of $P$-resolutions corresponding to the above 16 pairs such that two $P$-resolutions in each pairs have the same partial $(-1)$-data.

So we need to discriminate whether a given $P$-resolution in the above 16 pairs corresponds to a minimal symplectic filling from $\mathbb{CP}^2$ or from $\mathbb{CP}^1 \times \mathbb{CP}^1$.

\begin{proposition}
\label{proposition:CP2-vs-CP1*CP1}

In Table~\ref{table:CP2-vs-CP1*CP1}, we present all pairs of $P$-resolutions from the list of all $P$-resolutions in Section~\ref{section:list} that correspond to 16 pairs of the above identical BO $(-1)$-data. We denote also whether a given $P$-resolution corresponds to a minimal symplectic filling from $\mathbb{CP}^2$ or $\mathbb{CP}^1 \times \mathbb{CP}^1$ and the number of the corresponding minimal symplectic filling from the list of Bhupal-Ono~\cite[\S5]{Bhupal-Ono-2012}.
\end{proposition}

\begin{table}
\centering
\begin{tabular}{c c c c}
\toprule
$\mathbb{CP}^2$ & BO \# & $\mathbb{CP}^1 \times \mathbb{CP}^1$ & BO \# \\
\midrule
$T_{6(5-2)+1}[2]$ & 6 & $T_{6(5-2)+1}[3]$ & 7 \\
\midrule
$T_{6(4-2)+3}[3]$ & 18 & $T_{6(4-2)+3}[4]$ & 19 \\
\midrule
$T_{6(5-2)+3}[2]$ & 21 & $T_{6(5-2)+3}[4]$ & 22 \\
\midrule
$T_{6(5-2)+3}[3]$ & 24 & $T_{6(5-2)+3}[5]$ & 25 \\
\midrule
$O_{12(5-2)+1}[2]$ & 34 & $O_{12(5-2)+1}[3]$ & 35  \\
\midrule
$O_{12(3-2)+7}[3]$ & 46 & $O_{12(3-2)+7}[4]$ & 47 \\
\midrule
$O_{12(5-2)+7}[3]$ & 54 & $O_{12(5-2)+7}[4]$ & 55 \\
\midrule
$I_{30(5-2)+1}[3]$ & 68 & $I_{30(5-2)+1}[4]$ & 69 \\
\midrule
$I_{30(4-2)+7}[3]$ & 83 & $I_{30(4-2)+7}[4]$ & 84 \\
\midrule
$I_{30(5-2)+7}[5]$ & 88 & $I_{30(5-2)+7}[6]$ & 89 \\
\midrule
$I_{30(5-2)+7}[3]$ & 91 & $I_{30(5-2)+7}[4]$ & 92\\
\midrule
$I_{30(6-2)+7}[3]$ & 95 & $I_{30(6-2)+7}[4]$ & 96 \\
\midrule
$I_{30(4-2)+13}[4]$ & 102 & $I_{30(4-2)+13}[5]$ & 103 \\
\midrule
$I_{30(5-2)+13}[2]$ & 107 & $I_{30(5-2)+13}[4]$ & 108 \\
\midrule
$I_{30(5-2)+13}[3]$ & 110 & $I_{30(5-2)+13}[5]$ & 111 \\
\midrule
$I_{30(5-2)+19}[2]$ & 122 & $I_{30(5-2)+19}[3]$ & 123 \\
\bottomrule
\end{tabular}

\medskip

\caption{$\mathbb{CP}^2$ vs $\mathbb{CP}^1 \times \mathbb{CP}^1$}
\label{table:CP2-vs-CP1*CP1}
\end{table}

\begin{proof}
For each $P$-resolutions in the above table, we blow down successively $(-1)$-curves colored by green in the corresponding MMP $(-1)$-data so that we finally get a configuration of non-negative curves; Figures~\ref{figure:CP2-vs-CP1*CP1-first}--\ref{figure:CP2-vs-CP1*CP1-end}. We then are able to conclude whether the MMP $(-1)$-data come from $\mathbb{CP}^2$ or $\mathbb{CP}^1 \times \mathbb{CP}^1$.
\end{proof}

\begin{figure}[H]
  \centering
  \includegraphics{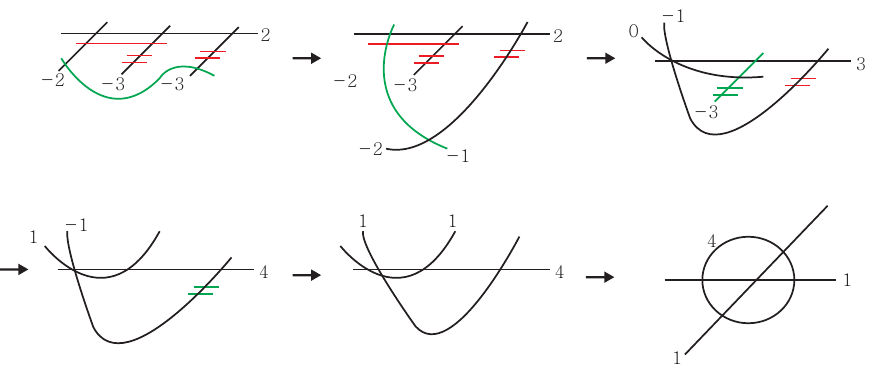}{}
  \caption{$T_{6(5-2)+1}[2]$ : $\mathbb{CP}^2$}
  \label{figure:CP2-vs-CP1*CP1-first}
\end{figure}

\begin{figure}[H]
  \centering
  \includegraphics{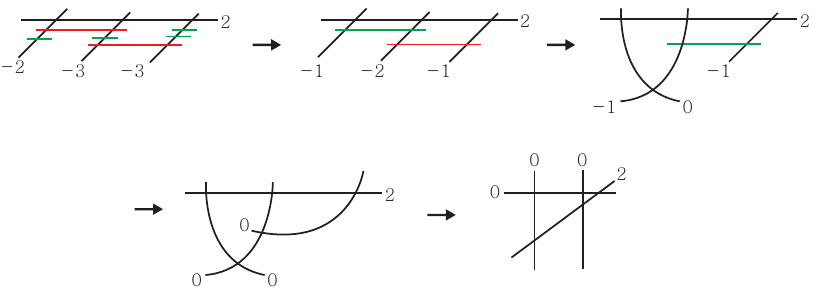}
  \caption{$T_{6(5-2)+1}[3]$ : $\mathbb{CP}^1 \times \mathbb{CP}^1$}
\end{figure}

\begin{figure}[H]
  \centering
  \includegraphics{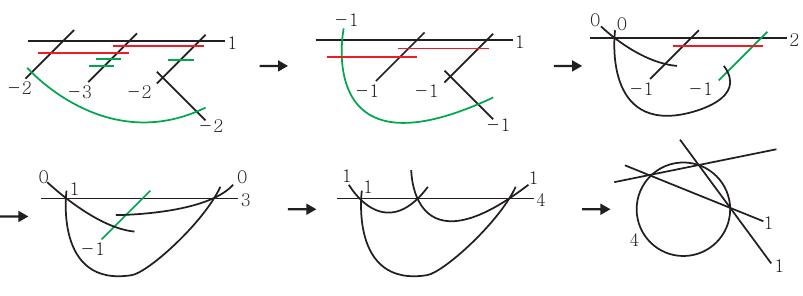}
  \caption{$T_{6(4-2)+3}[3]$ : $\mathbb{CP}^2$}
\end{figure}

\begin{figure}[H]
  \centering
  \includegraphics{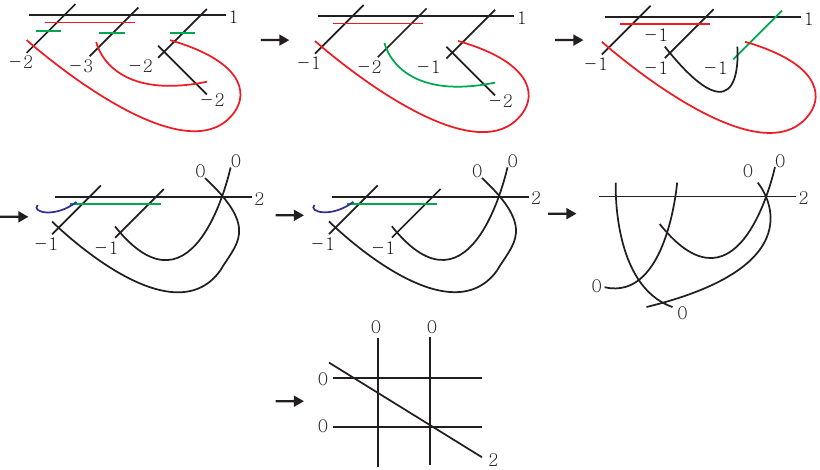}
  \caption{$T_{6(4-2)+3}[4]$ : $\mathbb{CP}^1 \times \mathbb{CP}^1$}
\end{figure}

\begin{figure}[H]
  \centering
  \includegraphics{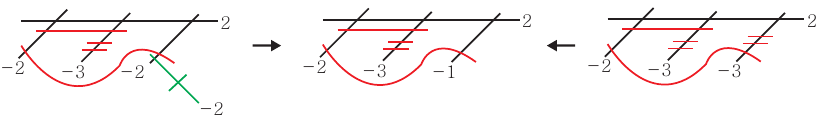}
  \caption{$T_{6(5-2)+3}[2]$ : $\mathbb{CP}^2$ ; same with $T_{6(5-2)+1}[2]$}
\end{figure}

\begin{figure}[H]
  \centering
  \includegraphics{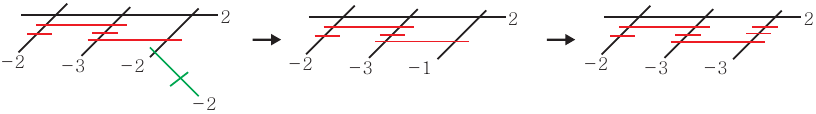}
  \caption{$T_{6(5-2)+3}[4]$ : $\mathbb{CP}^1 \times \mathbb{CP}^1$ ; same with $T_{6(5-2)+1}[3]$}
\end{figure}

\begin{figure}[H]
  \centering
  \includegraphics{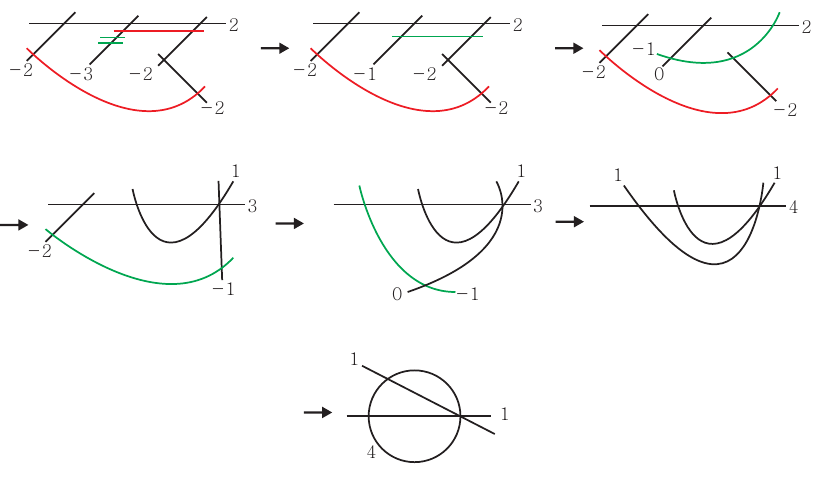}
  \caption{$T_{6(5-2)+3}[3] : \mathbb{CP}^2$}
\end{figure}

\begin{figure}[H]
  \centering
  \includegraphics{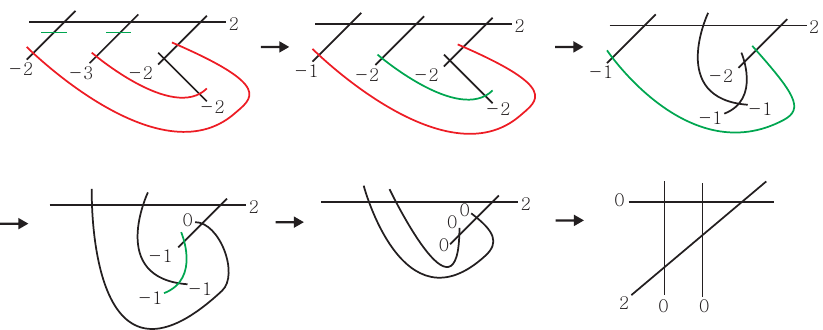}
  \caption{$T_{6(5-2)+3}[5]$ : $\mathbb{CP}^1 \times \mathbb{CP}^1$}
\end{figure}

\begin{figure}[H]
  \centering
  \includegraphics{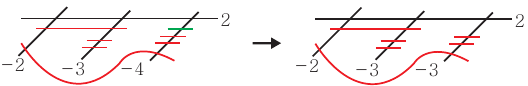}
  \caption{$O_{12(5-2)+1}[2]$ : $\mathbb{CP}^2$ ; same with $T_{6(5-2)+1}[2]$}
\end{figure}

\begin{figure}[H]
  \centering
  \includegraphics{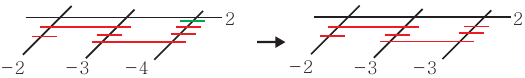}
  \caption{$O_{12(5-2)+1}[3]$ : $\mathbb{CP}^1 \times \mathbb{CP}^1$ ; same with $T_{6(5-2)+1}[3]$}
\end{figure}

\begin{figure}[H]
  \centering
  \includegraphics{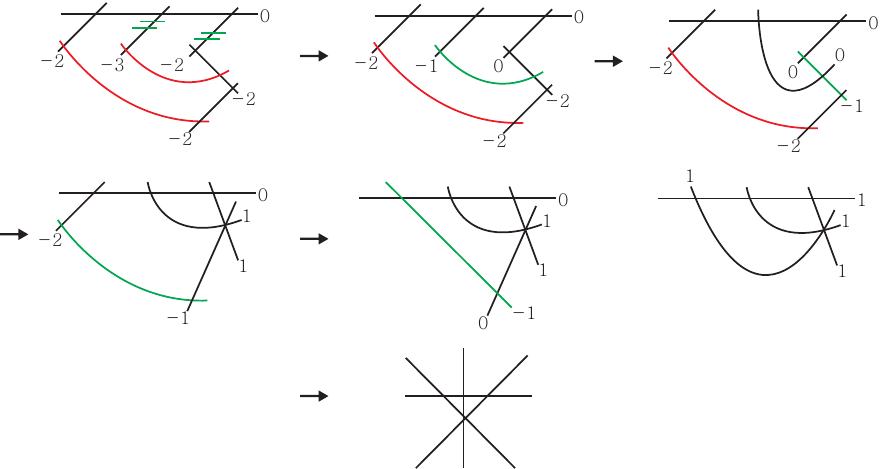}
  \caption{$O_{12(3-2)+7}[3]$ : $\mathbb{CP}^2$ ;same with $T_{6(4-2)+3}[3]$}
\end{figure}

\begin{figure}[H]
  \centering
  \includegraphics{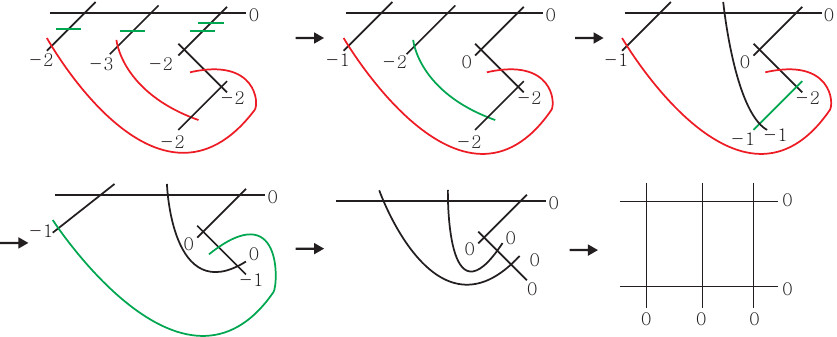}
  \caption{$O_{12(3-2)+7}[4]$ : $\mathbb{CP}^1 \times \mathbb{CP}^1$}
\end{figure}

\begin{figure}[H]
  \centering
  \includegraphics{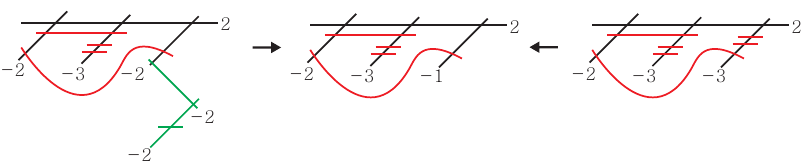}
  \caption{$O_{12(5-2)+7}[3]$ : $\mathbb{CP}^2$ ; same with $T_{6(5-2)+1}[2]$}
\end{figure}

\begin{figure}[H]
  \centering
  \includegraphics{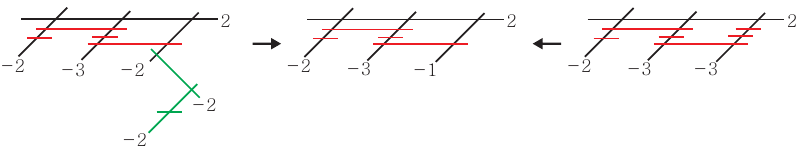}
  \caption{$O_{12(5-2)+7}[4]$ : $\mathbb{CP}^1 \times \mathbb{CP}^1$ ; same with $T_{6(5-2)+1}[3]$}
\end{figure}

\begin{figure}[H]
  \centering
  \includegraphics{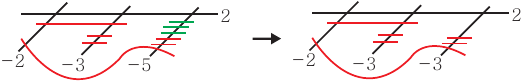}
  \caption{$I_{30(5-2)+1}[3]$ : $\mathbb{CP}^2$ ; same with $T_{6(5-2)+1}[2]$}
\end{figure}

\begin{figure}[H]
  \centering
  \includegraphics{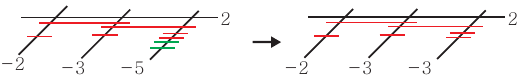}
  \caption{$I_{30(5-2)+1}[4]$ : $\mathbb{CP}^1 \times \mathbb{CP}^1$ ; same with $T_{6(5-2)+1}[3]$}
\end{figure}

\begin{figure}[H]
  \centering
  \includegraphics{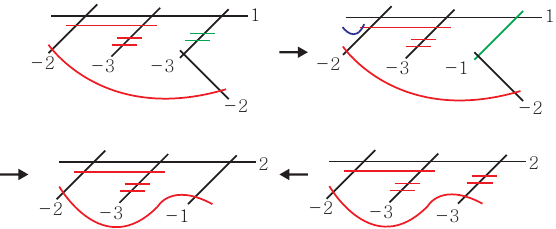}
  \caption{$I_{30(4-2)+7}[3]$ : $\mathbb{CP}^2$ ; same with $T_{6(5-2)+1}[2]$}
\end{figure}

\begin{figure}[H]
  \centering
  \includegraphics{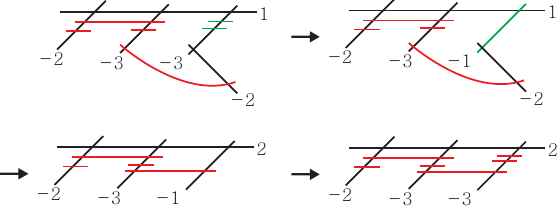}
  \caption{$I_{30(4-2)+7}[4]$ : $\mathbb{CP}^1 \times \mathbb{CP}^1$ ; same with $T_{6(5-2)+1}[3]$}
\end{figure}

\begin{figure}[H]
  \centering
  \includegraphics{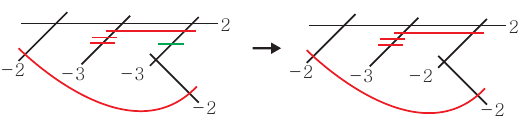}
  \caption{$I_{30(5-2)+7}[5]$ : $\mathbb{CP}^2$ ; same with $T_{6(5-2)+3}[3]$}
\end{figure}

\begin{figure}[H]
  \centering
  \includegraphics{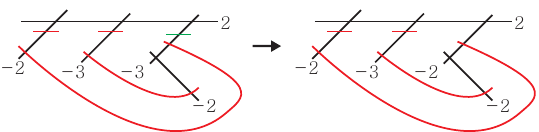}
  \caption{$I_{30(5-2)+7}[6]$ : $\mathbb{CP}^1 \times \mathbb{CP}^1$ ; same with $T_{6(5-2)+3}[5]$}
\end{figure}

\begin{figure}[H]
  \centering
  \includegraphics{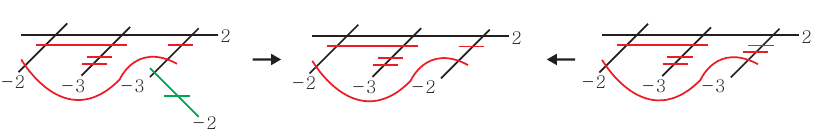}
  \caption{$I_{30(5-2)+7}[3]$ : $\mathbb{CP}^2$ ;same with $T_{6(5-2)+1}[2]$}
\end{figure}

\begin{figure}[H]
  \centering
  \includegraphics{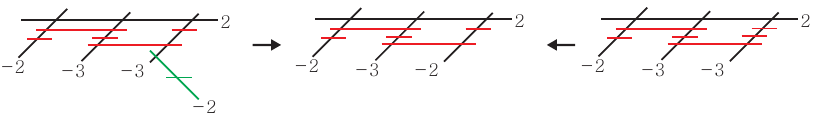}
  \caption{$I_{30(5-2)+7}[4]$ : $\mathbb{CP}^1 \times \mathbb{CP}^1$ ; same with $T_{6(5-2)+1}[3]$}
\end{figure}

\begin{figure}[H]
  \centering
  \includegraphics{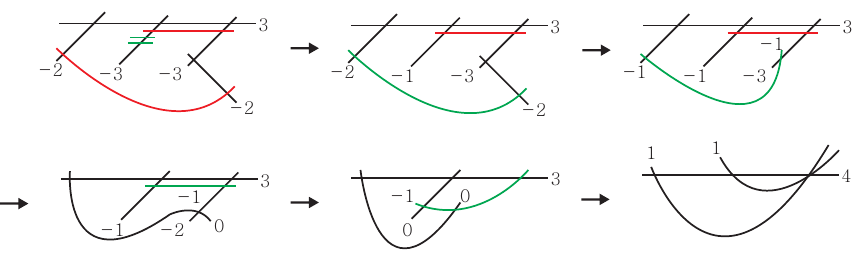}
  \caption{$I_{30(6-2)+7}[3]$ : $\mathbb{CP}^2$}
\end{figure}

\begin{figure}[H]
  \centering
  \includegraphics{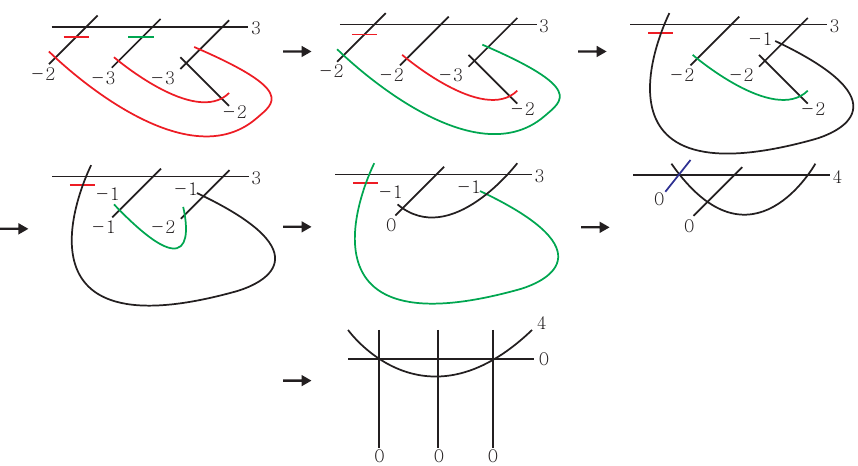}
  \caption{$I_{30(6-2)+7}[4]$ : $\mathbb{CP}^1 \times \mathbb{CP}^1$}
\end{figure}

\begin{figure}[H]
  \centering
  \includegraphics{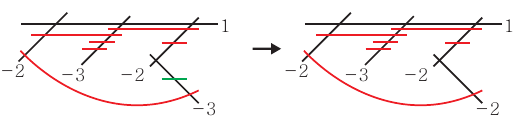}
  \caption{$I_{30(4-2)+13}[4]$ : $\mathbb{CP}^2$ ; same with $T_{6(4-2)+3}[3]$}
\end{figure}

\begin{figure}[H]
  \centering
  \includegraphics{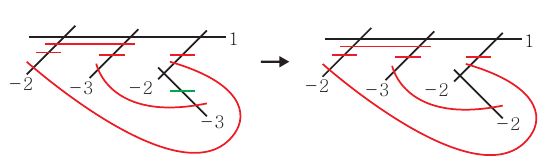}
  \caption{$I_{30(4-2)+13}[5]$ : $\mathbb{CP}^1 \times \mathbb{CP}^1$ ; same with $T_{6(4-2)+3}[4]$}
\end{figure}

\begin{figure}[H]
  \centering
  \includegraphics{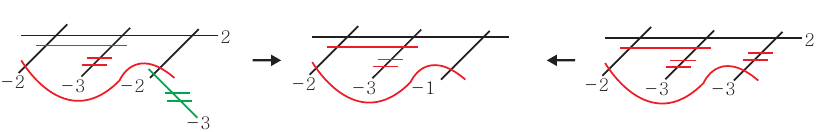}
  \caption{$I_{30(5-2)+13}[2]$ : $\mathbb{CP}^2$ ; same with $T_{6(5-2)+1}[2]$}
\end{figure}

\begin{figure}[H]
  \centering
  \includegraphics{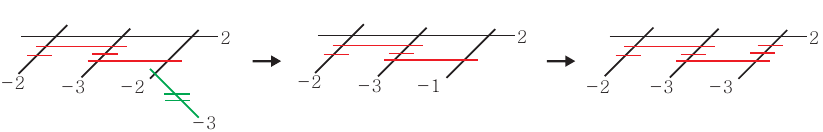}
  \caption{$I_{30(5-2)+13}[4]$ : $\mathbb{CP}^1 \times \mathbb{CP}^1$ ;s ame with $T_{6(5-2)+1}[3]$}
\end{figure}

\begin{figure}[H]
  \centering
  \includegraphics{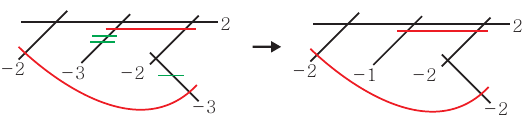}
  \caption{$I_{30(5-2)+13}[3]$ : $\mathbb{CP}^2$ ; same with $T_{6(5-2)+3}[3]$}
\end{figure}

\begin{figure}[H]
  \centering
  \includegraphics{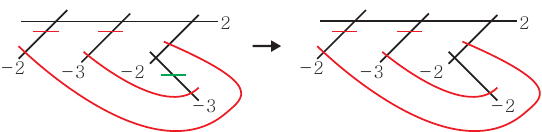}
  \caption{$I_{30(5-2)+13}[5]$ : $\mathbb{CP}^1 \times \mathbb{CP}^1$ ; same with $T_{6(5-2)+3}[5]$}
\end{figure}

\begin{figure}[H]
  \centering
  \includegraphics{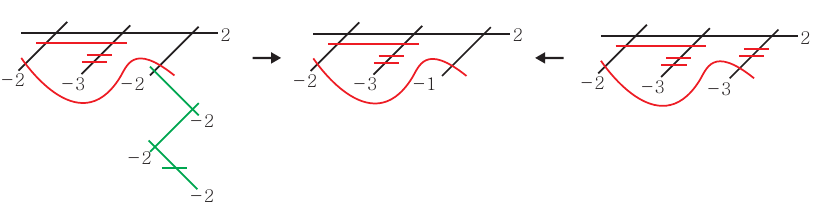}
  \caption{$I_{30(5-2)+19}[2]$ : $\mathbb{CP}^2$ ; same with $T_{6(5-2)+1}[2]$}
\end{figure}

\begin{figure}[H]
  \centering
  \includegraphics{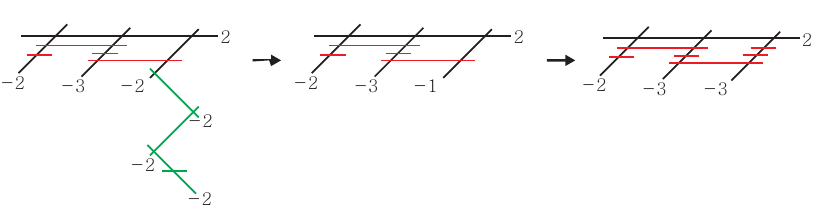}
  \caption{$I_{30(5-2)+19}[3]$ : $\mathbb{CP}^1 \times \mathbb{CP}^1$ ;same with $T_{6(5-2)+1}[3]$}
  \label{figure:CP2-vs-CP1*CP1-end}
\end{figure}

\subsection{Case B: Identical partial $(-1)$-data}

Except Case A, there are 19 pairs of $P$-resolutions that have the same partial $(-1)$-data. They are $P$-resolutions of singularities of type $(3,1)$. So we need to identify whether a give $P$-resolution in the pairs belongs to Case I or Case II.

\begin{proposition}
In Table~\ref{table:CaseI-vs-CaseII} we present 19 pairs of $P$-resolutions from the list of all $P$-resolutions in Section~\ref{section:list} with the same partial $(-1)$-data except pairs in Case A. We denote whether a given $P$-resolutions belongs to Case I or Case II and the number of the corresponding minimal symplectic fillings in the list of Bhupal-Ono~\cite[\S5]{Bhupal-Ono-2012}.
\end{proposition}

\begin{table}
\centering

\begin{tabular}{c c c c} 
\toprule
Case I & BO \# & Case II & BO \# \\ 
\midrule
$T_{6(4-2)+5}[2]$ & 131 &   $T_{6(4-2)+5}[3]$ & 211 \\
\midrule
$T_{6(4-2)+5}[5]$ & 132 &   $T_{6(4-2)+5}[6]$ & 212 \\
\midrule
$T_{6(5-2)+5}[2]$ & 134 &   $T_{6(5-2)+5}[3]$ & 214 \\
\midrule
$O_{12(4-2)+5}[2]$ & 140 &   $O_{12(4-2)+5}[3]$ & 216 \\
\midrule
$O_{12(5-2)+5}[2]$ & 142 &   $O_{12(5-2)+5}[3]$ & 218\\
\midrule
$O_{12(5-2)+5}[4]$ & 143 &   $O_{12(5-2)+5}[5]$ & 219 \\
\midrule
$O_{12(5-2)+11}[3]$ & 158 &   $O_{12(5-2)+11}[4]$ & 226 \\
\midrule
$I_{30(4-2)+11}[2]$ & 165 &   $I_{30(4-2)+11}[3]$ & 228 \\
\midrule
$I_{30(5-2)+11}[2]$ & 167 &   $I_{30(5-2)+11}[3]$ & 230 \\
\midrule
$I_{30(5-2)+11}[4]$ & 168 &   $I_{30(5-2)+11}[5]$ & 231 \\
\midrule
$I_{30(4-2)+17}[2]$ & 181 &   $I_{30(4-2)+17}[3]$ & 238 \\
\midrule
$I_{30(4-2)+17}[5]$ & 180 &   $I_{30(4-2)+17}[6]$ & 239 \\
\midrule
$I_{30(5-2)+17}[2]$ & 185 &   $I_{30(5-2)+17}[3]$ & 242 \\
\midrule
$I_{30(5-2)+17}[4]$ & 184 &   $I_{30(5-2)+17}[5]$ & 243 \\
\midrule
$I_{30(4-2)+23}[2]$ & 194 &   $I_{30(4-2)+23}[4]$ & 249 \\
\midrule
$I_{30(4-2)+23}[6]$ & 193 &   $I_{30(4-2)+23}[7]$ & 250 \\
\midrule
$I_{30(5-2)+23}[2]$ & 197  &   $I_{30(5-2)+23}[3]$ & 253 \\
\midrule
$I_{30(4-2)+29}[2]$ & 204 &   $I_{30(4-2)+29}[3]$ & 256 \\
\midrule
$I_{30(5-2)+29}[2]$ & 206  &   $I_{30(5-2)+29}[3]$ & 257 \\
\bottomrule
\end{tabular}

\medskip

\caption{Case I vs Case II}
\label{table:CaseI-vs-CaseII}
\end{table}

\begin{proof}
According to Remark~\ref{remark:CaseI-vs-CaseII}, the set of minimal symplectic fillings are divided into two classes, Case I and Case II, by the existence of the $(-1)$-curve $E$ in $Z_2$ such that $B \cdot E=1$, and $C_i \cdot E = 0$ for all $i$ or $C_i \cdot E=1$ for some $i$.

Let $\widetilde{B}$ is the curve in the compactifying divisor that is transformed into the curve $B$ via the sequence of blow-ups and blow-downs described in Figures~\ref{figure:sequence-TOI-32}, \ref{figure:sequence-TOI-31}. For example, $\widetilde{B}$ for $T_{6(4-2)+5}$ is the second curve in the middle arm; see the below picture. So if there exists a $(-1)$-curve $E$ intersecting only $\widetilde{B}$, then one can conclude that the MMP $(-1)$-data is in Case I. For example, one can easily check that $T_{6(4-2)+5}[2]$ is of Case I but $T_{6(4-2)+5}[3]$ is of Case II.

\begin{minipage}{\textwidth}
\centering
\includegraphics[width=0.4\textwidth]{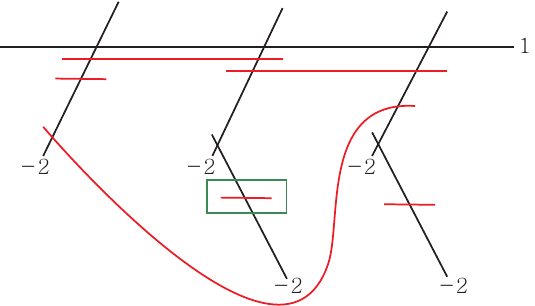} \quad \includegraphics[width=0.4\textwidth]{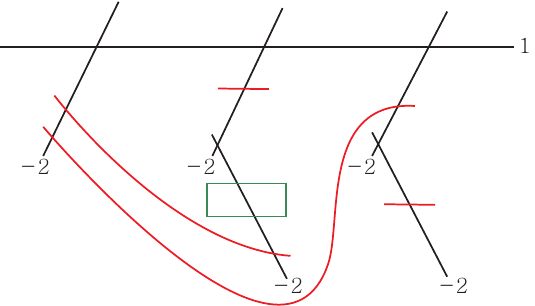}

$T_{6(4-2)+5}[2]$ \hspace{3.5cm} $T_{6(4-2)+5}[3]$
\end{minipage}

By a similar method, the assertion follows.
\end{proof}

\section{$P$-resolutions and their invariants}
\label{section:list}

We now present all $P$-resolutions with their MMP $(-1)$-data, invariants (the dimension, $\dim$, of the corresponding irreducible component of $\Def(X_0)$ and Milnor number $\mu$) and the corresponding BO entries.

\subsection*{Tetrahedral $T_{6(n-2)+1}$}\hfill

\begin{minipage}{\textwidth}
\plist{E_6}


\quad
\includegraphics[width=0.3\textwidth]{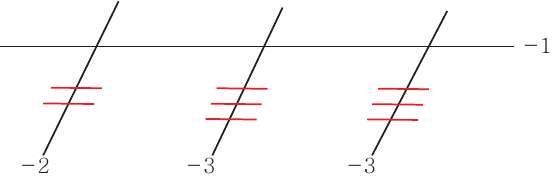}\quad
$\dim= 6$, $\mu= 6$, BO \#1
\end{minipage}

\begin{minipage}{\textwidth}
\plist{T_{6(3-2)+1}[1]}

%
\quad
\includegraphics[width=0.3\textwidth]{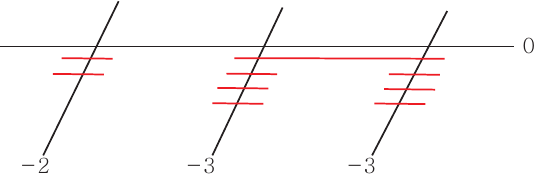}\quad
$\dim= 7$, $\mu= 6$, BO \#2
\end{minipage}

\begin{minipage}{\textwidth}
\plist{T_{6(4-2)+1}[1]}

%
\quad
\includegraphics[width=0.3\textwidth]{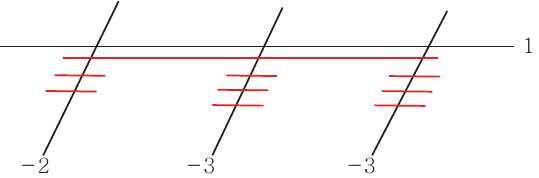}\quad
$\dim= 8$, $\mu= 6$, BO \#3
\end{minipage}

\begin{minipage}{\textwidth}
\plist{T_{6(4-2)+1}[2]}

%
\quad
\includegraphics[width=0.3\textwidth]{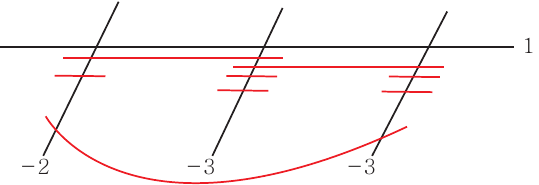}\quad
$\dim= 6$, $\mu= 5$, BO \#4
\end{minipage}

\begin{minipage}{\textwidth}
\plist{T_{6(5-2)+1}[1]}

%
\quad
\includegraphics[width=0.3\textwidth]{T-1-5-1}\quad
$\dim= 9$, $\mu= 6$, BO \#5
\end{minipage}

\begin{minipage}{\textwidth}
\plist{T_{6(5-2)+1}[2]}

%
\quad
\includegraphics[width=0.3\textwidth]{T-1-5-2}\quad
$\dim= 5$, $\mu= 4$, BO \#6
\end{minipage}

\begin{minipage}{\textwidth}
\plist{T_{6(5-2)+1}[3]}

%
\quad
\includegraphics[width=0.3\textwidth]{T-1-5-3}\quad
$\dim= 5$, $\mu= 4$, BO \#7
\end{minipage}

\begin{minipage}{\textwidth}
\plist{T_{6(5-2)+1}[4]}

%
\quad
\includegraphics[width=0.3\textwidth]{T-1-5-4}\quad
$\dim= 5$, $\mu= 4$, BO \#8
\end{minipage}

\begin{minipage}{\textwidth}
\plist{T_{6(6-2)+1}[1]}

%
\quad
\includegraphics[width=0.3\textwidth]{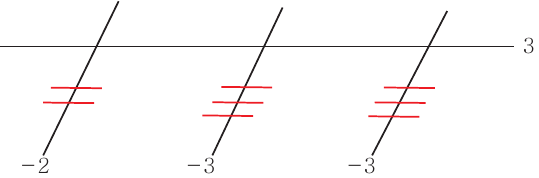}\quad
$\dim= 10$, $\mu= 6$, BO \#9
\end{minipage}

\begin{minipage}{\textwidth}
\plist{T_{6(6-2)+1}[2]}

%
\quad
\includegraphics[width=0.3\textwidth]{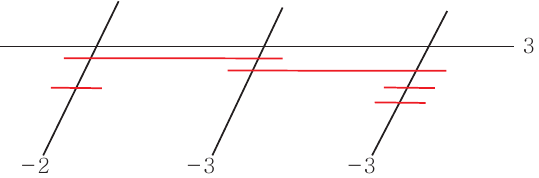}\quad
$\dim= 4$, $\mu= 3$, BO \#10
\end{minipage}

\begin{minipage}{\textwidth}
\plist{T_{6(6-2)+1}[3]}

%
\quad
\includegraphics[width=0.3\textwidth]{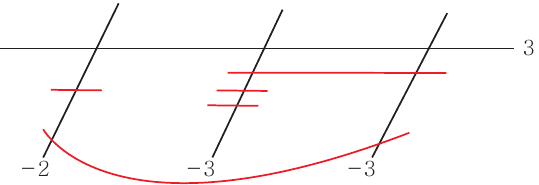}\quad
$\dim= 4$, $\mu= 3$, BO \#11
\end{minipage}

\begin{minipage}{\textwidth}
\plist{T_{6(n-2)+1}[1], n>6}

%
\quad
\includegraphics[width=0.3\textwidth]{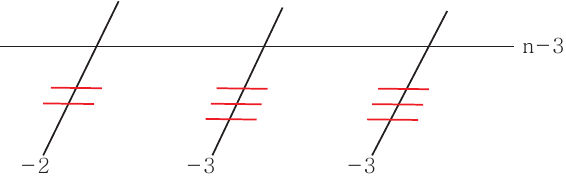}\quad
$\dim= n+4$,  $\mu= 6$, BO \#12
\end{minipage}

\subsection*{Tetrahedral $T_{6(n-2)+3}$}\hfill

\begin{minipage}{\textwidth}
\plist{T_{6(2-2)+3}[1]}

%
   \quad
\includegraphics[width=0.3\textwidth]{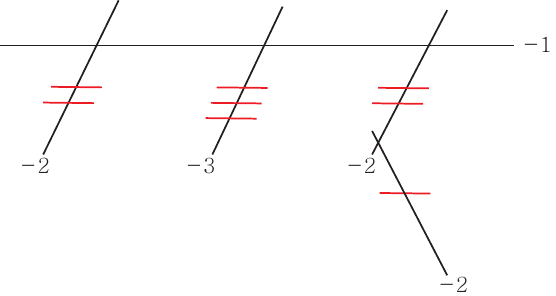}\quad
$\dim= 6$, $\mu= 5$, BO \#13
\end{minipage}

\begin{minipage}{\textwidth}
\plist{T_{6(3-2)+3}[1]}

%
\quad
\includegraphics[width=0.3\textwidth]{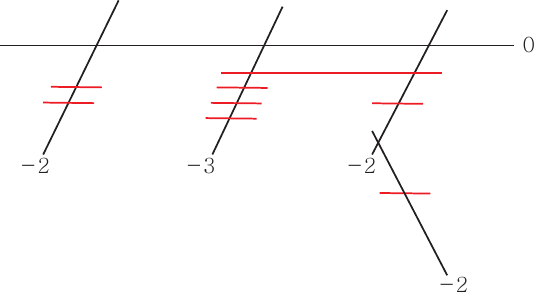}\quad
$\dim= 7$, $\mu= 5$, BO \#14
\end{minipage}

\begin{minipage}{\textwidth}
\plist{T_{6(3-2)+3}[2]}

%
\quad
\includegraphics[width=0.3\textwidth]{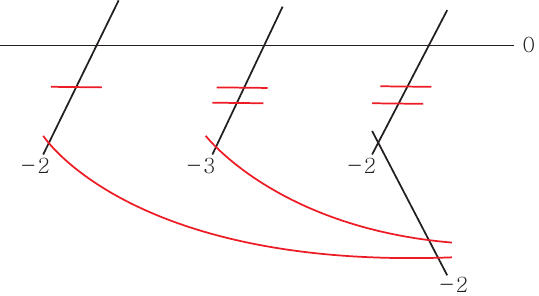}\quad
$\dim= 5$, $\mu= 4$, BO \#15
\end{minipage}

\begin{minipage}{\textwidth}
\plist{T_{6(4-2)+3}[1]}

%
\quad
\includegraphics[width=0.3\textwidth]{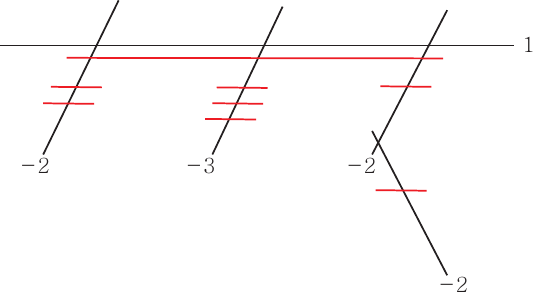}\quad
$\dim= 8$, $\mu= 5$, BO \#16
\end{minipage}

\begin{minipage}{\textwidth}
\plist{T_{6(4-2)+3}[2]}

%
\quad
\includegraphics[width=0.3\textwidth]{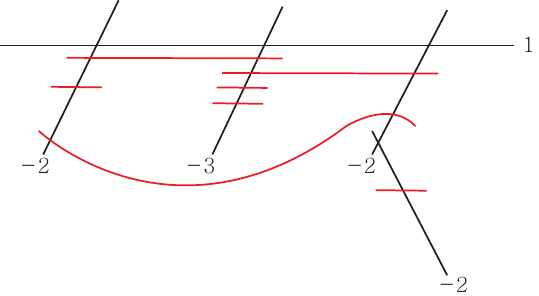}\quad
$\dim= 6$, $\mu= 4$, BO \#17
\end{minipage}

\begin{minipage}{\textwidth}
\plist{T_{6(4-2)+3}[3]}

%
\quad
\includegraphics[width=0.3\textwidth]{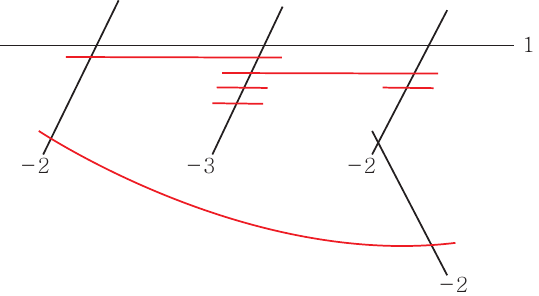}\quad
$\dim= 4$, $\mu= 3$, BO \#18
\end{minipage}

\begin{minipage}{\textwidth}
\plist{T_{6(4-2)+3}[4]}

%
\quad
\includegraphics[width=0.3\textwidth]{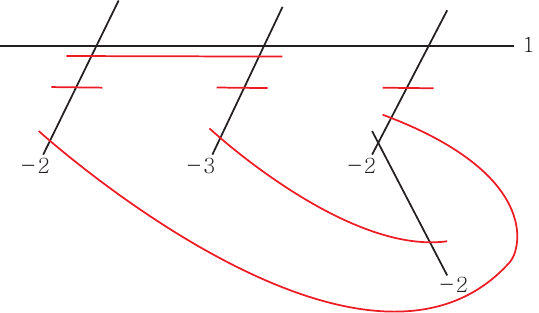}\quad
$\dim= 4$, $\mu= 3$, BO \#19
\end{minipage}

\begin{minipage}{\textwidth}
\plist{T_{6(5-2)+3}[1]}

%
\quad
\includegraphics[width=0.3\textwidth]{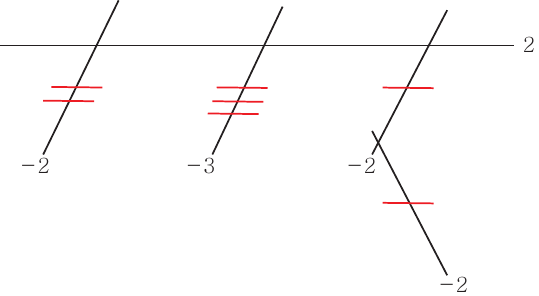}\quad
$\dim= 9$, $\mu= 5$, BO \#20
\end{minipage}

\begin{minipage}{\textwidth}
\plist{T_{6(5-2)+3}[2]}

%
\quad
\includegraphics[width=0.3\textwidth]{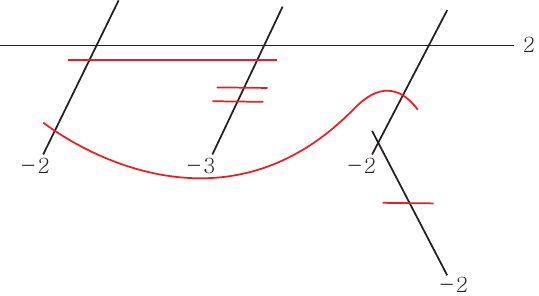}\quad
$\dim= 5$, $\mu= 3$, BO \#21
\end{minipage}

\begin{minipage}{\textwidth}
\plist{T_{6(5-2)+3}[3]}

%
\quad
\includegraphics[width=0.3\textwidth]{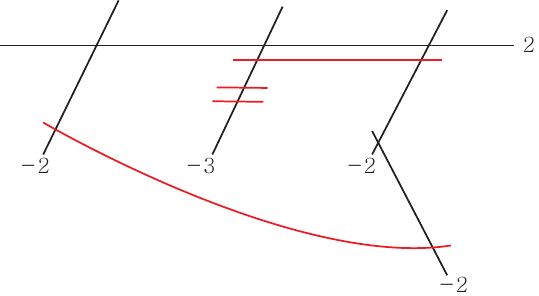}\quad
$\dim= 3$, $\mu= 2$, BO \#24
\end{minipage}

\begin{minipage}{\textwidth}
\plist{T_{6(5-2)+3}[4]}

%
\quad
\includegraphics[width=0.3\textwidth]{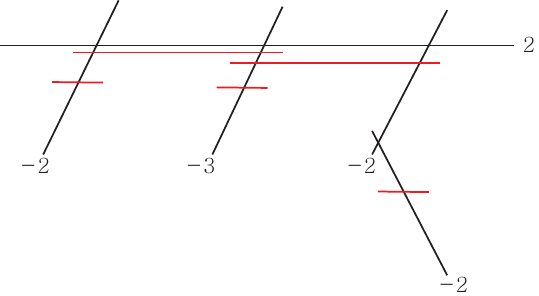}\quad
$\dim= 5$, $\mu= 3$, BO \#22
\end{minipage}

\begin{minipage}{\textwidth}
\plist{T_{6(5-2)+3}[5]}

%
\quad
\includegraphics[width=0.3\textwidth]{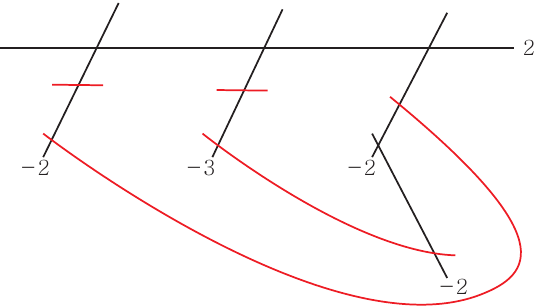}\quad
$\dim= 3$, $\mu= 2$, BO \#25
\end{minipage}

\begin{minipage}{\textwidth}
\plist{T_{6(5-2)+3}[6]}

%
\quad
\includegraphics[width=0.3\textwidth]{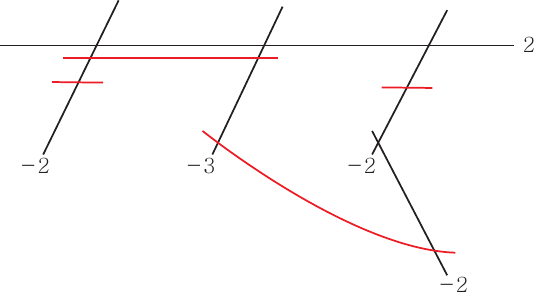}\quad
$\dim= 3$, $\mu= 2$, BO \#23
\end{minipage}

\begin{minipage}{\textwidth}
\plist{T_{6(6-2)+3}[1]}

%
\quad
\includegraphics[width=0.3\textwidth]{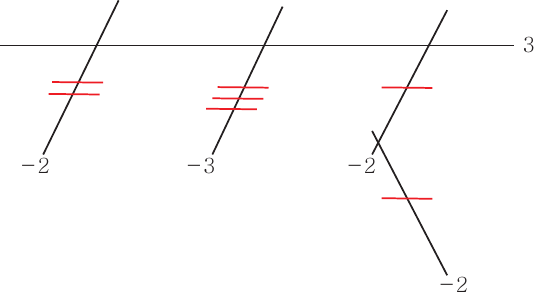}\quad
$\dim= 10$, $\mu= 5$, BO \#26
\end{minipage}

\begin{minipage}{\textwidth}
\plist{T_{6(6-2)+3}[2]}

%
\quad
\includegraphics[width=0.3\textwidth]{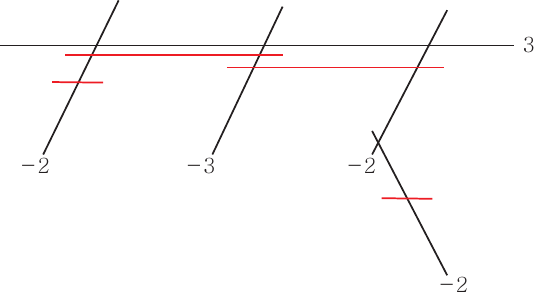}\quad
$\dim= 4$, $\mu= 2$, BO \#27
\end{minipage}

\begin{minipage}{\textwidth}
\plist{T_{6(n-2)+3}[1], n>6}

%
\quad
\includegraphics[width=0.3\textwidth]{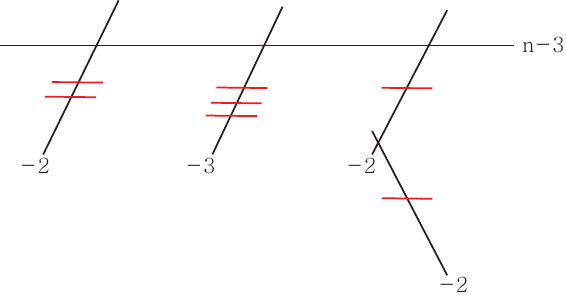}\quad
$\dim= n+4$, $\mu= 5$, BO \#28
\end{minipage}

\subsection*{Tetrahedral $T_{6(n-2)+5}$}\hfill

\begin{minipage}{\textwidth}
\plist{T_{6(2-2)+5}[1]}

%
\quad
\includegraphics[width=0.3\textwidth]{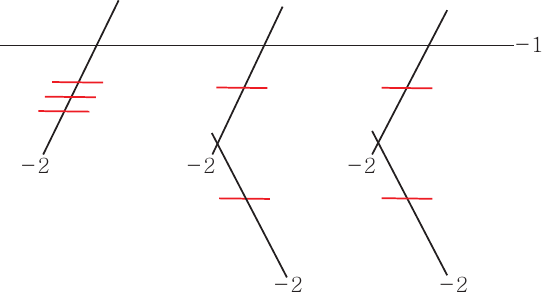}\quad
$\dim= 6$, $\mu= 4$, BO \#127
\end{minipage}

\begin{minipage}{\textwidth}
\plist{T_{6(2-2)+5}[2]}

%
\quad
\includegraphics[width=0.3\textwidth]{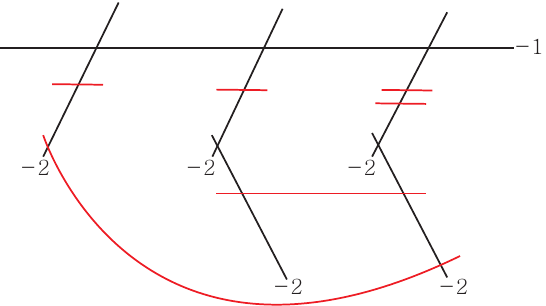}\quad
$\dim= 4$, $\mu= 3$, BO \#208
\end{minipage}

\begin{minipage}{\textwidth}
\plist{T_{6(3-2)+5}[1]}

%
\quad
\includegraphics[width=0.3\textwidth]{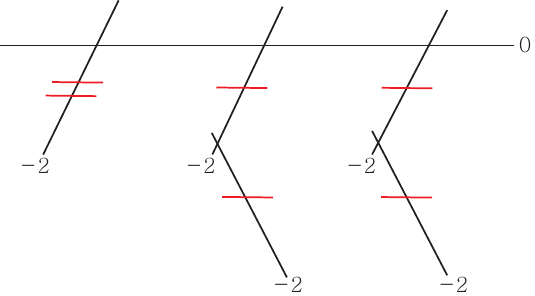}\quad
$\dim= 7$, $\mu= 4$, BO \#128
\end{minipage}

\begin{minipage}{\textwidth}
\plist{T_{6(3-2)+5}[2]}

%
\quad
\includegraphics[width=0.3\textwidth]{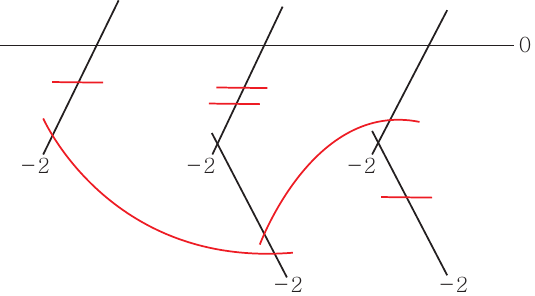}\quad
$\dim= 5$, $\mu= 3$, BO \#209
\end{minipage}

\begin{minipage}{\textwidth}
\plist{T_{6(3-2)+5}[3]}

%
\quad
\includegraphics[width=0.3\textwidth]{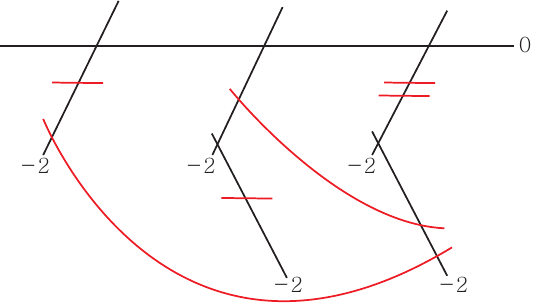}\quad
$\dim= 5$, $\mu= 3$, BO \#129
\end{minipage}

\begin{minipage}{\textwidth}
\plist{T_{6(3-2)+5}[4]}

%
\quad
\includegraphics[width=0.3\textwidth]{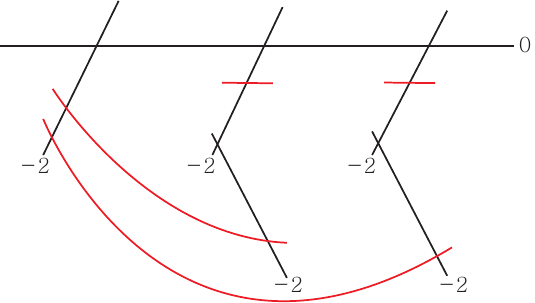}\quad
$\dim= 3$, $\mu= 2$, BO \#210
\end{minipage}

\begin{minipage}{\textwidth}
\plist{T_{6(4-2)+5}[1]}

%
\quad
\includegraphics[width=0.3\textwidth]{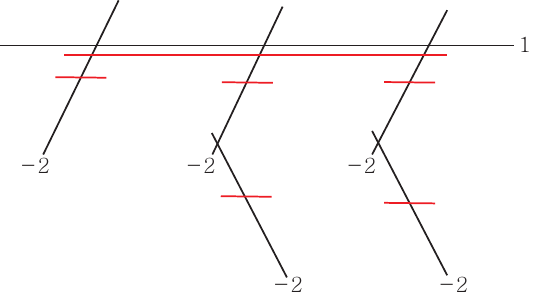}\quad
$\dim= 8$, $\mu= 4$, BO \#130
\end{minipage}

\begin{minipage}{\textwidth}
\plist{T_{6(4-2)+5}[2]}

%
\quad
\includegraphics[width=0.3\textwidth]{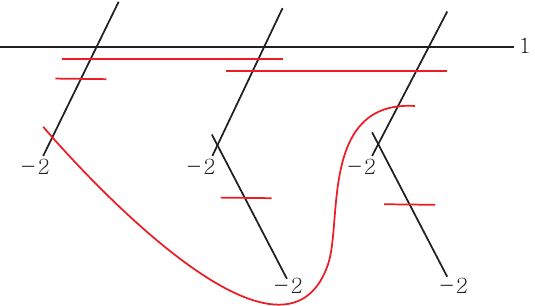}\quad
$\dim= 6$, $\mu= 3$, BO \#131
\end{minipage}

\begin{minipage}{\textwidth}
\plist{T_{6(4-2)+5}[3]}

%
\quad
\includegraphics[width=0.3\textwidth]{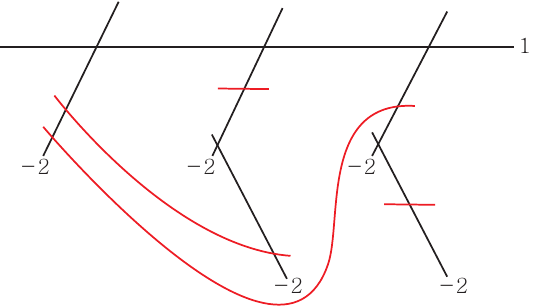}\quad
$\dim= 4$, $\mu= 2$, BO \#211
\end{minipage}

\begin{minipage}{\textwidth}
\plist{T_{6(4-2)+5}[4]}

%
\quad
\includegraphics[width=0.3\textwidth]{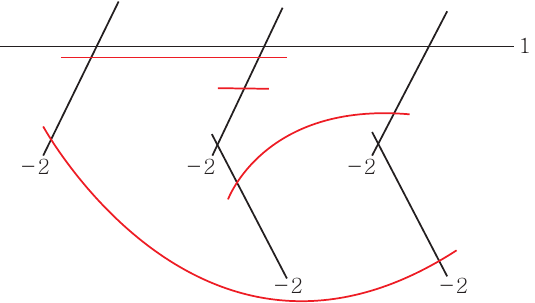}\quad
$\dim= 2$, $\mu= 1$, BO \#213
\end{minipage}

\begin{minipage}{\textwidth}
\plist{T_{6(4-2)+5}[5]}

%
\quad
\includegraphics[width=0.3\textwidth]{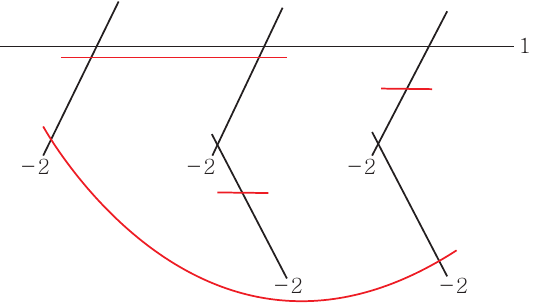}\quad
$\dim= 4$, $\mu= 2$, BO \#132
\end{minipage}

\begin{minipage}{\textwidth}
\plist{T_{6(4-2)+5}[6]}

%
\quad
\includegraphics[width=0.3\textwidth]{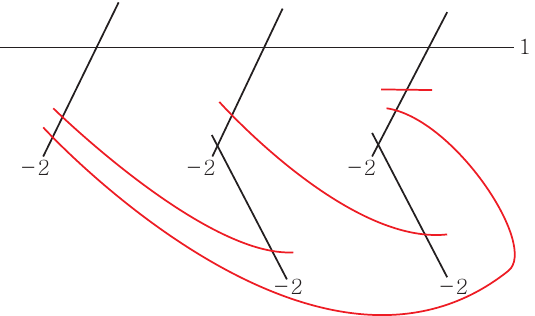}\quad
$\dim= 2$, $\mu= 1$, BO \#212
\end{minipage}

\begin{minipage}{\textwidth}
\plist{T_{6(5-2)+5}[1]}

%
\quad
\includegraphics[width=0.3\textwidth]{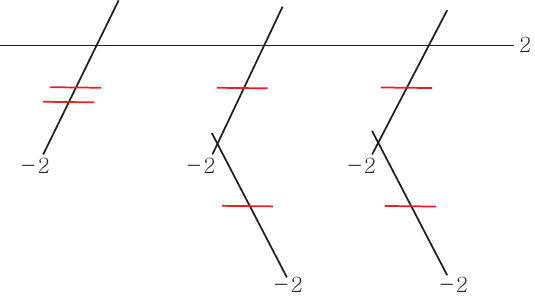}\quad
$\dim= 9$, $\mu= 4$, BO \#133
\end{minipage}

\begin{minipage}{\textwidth}
\plist{T_{6(5-2)+5}[2]}

%
\quad
\includegraphics[width=0.3\textwidth]{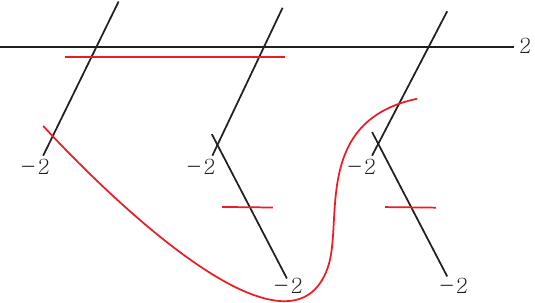}\quad
$\dim= 5$, $\mu= 2$, BO \#134
\end{minipage}

\begin{minipage}{\textwidth}
\plist{T_{6(5-2)+5}[3]} \quad

%
\quad
\includegraphics[width=0.3\textwidth]{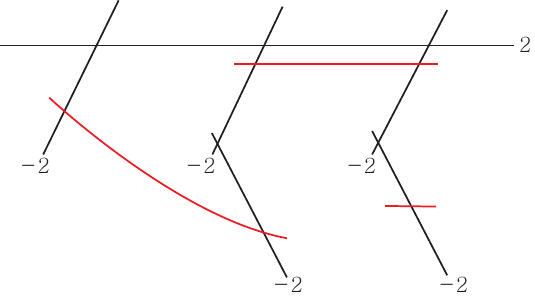}\quad
$\dim= 3$, $\mu= 1$, BO \#214
\end{minipage}

\begin{minipage}{\textwidth}
\plist{T_{6(5-2)+5}[4]}

%
\quad
\includegraphics[width=0.3\textwidth]{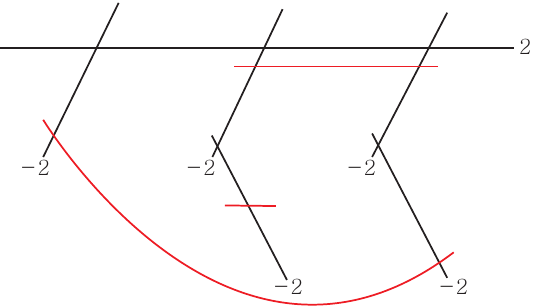}\quad
$\dim= 3$, $\mu= 1$, BO \#135
\end{minipage}

\begin{minipage}{\textwidth}
\plist{T_{6(n-2)+5}[1], n>5}

%
\quad
\includegraphics[width=0.3\textwidth]{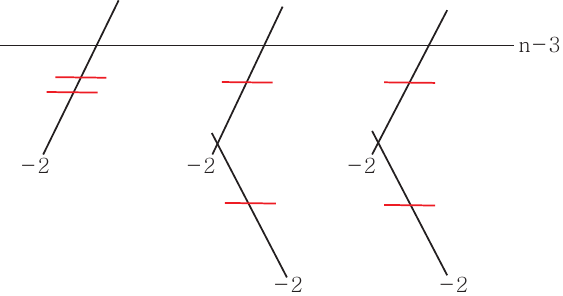}\quad
$\dim= n+4$, $\mu= 4$, BO \#136
\end{minipage}

\subsection*{Octahedral $O_{12(n-2)+1}$}\hfill

\begin{minipage}{\textwidth}
\plist{E_7}

%
\quad
\includegraphics[width=0.3\textwidth]{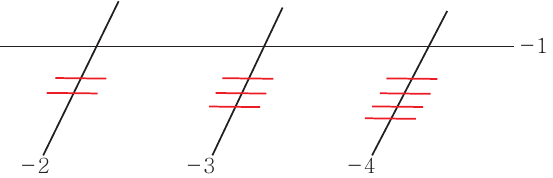}\quad
$\dim= 8$, $\mu= 7$, BO \#29
\end{minipage}

\begin{minipage}{\textwidth}
\plist{O_{12(3-2)+1}[1]}

%
\quad
\includegraphics[width=0.3\textwidth]{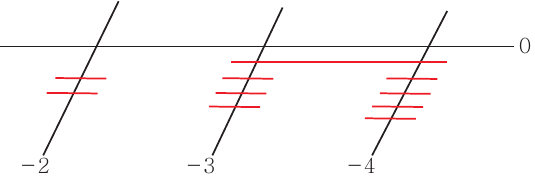}\quad
$\dim= 8$, $\mu= 7$, BO \#30
\end{minipage}

\begin{minipage}{\textwidth}
\plist{O_{12(4-2)+1}[1]}

%
\quad
\includegraphics[width=0.3\textwidth]{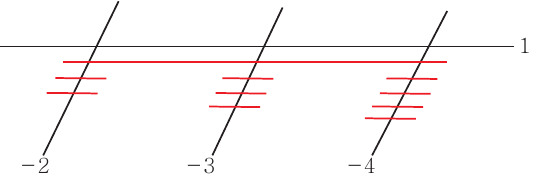}\quad
$\dim= 9$, $\mu= 7$, BO \#31
\end{minipage}

\begin{minipage}{\textwidth}
\plist{O_{12(4-2)+1}[2]}

%
\quad
\includegraphics[width=0.3\textwidth]{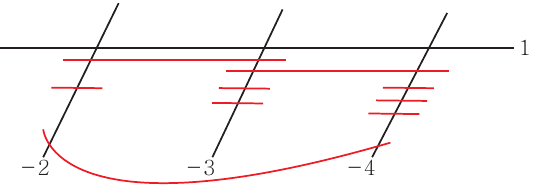}\quad
$\dim= 7$, $\mu= 6$, BO \#32
\end{minipage}

\begin{minipage}{\textwidth}
\plist{O_{12(5-2)+1}[1]}

%
\quad
\includegraphics[width=0.3\textwidth]{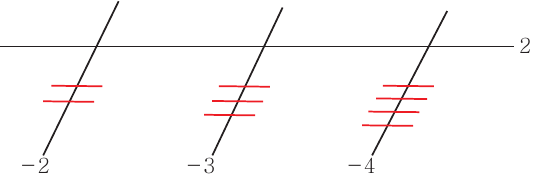}\quad
$\dim= 10$, $\mu= 7$, BO \#33
\end{minipage}

\begin{minipage}{\textwidth}
\plist{O_{12(5-2)+1}[2]}

%
\quad
\includegraphics[width=0.3\textwidth]{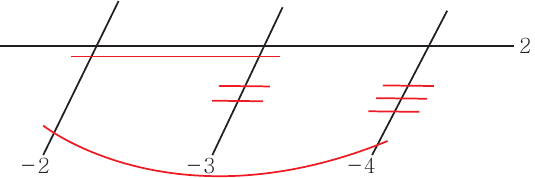}\quad
$\dim= 6$, $\mu= 5$, BO \#34
\end{minipage}

\begin{minipage}{\textwidth}
\plist{O_{12(5-2)+1}[3]}

%
\quad
\includegraphics[width=0.3\textwidth]{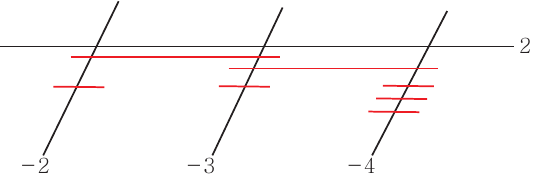}\quad
$\dim= 6$, $\mu= 5$, BO \#35
\end{minipage}

\begin{minipage}{\textwidth}
\plist{O_{12(5-2)+1}[4]}

%
\quad
\includegraphics[width=0.3\textwidth]{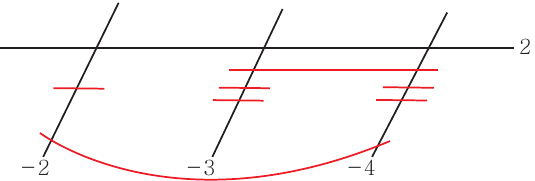}\quad
$\dim= 6$, $\mu= 5$, BO \#36
\end{minipage}

\begin{minipage}{\textwidth}
\plist{O_{12(6-2)+1}[1]}

%
\quad
\includegraphics[width=0.3\textwidth]{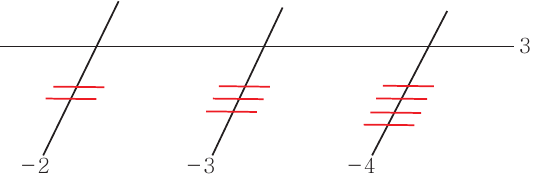}\quad
$\dim= 11$, $\mu= 7$, BO \#37
\end{minipage}

\begin{minipage}{\textwidth}
\plist{O_{12(6-2)+1}[2]}

%
\quad
\includegraphics[width=0.3\textwidth]{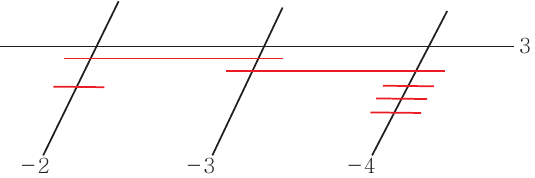}\quad
$\dim= 5$, $\mu= 4$, BO \#38
\end{minipage}

\begin{minipage}{\textwidth}
\plist{O_{12(6-2)+1}[3]}

%
\quad
\includegraphics[width=0.3\textwidth]{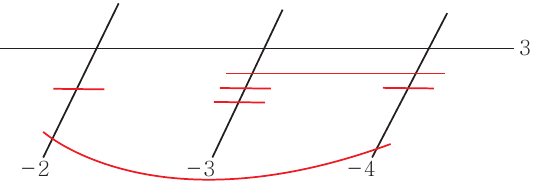}\quad
$\dim= 5$, $\mu= 4$, BO \#39
\end{minipage}

\begin{minipage}{\textwidth}
\plist{O_{12(7-2)+1}[1]}

%
\quad
\includegraphics[width=0.3\textwidth]{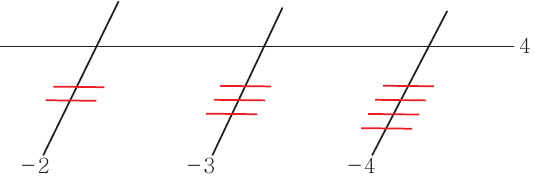}\quad
$\dim= 12$, $\mu= 7$, BO \#40
\end{minipage}

\begin{minipage}{\textwidth}
\plist{O_{12(7-2)+1}[2]}

%
\quad
\includegraphics[width=0.3\textwidth]{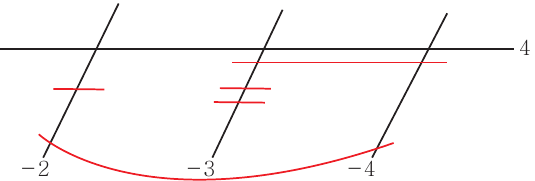}\quad
$\dim= 4$, $\mu= 3$, BO \#41
\end{minipage}

\begin{minipage}{\textwidth}
\plist{O_{12(n-2)+1}[1], n>7}

%
\quad
\includegraphics[width=0.3\textwidth]{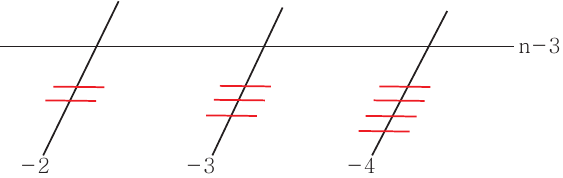}\quad
$\dim= n+5$, $\mu= 7$, BO \#42
\end{minipage}

\subsection*{Octahedral $O_{12(n-2)+5}$}\hfill

\begin{minipage}{\textwidth}
\plist{O_{12(2-2)+5}[1]}

%
\quad
\includegraphics[width=0.3\textwidth]{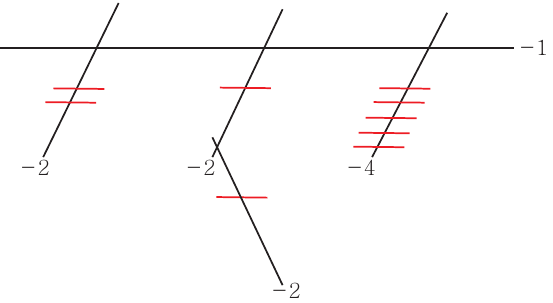}\quad
$\dim= 7$, $\mu= 6$, BO \#137
\end{minipage}

\begin{minipage}{\textwidth}
\plist{O_{12(3-2)+5}[1]}

%
\quad
\includegraphics[width=0.3\textwidth]{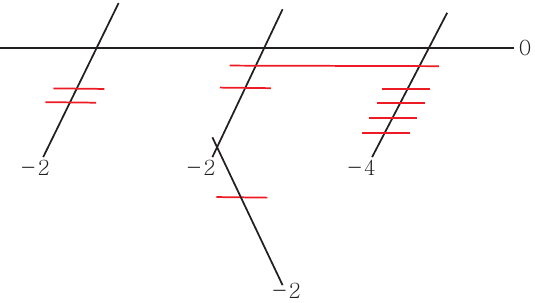}\quad
$\dim= 8$, $\mu= 6$, BO \#138
\end{minipage}

\begin{minipage}{\textwidth}
\plist{O_{12(3-2)+5}[2]}

%
\quad
\includegraphics[width=0.3\textwidth]{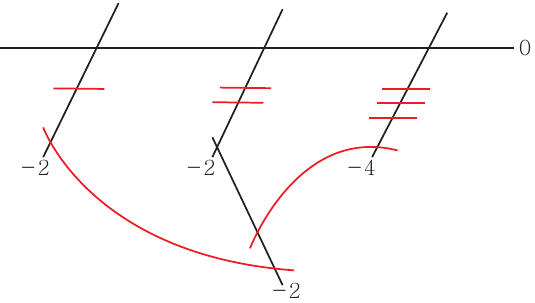}\quad
$\dim= 6$, $\mu= 5$, BO \#215
\end{minipage}

\begin{minipage}{\textwidth}
\plist{O_{12(4-2)+5}[1]}

%
\quad
\includegraphics[width=0.3\textwidth]{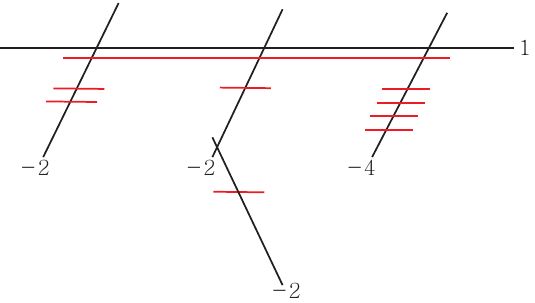}\quad
$\dim= 9$, $\mu= 6$, BO \#139
\end{minipage}

\begin{minipage}{\textwidth}
\plist{O_{12(4-2)+5}[2]}

%
\quad
\includegraphics[width=0.3\textwidth]{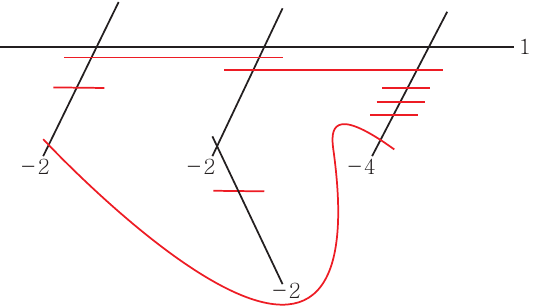}\quad
$\dim= 7$, $\mu= 5$, BO \#140
\end{minipage}

\begin{minipage}{\textwidth}
\plist{O_{12(4-2)+5}[3]}

%
\quad
\includegraphics[width=0.3\textwidth]{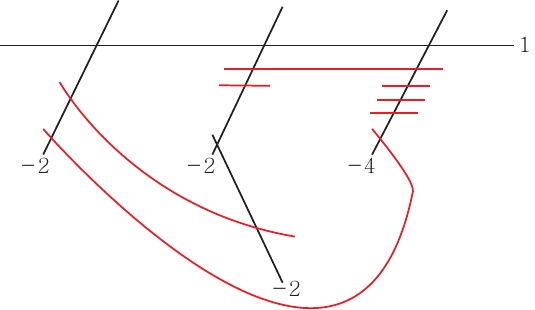}\quad
$\dim= 5$, $\mu= 4$, BO \#216
\end{minipage}

\begin{minipage}{\textwidth}
\plist{O_{12(4-2)+5}[4]}

%
\quad
\includegraphics[width=0.3\textwidth]{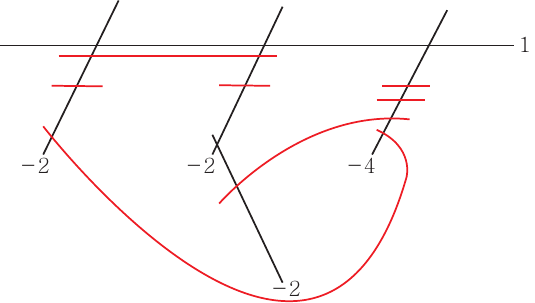}\quad
$\dim= 5$, $\mu= 4$, BO \#217
\end{minipage}

\begin{minipage}{\textwidth}
\plist{O_{12(5-2)+5}[1]}

%
\quad
\includegraphics[width=0.3\textwidth]{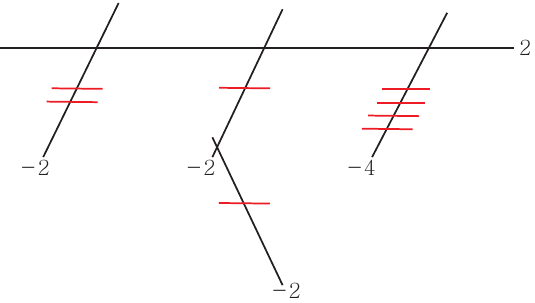}\quad
$\dim= 10$, $\mu= 6$, BO \#141
\end{minipage}

\begin{minipage}{\textwidth}
\plist{O_{12(5-2)+5}[2]}

%
\quad
\includegraphics[width=0.3\textwidth]{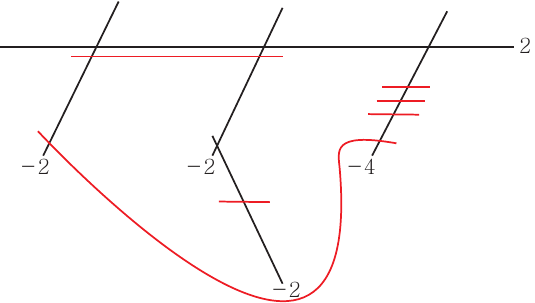}\quad
$\dim= 6$, $\mu= 4$, BO \#142
\end{minipage}

\begin{minipage}{\textwidth}
\plist{O_{12(5-2)+5}[3]}

%
\quad
\includegraphics[width=0.3\textwidth]{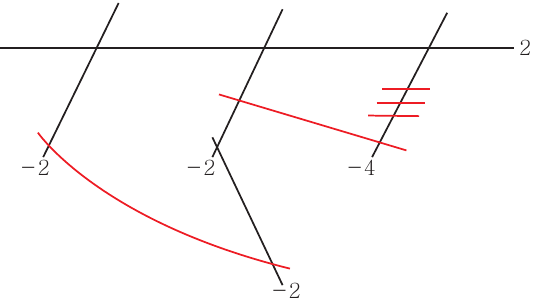}\quad
$\dim= 4$, $\mu= 3$, BO \#218
\end{minipage}

\begin{minipage}{\textwidth}
\plist{O_{12(5-2)+5}[4]}

%
\quad
\includegraphics[width=0.3\textwidth]{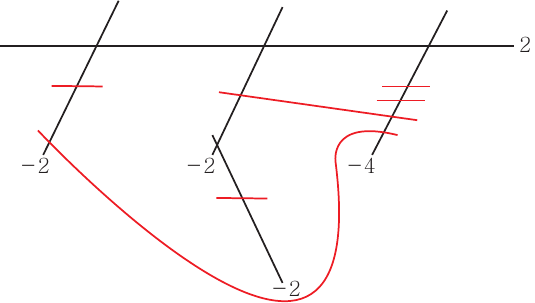}\quad
$\dim= 6$, $\mu= 4$, BO \#143
\end{minipage}

\begin{minipage}{\textwidth}
\plist{O_{12(5-2)+5}[5]}

%
\quad
\includegraphics[width=0.3\textwidth]{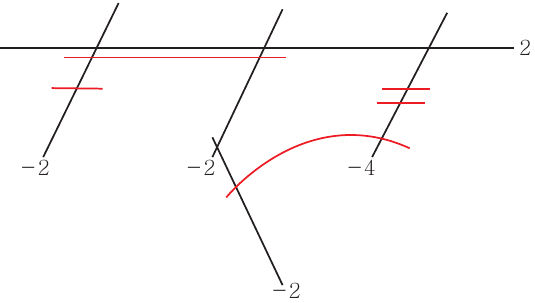}\quad
$\dim= 4$, $\mu= 3$, BO \#219
\end{minipage}

\begin{minipage}{\textwidth}
\plist{O_{12(5-2)+5}[6]}

%
\quad
\includegraphics[width=0.3\textwidth]{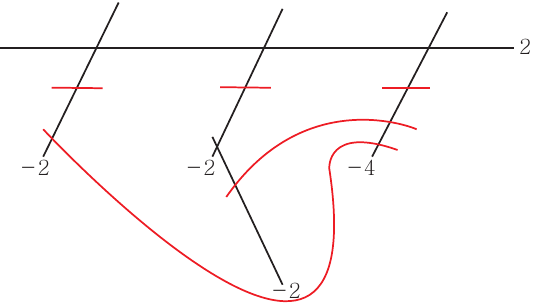}\quad
$\dim= 4$, $\mu= 3$, BO \#220
\end{minipage}

\begin{minipage}{\textwidth}
\plist{O_{12(6-2)+5}[1]}

%
\quad
\includegraphics[width=0.3\textwidth]{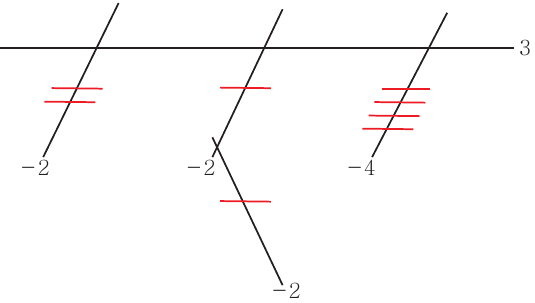}\quad
$\dim= 11$, $\mu= 6$, BO \#144
\end{minipage}

\begin{minipage}{\textwidth}
\plist{O_{12(6-2)+5}[2]}

%
\quad
\includegraphics[width=0.3\textwidth]{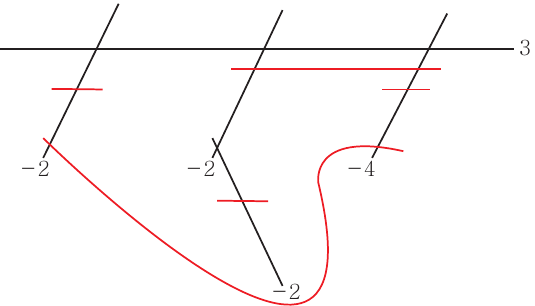}\quad
$\dim= 5$, $\mu= 3$, BO \#145
\end{minipage}

\begin{minipage}{\textwidth}
\plist{O_{12(6-2)+5}[3]}

%
\quad
\includegraphics[width=0.3\textwidth]{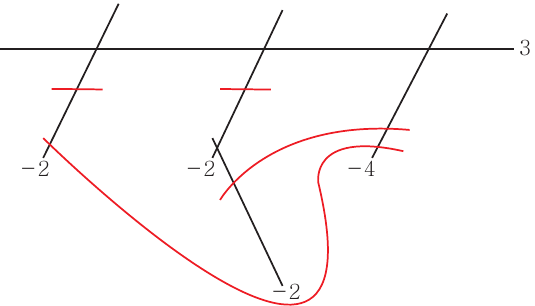}\quad
$\dim= 3$, $\mu= 2$, BO \#221
\end{minipage}

\begin{minipage}{\textwidth}
\plist{O_{12(7-2)+5}[1]}

%
\quad
\includegraphics[width=0.3\textwidth]{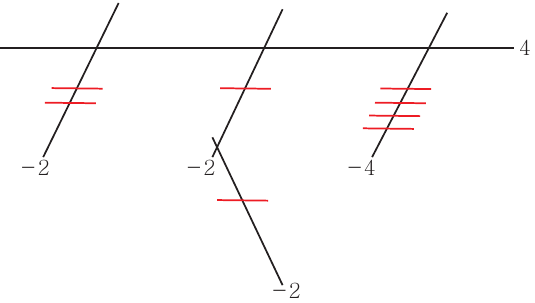}\quad
$\dim= 12$, $\mu= 6$, BO \#146
\end{minipage}

\begin{minipage}{\textwidth}
\plist{O_{12(7-2)+5}[2]}

%
\quad
\includegraphics[width=0.3\textwidth]{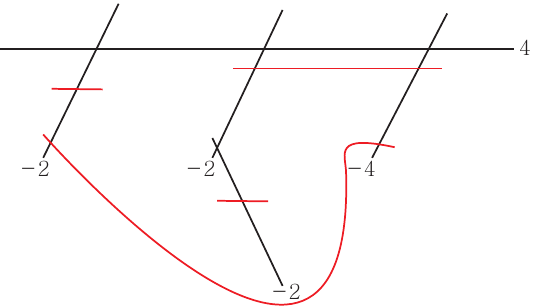}\quad
$\dim= 4$, $\mu= 2$, BO \#147
\end{minipage}

\begin{minipage}{\textwidth}
\plist{O_{12(n-2)+5}[1], n>7}

%
\quad
\includegraphics[width=0.3\textwidth]{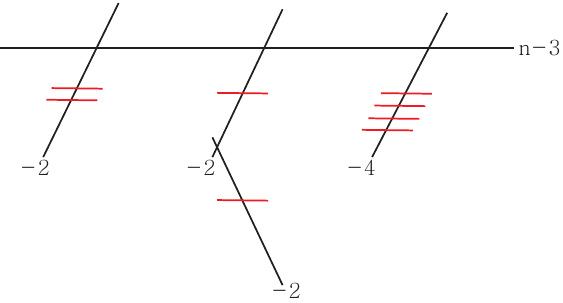}\quad
$\dim= n+5$, $\mu= 6$, BO \#148
\end{minipage}

\subsection*{Octahedral $O_{12(n-2)+7}$}\hfill

\begin{minipage}{\textwidth}
\plist{O_{12(2-2)+7}[1]}

%
\quad
\includegraphics[width=0.3\textwidth]{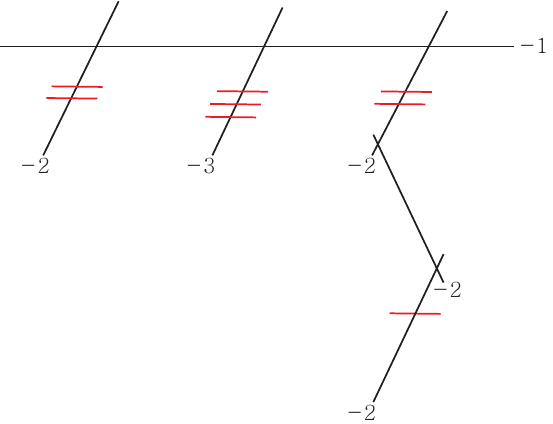}\quad
$\dim= 7$, $\mu= 5$, BO \#43
\end{minipage}

\begin{minipage}{\textwidth}
\plist{O_{12(2-2)+7}[2]}

%
\quad
\includegraphics[width=0.3\textwidth]{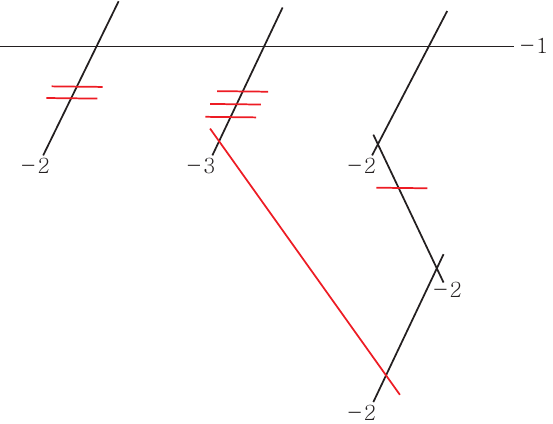}\quad
$\dim= 5$, $\mu= 4$, BO \#44
\end{minipage}

\begin{minipage}{\textwidth}
\plist{O_{12(3-2)+7}[1]}

%
\quad
\includegraphics[width=0.3\textwidth]{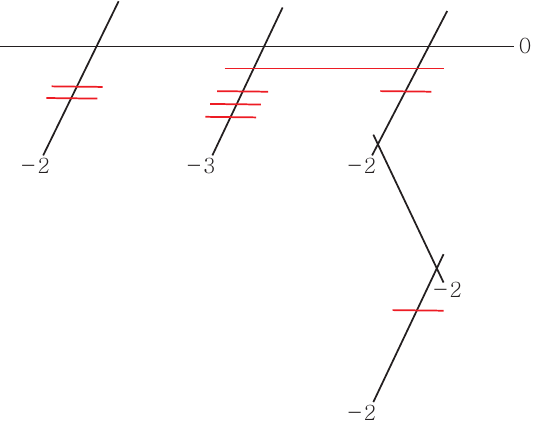}\quad
$\dim= 8$, $\mu= 5$, BO \#45
\end{minipage}

\begin{minipage}{\textwidth}
\plist{O_{12(3-2)+7}[2]}

%
\quad
\includegraphics[width=0.3\textwidth]{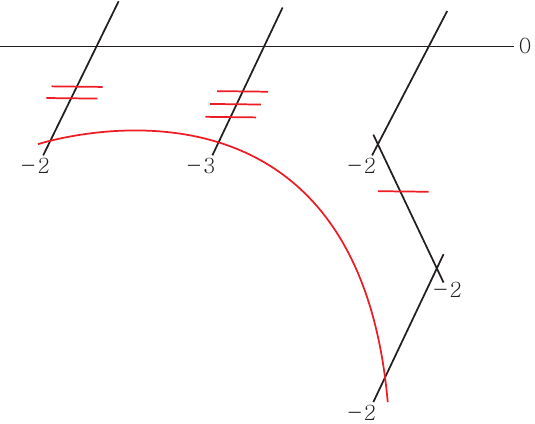}\quad
$\dim= 6$, $\mu= 4$, BO \#48
\end{minipage}

\begin{minipage}{\textwidth}
\plist{O_{12(3-2)+7}[3]}

%
\quad
\includegraphics[width=0.3\textwidth]{O-7-3-3}\quad
$\dim= 4$, $\mu= 3$, BO \#46
\end{minipage}

\begin{minipage}{\textwidth}
\plist{O_{12(3-2)+7}[4]}

%
\quad
\includegraphics[width=0.3\textwidth]{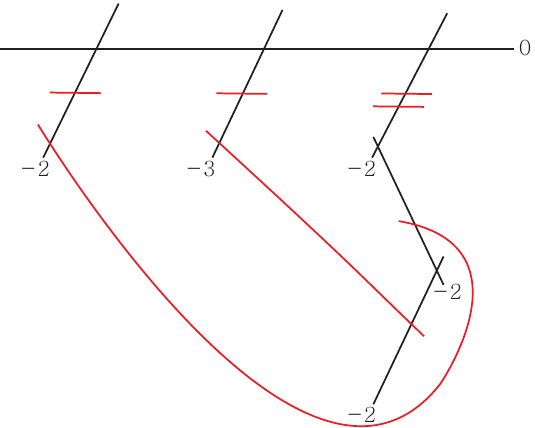}\quad
$\dim= 4$, $\mu= 3$, BO \#47
\end{minipage}

\begin{minipage}{\textwidth}
\plist{O_{12(4-2)+7}[1]}

%
\quad
\includegraphics[width=0.3\textwidth]{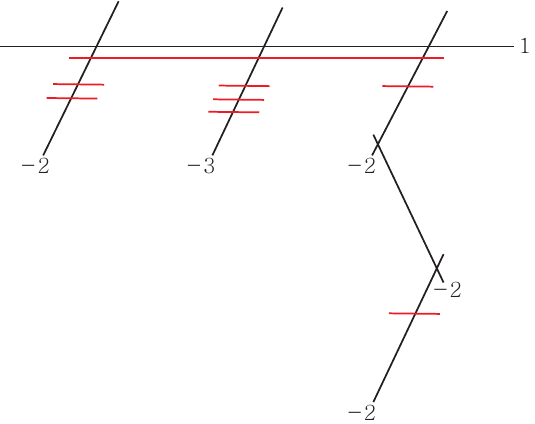}\quad
$\dim= 9$, $\mu= 5$, BO \#49
\end{minipage}

\begin{minipage}{\textwidth}
\plist{O_{12(4-2)+7}[2]}

%
\quad
\includegraphics[width=0.3\textwidth]{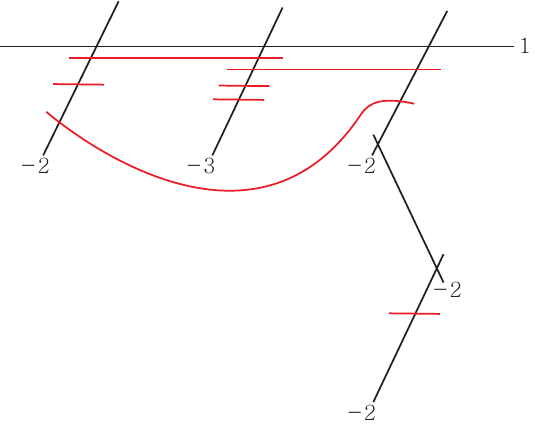}\quad
$\dim= 7$, $\mu= 4$, BO \#51
\end{minipage}

\begin{minipage}{\textwidth}
\plist{O_{12(4-2)+7}[3]}

%
\quad
\includegraphics[width=0.3\textwidth]{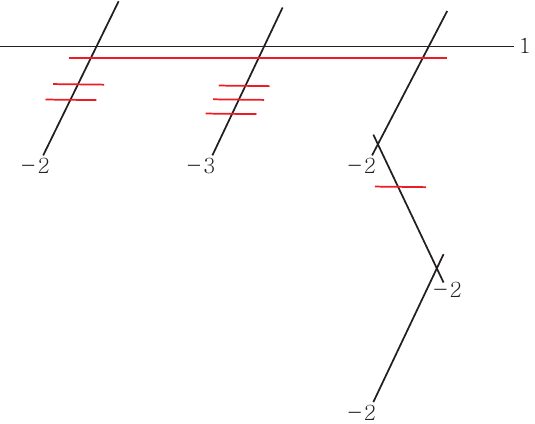}\quad
$\dim= 7$, $\mu= 5$, BO \#52
\end{minipage}

\begin{minipage}{\textwidth}
\plist{O_{12(4-2)+7}[4]}

%
\quad
\includegraphics[width=0.3\textwidth]{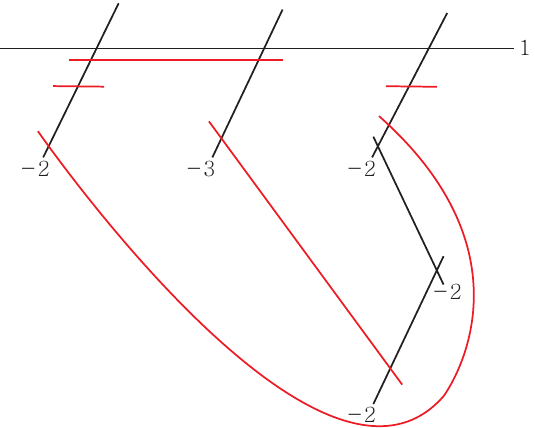}\quad
$\dim= 3$, $\mu= 2$, BO \#50
\end{minipage}

\begin{minipage}{\textwidth}
\plist{O_{12(5-2)+7}[1]}

%
\quad
\includegraphics[width=0.3\textwidth]{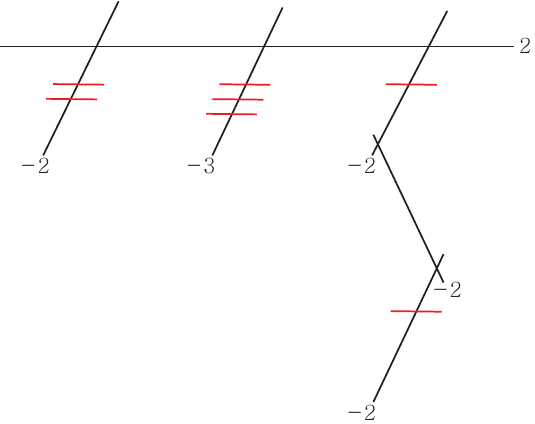}\quad
$\dim= 10$, $\mu= 5$, BO \#53
\end{minipage}

\begin{minipage}{\textwidth}
\plist{O_{12(5-2)+7}[2]}

%
\quad
\includegraphics[width=0.3\textwidth]{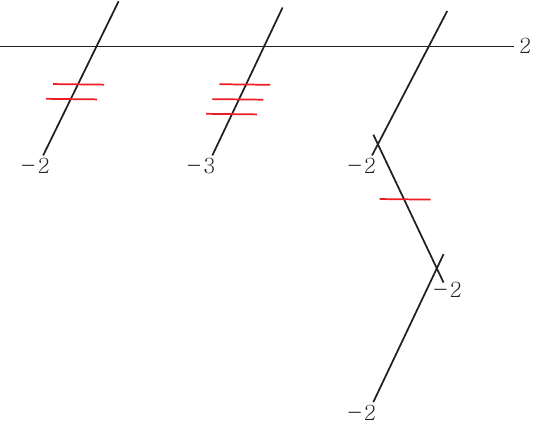}\quad
$\dim= 8$, $\mu= 4$, BO \#56
\end{minipage}

\begin{minipage}{\textwidth}
\plist{O_{12(5-2)+7}[3]}

%
\quad
\includegraphics[width=0.3\textwidth]{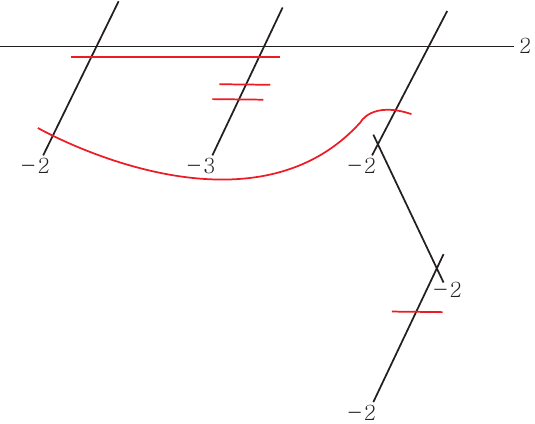}\quad
$\dim= 6$, $\mu= 3$, BO \#54
\end{minipage}

\begin{minipage}{\textwidth}
\plist{O_{12(5-2)+7}[4]}

%
\quad
\includegraphics[width=0.3\textwidth]{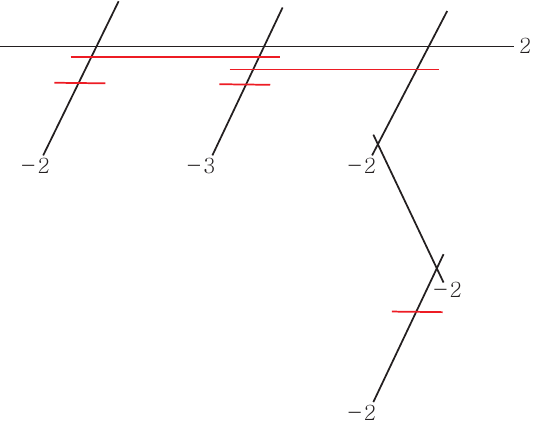}\quad
$\dim= 6$, $\mu= 3$, BO \#55
\end{minipage}

\begin{minipage}{\textwidth}
\plist{O_{12(5-2)+7}[5]}

%
\quad
\includegraphics[width=0.3\textwidth]{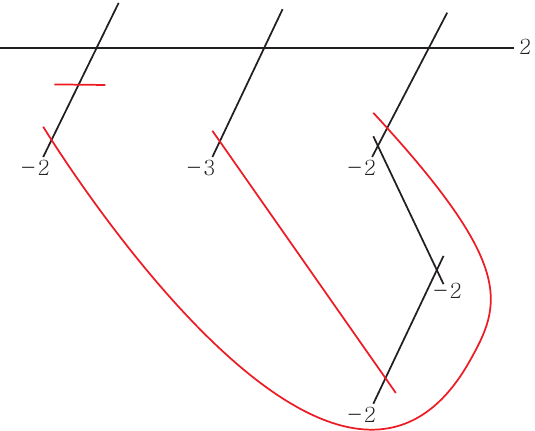}\quad
$\dim= 2$, $\mu= 1$, BO \#57
\end{minipage}

\begin{minipage}{\textwidth}
\plist{O_{12(6-2)+7}[1]}

%
\quad
\includegraphics[width=0.3\textwidth]{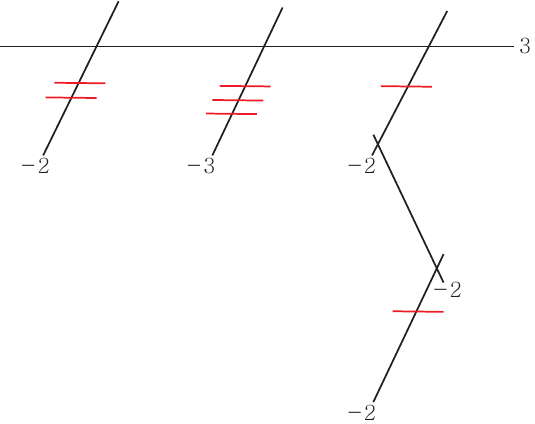}\quad
$\dim= 11$, $\mu= 5$, BO \#58
\end{minipage}

\begin{minipage}{\textwidth}
\plist{O_{12(6-2)+7}[2]}

%
\quad
\includegraphics[width=0.3\textwidth]{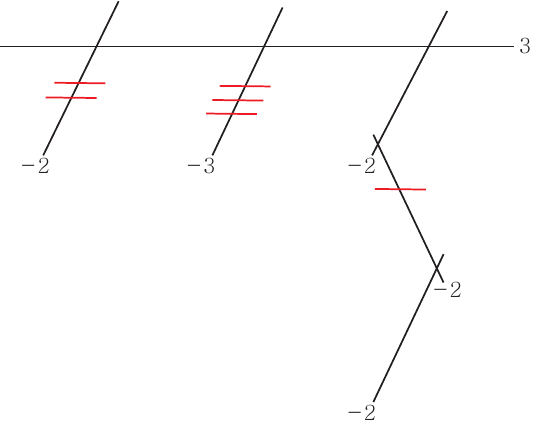}\quad
$\dim= 9$, $\mu= 4$, BO \#60
\end{minipage}

\begin{minipage}{\textwidth}
\plist{O_{12(6-2)+7}[3]}

%
\quad
\includegraphics[width=0.3\textwidth]{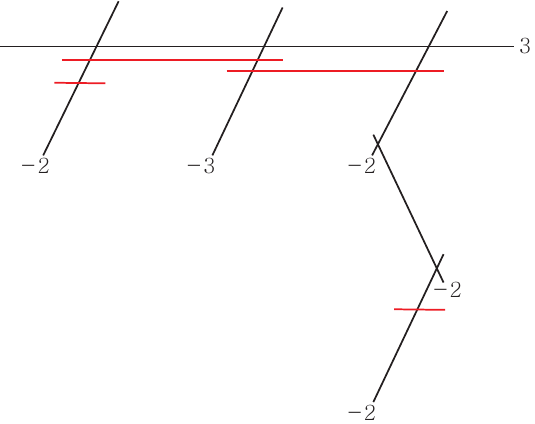}\quad
$\dim= 5$, $\mu= 2$, BO \#59
\end{minipage}

\begin{minipage}{\textwidth}
\plist{O_{12(n-2)+7}[1], n>6}

%
\quad
\includegraphics[width=0.3\textwidth]{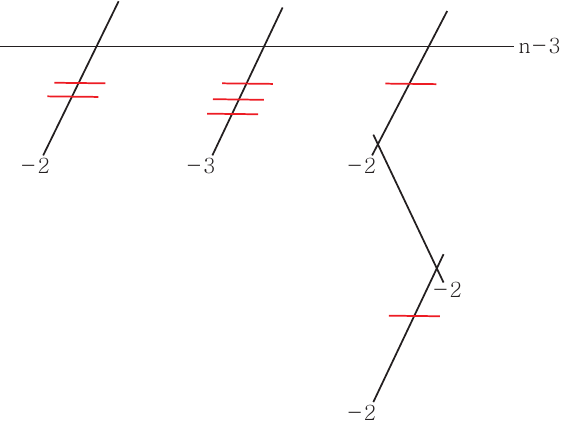}\quad
$\dim= n+5$, $\mu= 5$, BO \#61
\end{minipage}

\begin{minipage}{\textwidth}
\plist{O_{12(n-2)+7}[2], n>6}

%
\quad
\includegraphics[width=0.3\textwidth]{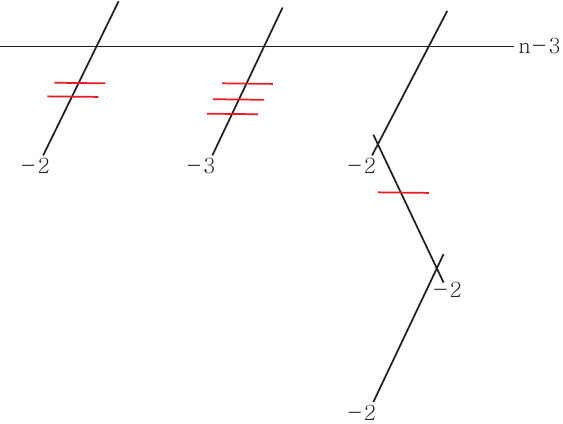}\quad
$\dim= n+3$, $\mu= 4$, BO \#62
\end{minipage}

\subsection*{Octahedral $O_{12(n-2)+11}$}\hfill

\begin{minipage}{\textwidth}
\plist{O_{12(2-2)+11}[1]}

%
\quad
\includegraphics[width=0.3\textwidth]{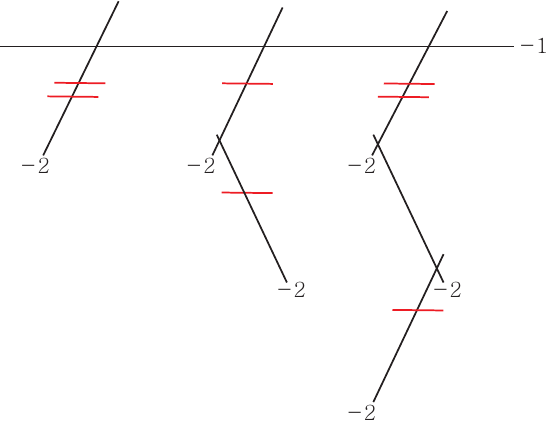}\quad
$\dim= 7$, $\mu= 4$, BO \#149
\end{minipage}

\begin{minipage}{\textwidth}
\plist{O_{12(2-2)+11},[2]}

%
\quad
\includegraphics[width=0.3\textwidth]{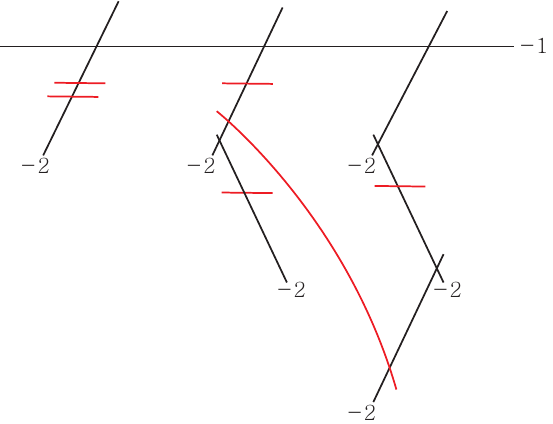}\quad
$\dim= 5$, $\mu= 3$, BO \#150
\end{minipage}

\begin{minipage}{\textwidth}
\plist{O_{12(2-2)+11}[3]}

%
\quad
\includegraphics[width=0.3\textwidth]{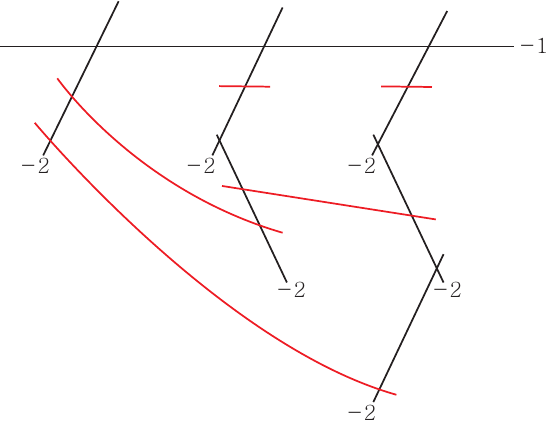}\quad
$\dim= 3$, $\mu= 2$, BO \#222
\end{minipage}

\begin{minipage}{\textwidth}
\plist{O_{12(3-2)+11}[1]}

%
\quad
\includegraphics[width=0.3\textwidth]{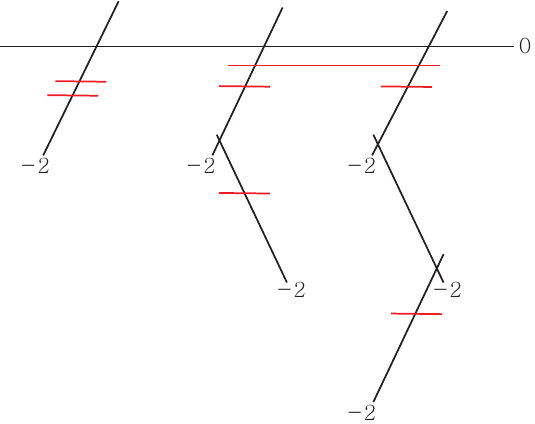}\quad
$\dim= 8$, $\mu= 4$, BO \#151
\end{minipage}

\begin{minipage}{\textwidth}
\plist{O_{12(3-2)+11}[2]}

%
\quad
\includegraphics[width=0.3\textwidth]{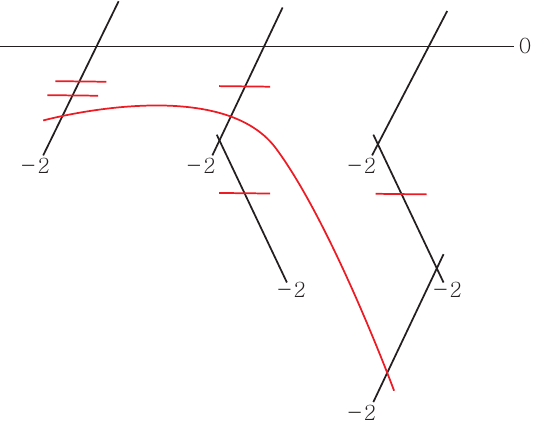}\quad
$\dim= 6$, $\mu= 3$, BO \#153
\end{minipage}

\begin{minipage}{\textwidth}
\plist{O_{12(3-2)+11}[3]}

%
\quad
\includegraphics[width=0.3\textwidth]{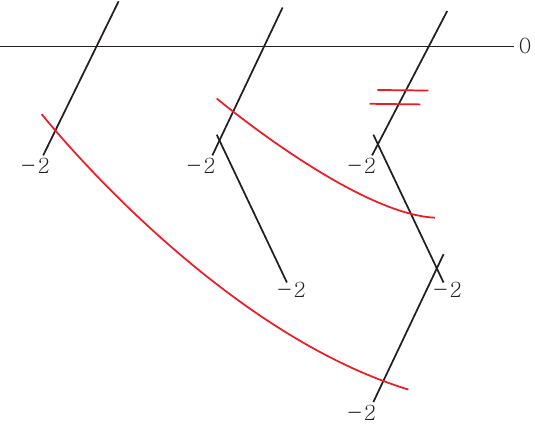}\quad
$\dim= 4$, $\mu= 2$, BO \#152
\end{minipage}

\begin{minipage}{\textwidth}
\plist{O_{12(3-2)+11}[4]}

%
\quad
\includegraphics[width=0.3\textwidth]{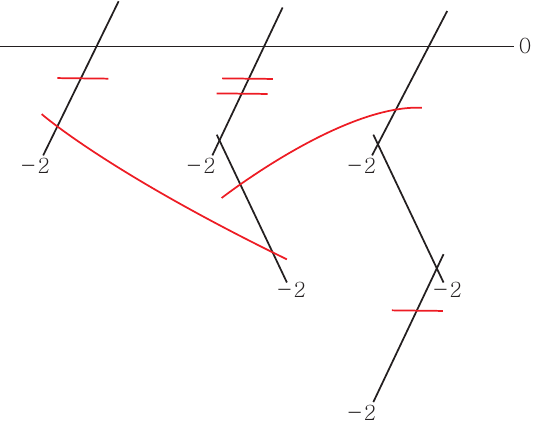}\quad
$\dim= 6$, $\mu= 3$, BO \#223
\end{minipage}

\begin{minipage}{\textwidth}
\plist{O_{12(3-2)+11}[5]}

%
\quad
\includegraphics[width=0.3\textwidth]{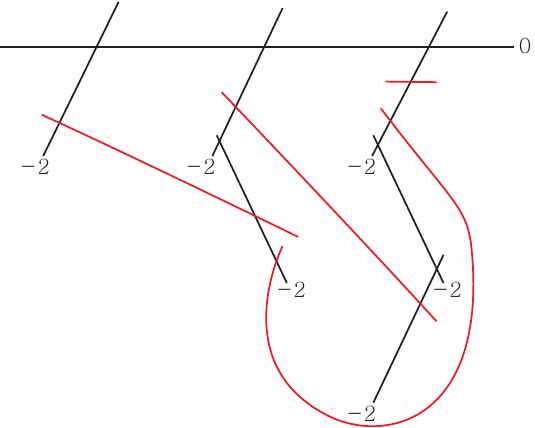}\quad
$\dim= 2$, $\mu= 1$, BO \#224
\end{minipage}

\begin{minipage}{\textwidth}
\plist{O_{12(4-2)+11}[1]}

%
\quad
\includegraphics[width=0.3\textwidth]{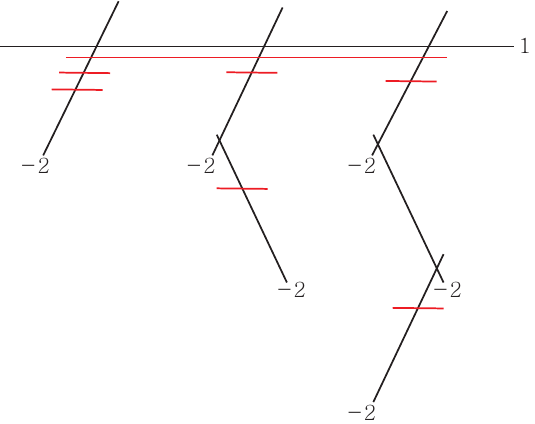}\quad
$\dim= 9$, $\mu= 4$, BO \#154
\end{minipage}

\begin{minipage}{\textwidth}
\plist{O_{12(4-2)+11}[2]}

%
\quad
\includegraphics[width=0.3\textwidth]{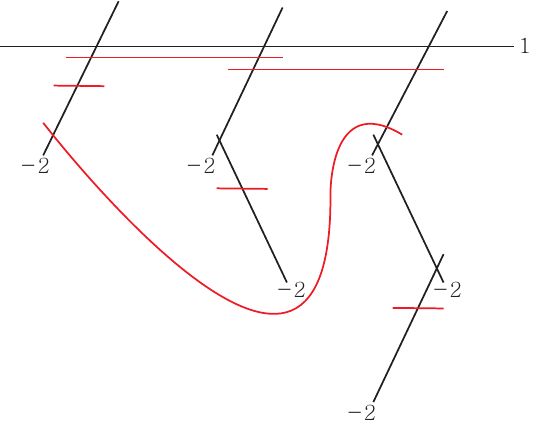}\quad
$\dim= 7$, $\mu= 3$, BO \#225
\end{minipage}

\begin{minipage}{\textwidth}
\plist{O_{12(4-2)+11}[3]}

%
\quad
\includegraphics[width=0.3\textwidth]{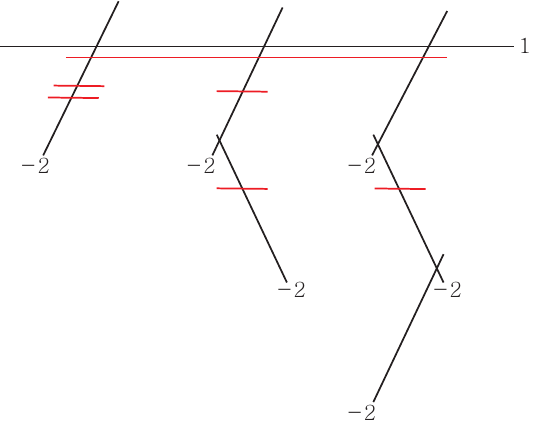}\quad
$\dim= 7$, $\mu= 3$, BO \#156
\end{minipage}

\begin{minipage}{\textwidth}
\plist{O_{12(4-2)+11}[4]}

%
\quad
\includegraphics[width=0.3\textwidth]{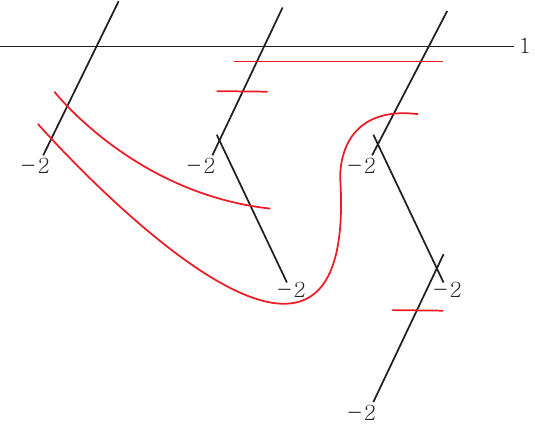}\quad
$\dim= 5$, $\mu= 2$, BO \#155
\end{minipage}

\begin{minipage}{\textwidth}
\plist{O_{12(5-2)+11}[1]}

%
\quad
\includegraphics[width=0.3\textwidth]{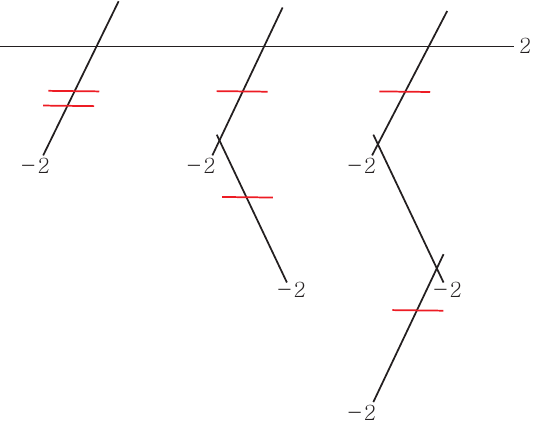}\quad
$\dim= 10$, $\mu= 4$, BO \#157
\end{minipage}

\begin{minipage}{\textwidth}
\plist{O_{12(5-2)+11}[2]}

%
\quad
\includegraphics[width=0.3\textwidth]{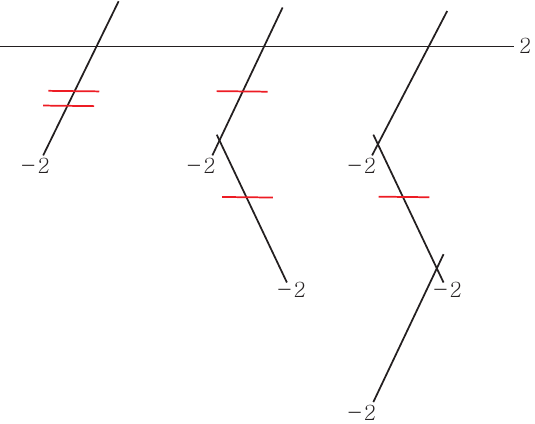}\quad
$\dim= 8$, $\mu= 3$, BO \#159
\end{minipage}

\begin{minipage}{\textwidth}
\plist{O_{12(5-2)+11}[3]}

%
\quad
\includegraphics[width=0.3\textwidth]{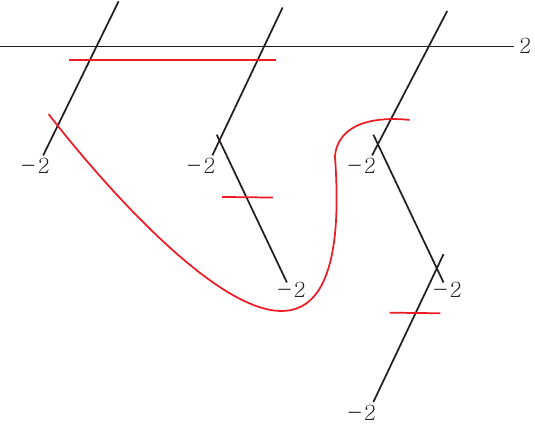}\quad
$\dim= 6$, $\mu= 2$, BO \#158
\end{minipage}

\begin{minipage}{\textwidth}
\plist{O_{12(5-2)+11}[4]}

%
\quad
\includegraphics[width=0.3\textwidth]{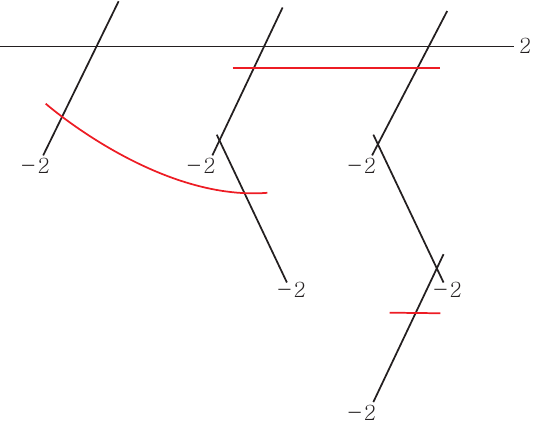}\quad
$\dim= 4$, $\mu= 1$, BO \#226
\end{minipage}

\begin{minipage}{\textwidth}
\plist{O_{12(n-2)+11}[1], n>5}

%
\quad
\includegraphics[width=0.3\textwidth]{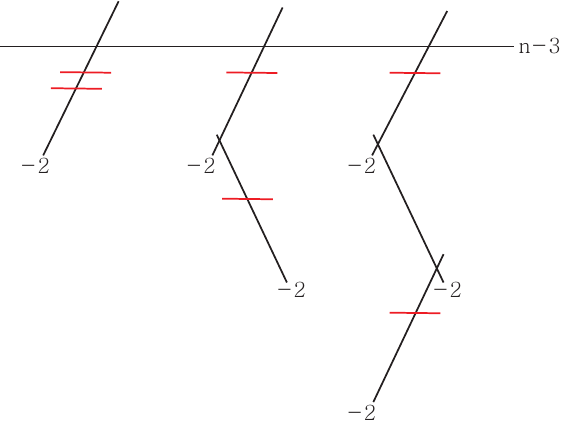}\quad
$\dim= n+5$, $\mu= 4$, BO \#160
\end{minipage}

\begin{minipage}{\textwidth}
\plist{O_{12(n-2)+11}[2], n>5}

%
\quad
\includegraphics[width=0.3\textwidth]{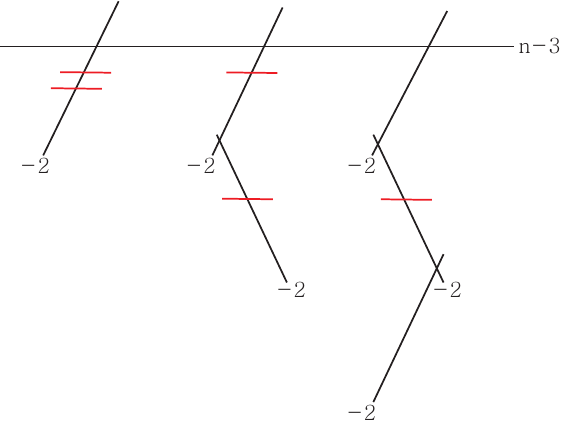}\quad
$\dim= n+3$,, $\mu= 3$, BO \#161
\end{minipage}

\subsection*{Icosahedral $I_{30(n-2)+1}$}\hfill

\begin{minipage}{\textwidth}
\plist{E_8}

%
\quad
\includegraphics[width=0.3\textwidth]{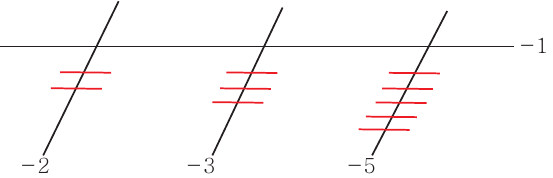}\quad
$\dim= 8$, $\mu= 8$, BO \#63
\end{minipage}

\begin{minipage}{\textwidth}
\plist{I_{30(3-2)+1}[1]}

%
\quad
\includegraphics[width=0.3\textwidth]{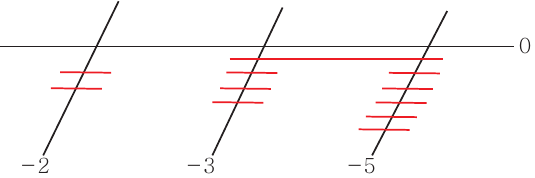}\quad
$\dim= 9$, $\mu= 8$, BO \#64
\end{minipage}

\begin{minipage}{\textwidth}
\plist{I_{30(4-2)+1}[1]}

%
\quad
\includegraphics[width=0.3\textwidth]{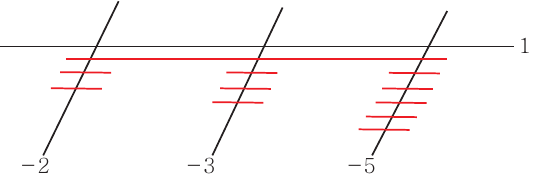}\quad
$\dim= 10$, $\mu= 8$, BO \#65
\end{minipage}

\begin{minipage}{\textwidth}
\plist{I_{30(4-2)+1}[2]}

%
\quad
\includegraphics[width=0.3\textwidth]{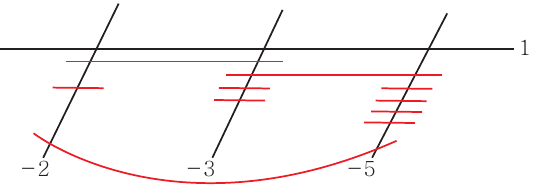}\quad
$\dim= 8$, $\mu= 7$, BO \#66
\end{minipage}

\begin{minipage}{\textwidth}
\plist{I_{30(5-2)+1}[1]}

%
\quad
\includegraphics[width=0.3\textwidth]{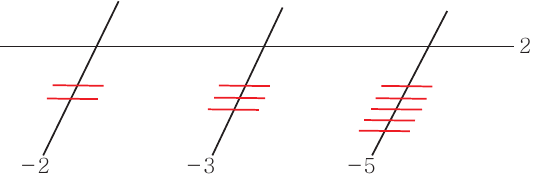}\quad
$\dim= 11$, $\mu= 8$, BO \#67
\end{minipage}

\begin{minipage}{\textwidth}
\plist{I_{30(5-2)+1}[2]}

%
\quad
\includegraphics[width=0.3\textwidth]{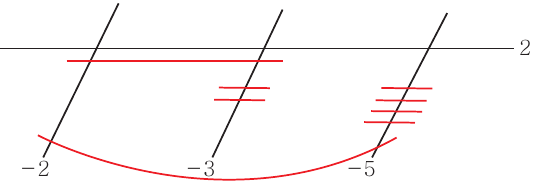}\quad
$\dim= 7$, $\mu= 6$, BO \#68
\end{minipage}

\begin{minipage}{\textwidth}
\plist{I_{30(5-2)+1}[3]}

%
\quad
\includegraphics[width=0.3\textwidth]{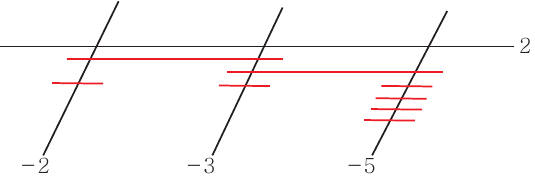}\quad
$\dim= 7$, $\mu= 6$, BO \#69
\end{minipage}

\begin{minipage}{\textwidth}
\plist{I_{30(5-2)+1}[4]}

%
\quad
\includegraphics[width=0.3\textwidth]{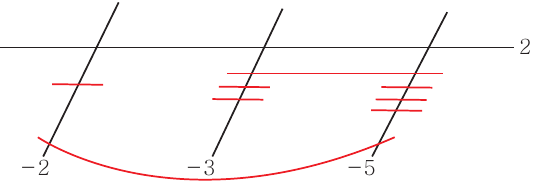}\quad
$\dim= 7$, $\mu= 6$, BO \#70
\end{minipage}

\begin{minipage}{\textwidth}
\plist{I_{30(6-2)+1}[1]}

%
\quad
\includegraphics[width=0.3\textwidth]{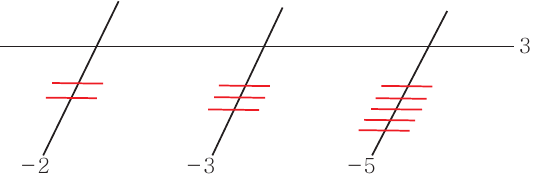}\quad
$\dim= 12$, $\mu= 8$, BO \#71
\end{minipage}

\begin{minipage}{\textwidth}
\plist{I_{30(6-2)+1}[2]}

%
\quad
\includegraphics[width=0.3\textwidth]{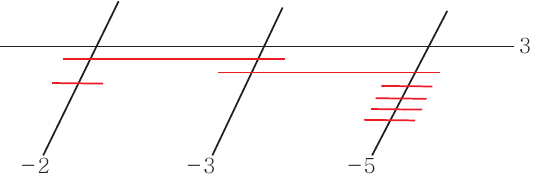}\quad
$\dim= 6$, $\mu= 5$, BO \#72
\end{minipage}

\begin{minipage}{\textwidth}
\plist{I_{30(6-2)+1}[3]}

%
\quad
\includegraphics[width=0.3\textwidth]{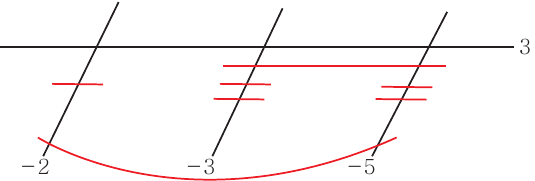}\quad
$\dim= 6$, $\mu= 5$, BO \#73
\end{minipage}

\begin{minipage}{\textwidth}
\plist{I_{30(7-2)+1}[1]}

%
\quad
\includegraphics[width=0.3\textwidth]{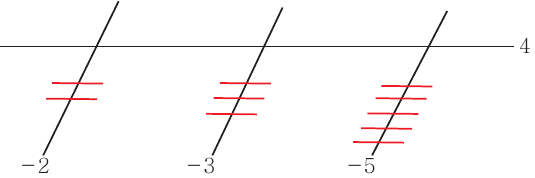}\quad
$\dim= 13$, $\mu= 8$, BO \#74
\end{minipage}

\begin{minipage}{\textwidth}
\plist{I_{30(7-2)+1}[2]}

%
\quad
\includegraphics[width=0.3\textwidth]{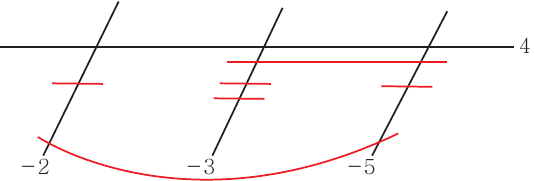}\quad
$\dim= 5$, $\mu= 4$, BO \#75
\end{minipage}

\begin{minipage}{\textwidth}
\plist{I_{30(8-2)+1}[1]}

%
\quad
\includegraphics[width=0.3\textwidth]{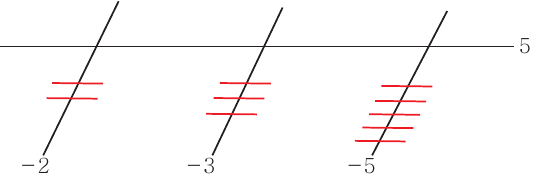}\quad
$\dim= 14$, $\mu= 8$, BO \#76
\end{minipage}

\begin{minipage}{\textwidth}
\plist{I_{30(8-2)+1}[2]}

%
\quad
\includegraphics[width=0.3\textwidth]{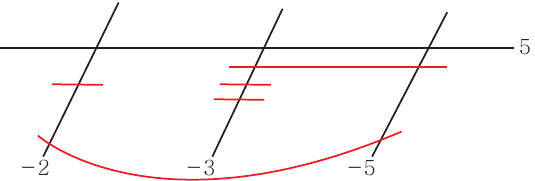}\quad
$\dim= 4$, $\mu= 3$, BO \#77
\end{minipage}

\begin{minipage}{\textwidth}
\plist{I_{30(n-2)+1}[1], n>8}

%
\quad
\includegraphics[width=0.3\textwidth]{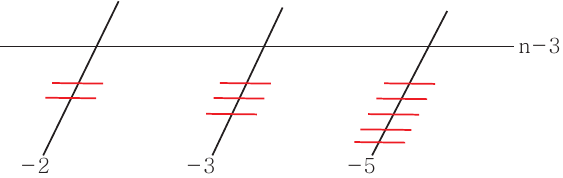}\quad
$\dim= n+6$,$\mu= 8$, BO \#78
\end{minipage}

\subsection*{Icosahedral $I_{30(n-2)+7}$}\hfill

\begin{minipage}{\textwidth}
\plist{I_{30(2-2)+7}[1]}

%
\quad
\includegraphics[width=0.3\textwidth]{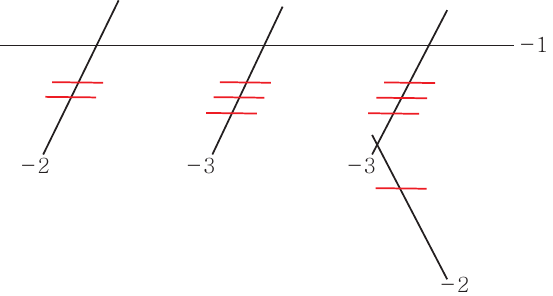}\quad
$\dim= 7$, $\mu= 6$, BO \#79
\end{minipage}

\begin{minipage}{\textwidth}
\plist{I_{30(3-2)+7}[1]}

%
\quad
\includegraphics[width=0.3\textwidth]{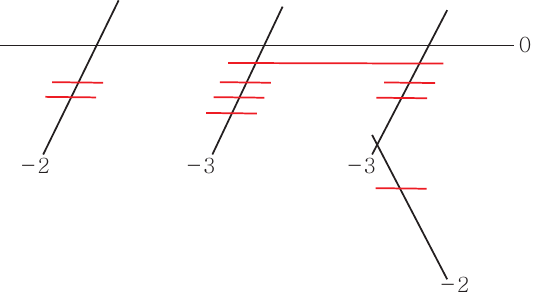}\quad
$\dim= 8$, $\mu= 6$, BO \#80
\end{minipage}

\begin{minipage}{\textwidth}
\plist{I_{30(3-2)+7}[2]}

%
\quad
\includegraphics[width=0.3\textwidth]{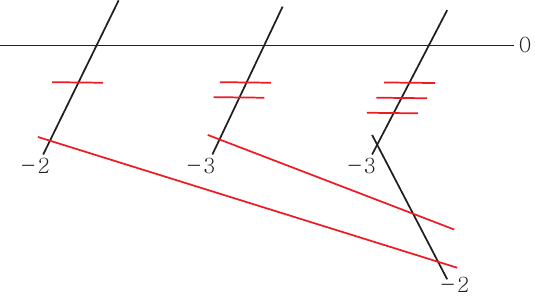}\quad
$\dim= 6$, $\mu= 5$, BO \#81
\end{minipage}

\begin{minipage}{\textwidth}
\plist{I_{30(4-2)+7}[1]}

%
\quad
\includegraphics[width=0.3\textwidth]{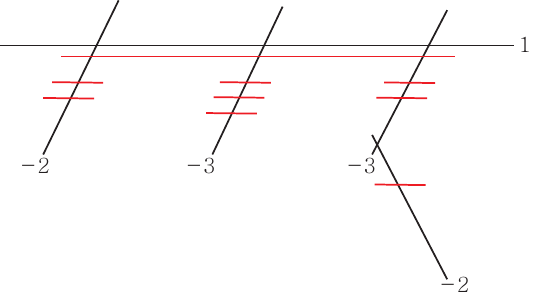}\quad
$\dim= 9$, $\mu= 6$, BO \#82
\end{minipage}

\begin{minipage}{\textwidth}
\plist{I_{30(4-2)+7}[2]}

%
\quad
\includegraphics[width=0.3\textwidth]{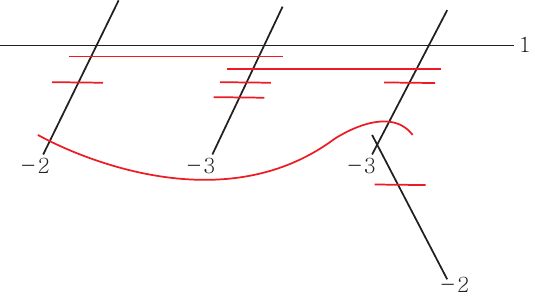}\quad
$\dim= 7$, $\mu= 5$, BO \#85
\end{minipage}

\begin{minipage}{\textwidth}
\plist{I_{30(4-2)+7}[3]}

%
\quad
\includegraphics[width=0.3\textwidth]{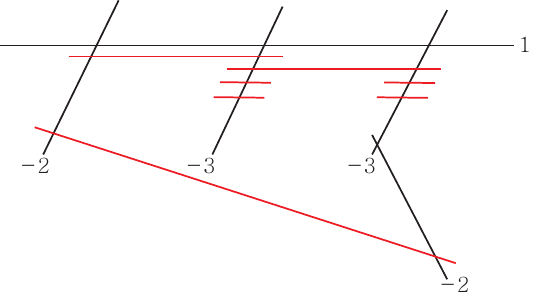}\quad
$\dim= 5$, $\mu= 4$, BO \#83
\end{minipage}

\begin{minipage}{\textwidth}
\plist{I_{30(4-2)+7}[4]}

%
\quad
\includegraphics[width=0.3\textwidth]{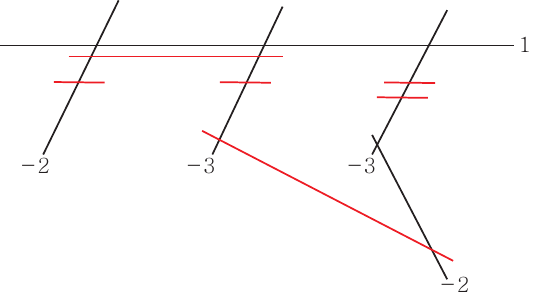}\quad
$\dim= 5$, $\mu= 4$, BO \#84
\end{minipage}

\begin{minipage}{\textwidth}
\plist{I_{30(5-2)+7}[1]}

%
\quad
\includegraphics[width=0.3\textwidth]{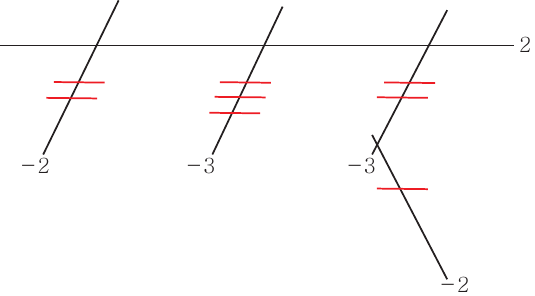}\quad
$\dim= 10$, $\mu= 6$, BO \#86
\end{minipage}

\begin{minipage}{\textwidth}
\plist{I_{30(5-2)+7}[2]}

%
\quad
\includegraphics[width=0.3\textwidth]{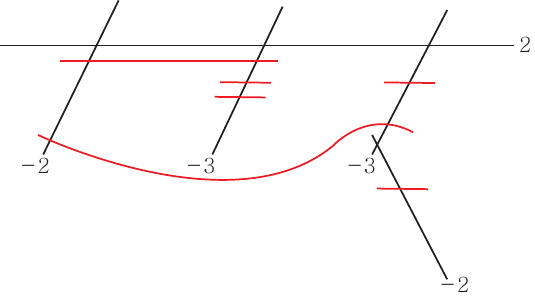}\quad
$\dim= 6$, $\mu= 4$, BO \#91
\end{minipage}

\begin{minipage}{\textwidth}
\plist{I_{30(5-2)+7}[3]}

%
\quad
\includegraphics[width=0.3\textwidth]{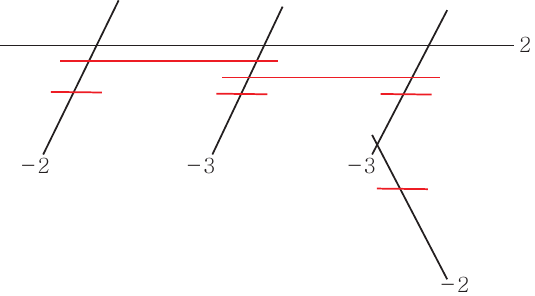}\quad
$\dim= 6$, $\mu= 4$, BO \#92
\end{minipage}

\begin{minipage}{\textwidth}
\plist{I_{30(5-2)+7}[4]}

%
\quad
\includegraphics[width=0.3\textwidth]{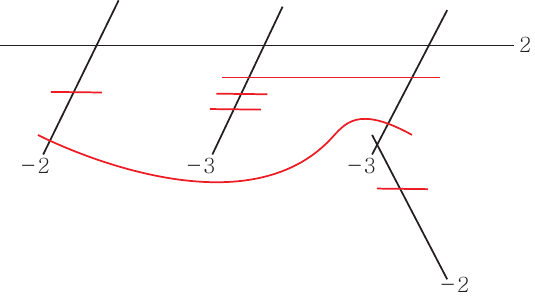}\quad
$\dim= 6$, $\mu= 4$, BO \#90
\end{minipage}

\begin{minipage}{\textwidth}
\plist{I_{30(5-2)+7}[5]}

%
\quad
\includegraphics[width=0.3\textwidth]{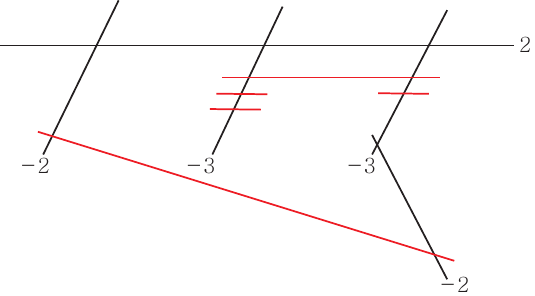}\quad
$\dim= 4$, $\mu= 3$, BO \#88
\end{minipage}

\begin{minipage}{\textwidth}
\plist{I_{30(5-2)+7}[6]}

%
\quad
\includegraphics[width=0.3\textwidth]{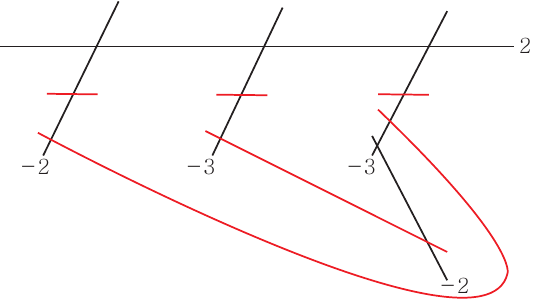}\quad
$\dim= 4$, $\mu= 3$, BO \#89
\end{minipage}

\begin{minipage}{\textwidth}
\plist{I_{30(5-2)+7}[7]}

%
\quad
\includegraphics[width=0.3\textwidth]{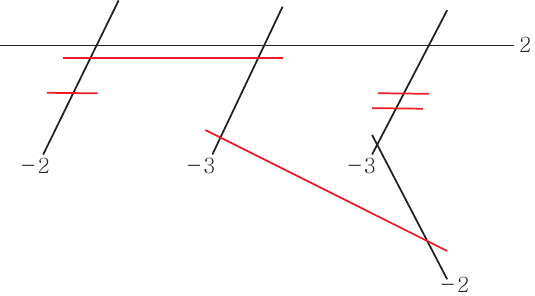}\quad
$\dim= 4$, $\mu= 3$, BO \#87
\end{minipage}

\begin{minipage}{\textwidth}
\plist{I_{30(6-2)+7}[1]}

%
\quad
\includegraphics[width=0.3\textwidth]{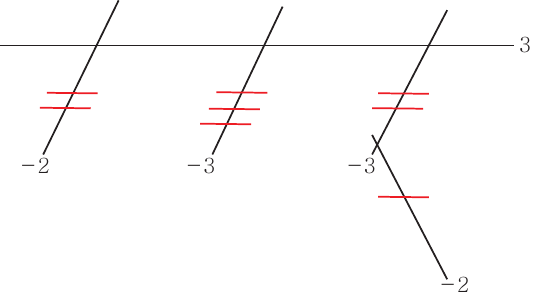}\quad
$\dim= 11$, $\mu= 6$, BO \#93
\end{minipage}

\begin{minipage}{\textwidth}
\plist{I_{30(6-2)+7}[2]}

%
\quad
\includegraphics[width=0.3\textwidth]{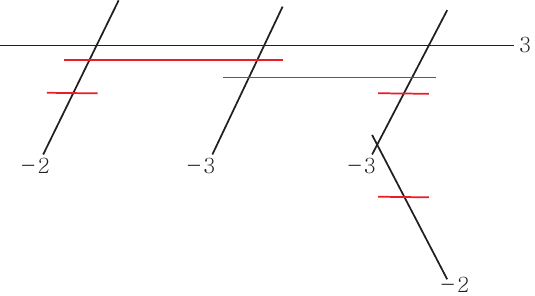}\quad
$\dim= 5$, $\mu= 3$, BO \#94
\end{minipage}

\begin{minipage}{\textwidth}
\plist{I_{30(6-2)+7}[3]}

%
\quad
\includegraphics[width=0.3\textwidth]{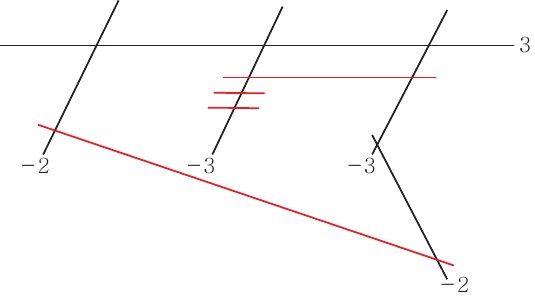}\quad
$\dim= 3$, $\mu= 2$, BO \#95
\end{minipage}

\begin{minipage}{\textwidth}
\plist{I_{30(6-2)+7}[4]}

%
\quad
\includegraphics[width=0.3\textwidth]{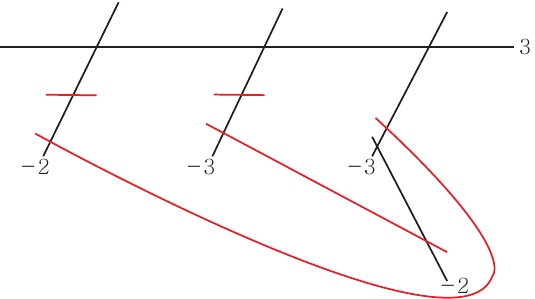}\quad
$\dim= 3$, $\mu= 2$, BO \#96
\end{minipage}

\begin{minipage}{\textwidth}
\plist{I_{30(n-2)+7}[1], n>6}

%
\quad
\includegraphics[width=0.3\textwidth]{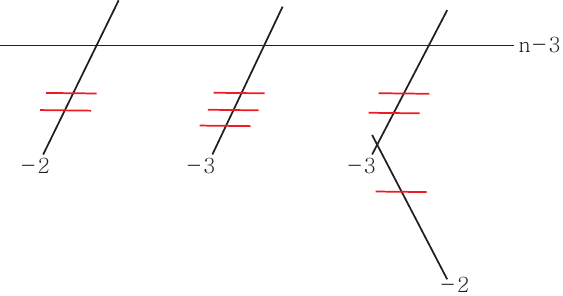}\quad
$\dim= n+5$, n+5 $\mu= 6$, BO \#97
\end{minipage}

\subsection*{Icosahedral $I_{30(n-2)+11}$}\hfill

\begin{minipage}{\textwidth}
\plist{I_{30(2-2)+11}[1]}

%
\quad
\includegraphics[width=0.3\textwidth]{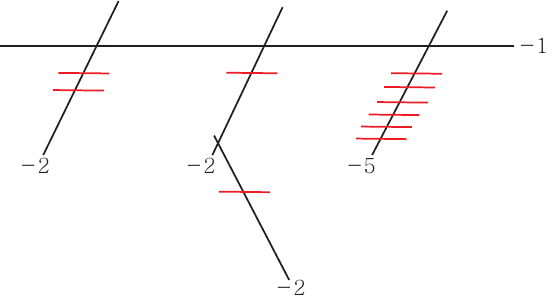}\quad
$\dim= 8$, $\mu= 7$, BO \#162
\end{minipage}

\begin{minipage}{\textwidth}
\plist{I_{30(3-2)+11}[1]}

%
\quad
\includegraphics[width=0.3\textwidth]{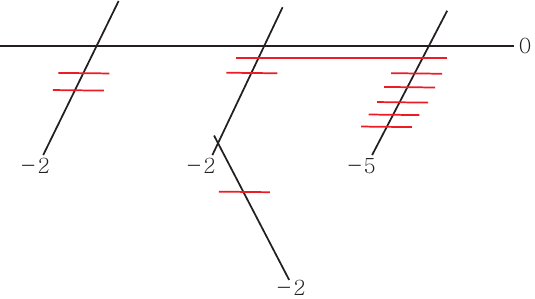}\quad
$\dim= 9$, $\mu= 7$, BO \#163
\end{minipage}

\begin{minipage}{\textwidth}
\plist{I_{30(3-2)+11}[2]}

%
\quad
\includegraphics[width=0.3\textwidth]{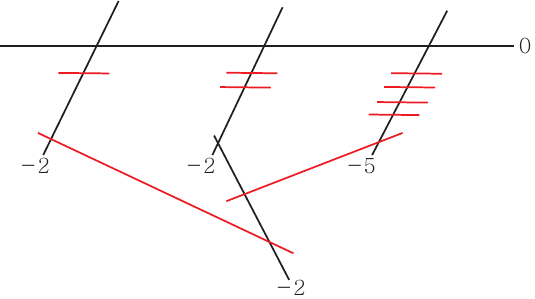}\quad
$\dim= 7$, $\mu= 5$, BO \#227
\end{minipage}

\begin{minipage}{\textwidth}
\plist{I_{30(4-2)+11}[1]}

%
\quad
\includegraphics[width=0.3\textwidth]{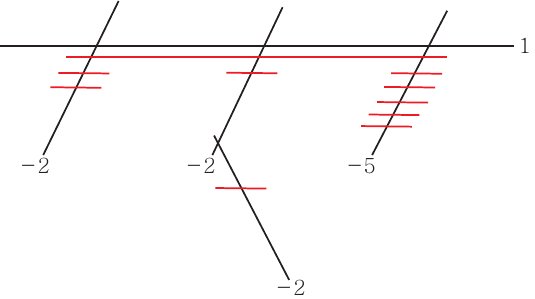}\quad
$\dim= 10$, $\mu= 7$, BO \#164
\end{minipage}

\begin{minipage}{\textwidth}
\plist{I_{30(4-2)+11}[2]}

%
\quad
\includegraphics[width=0.3\textwidth]{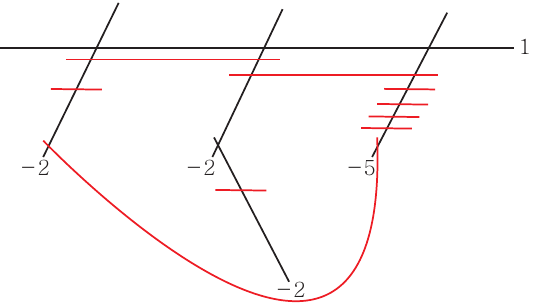}\quad
$\dim= 8$, $\mu= 6$, BO \#165
\end{minipage}

\begin{minipage}{\textwidth}
\plist{I_{30(4-2)+11}[3]}

%
\quad
\includegraphics[width=0.3\textwidth]{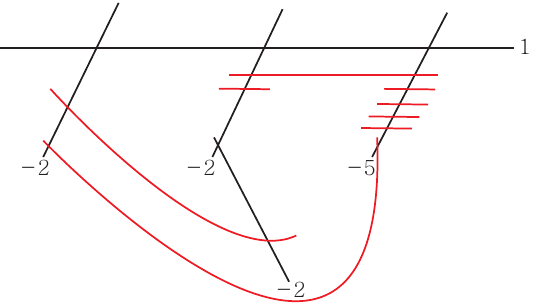}\quad
$\dim= 6$, $\mu= 5$, BO \#228
\end{minipage}

\begin{minipage}{\textwidth}
\plist{I_{30(4-2)+11}[4]}

%
\quad
\includegraphics[width=0.3\textwidth]{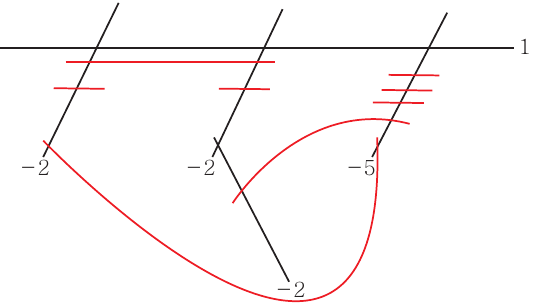}\quad
$\dim= 6$, $\mu= 5$, BO \#229
\end{minipage}

\begin{minipage}{\textwidth}
\plist{I_{30(5-2)+11}[1]}

%
\quad
\includegraphics[width=0.3\textwidth]{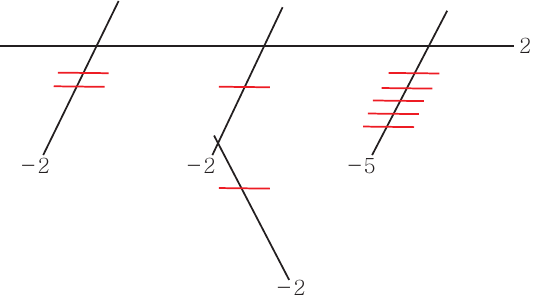}\quad
$\dim= 11$, $\mu= 7$, BO \#166
\end{minipage}

\begin{minipage}{\textwidth}
\plist{I_{30(5-2)+11}[2]}

%
\quad
\includegraphics[width=0.3\textwidth]{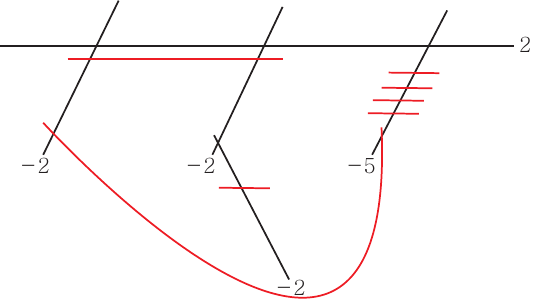}\quad
$\dim= 7$, $\mu= 5$, BO \#167
\end{minipage}

\begin{minipage}{\textwidth}
\plist{I_{30(5-2)+11}[3]}

%
\quad
\includegraphics[width=0.3\textwidth]{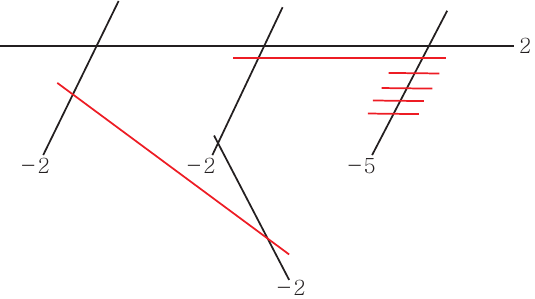}\quad
$\dim= 5$, $\mu= 4$, BO \#230
\end{minipage}

\begin{minipage}{\textwidth}
\plist{I_{30(5-2)+11}[4]}

%
\quad
\includegraphics[width=0.3\textwidth]{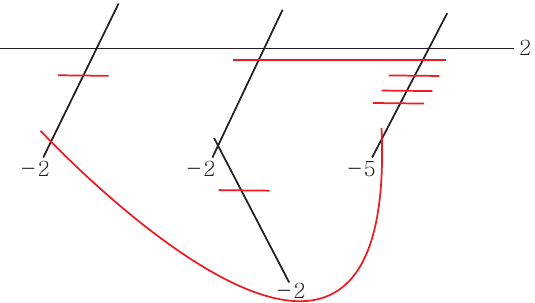}\quad
$\dim= 7$, $\mu= 5$, BO \#168
\end{minipage}

\begin{minipage}{\textwidth}
\plist{I_{30(5-2)+11}[5]}

%
\quad
\includegraphics[width=0.3\textwidth]{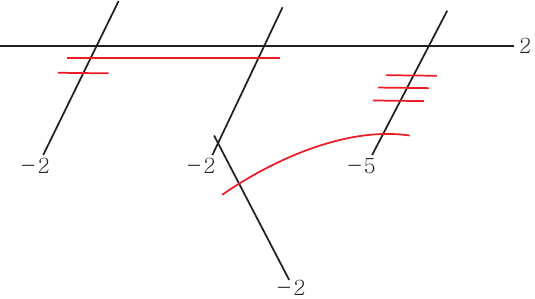}\quad
$\dim= 5$, $\mu= 4$, BO \#231
\end{minipage}

\begin{minipage}{\textwidth}
\plist{I_{30(5-2)+11}[6]}

%
\quad
\includegraphics[width=0.3\textwidth]{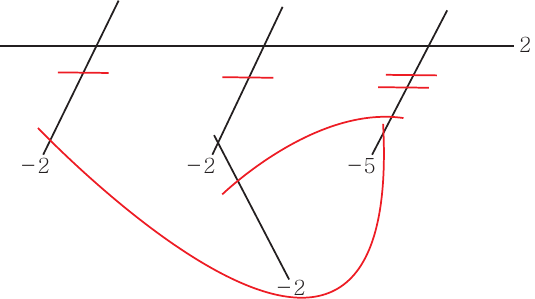}\quad
$\dim= 5$, $\mu= 4$, BO \#232
\end{minipage}

\begin{minipage}{\textwidth}
\plist{I_{30(6-2)+11}[1]}

%
\quad
\includegraphics[width=0.3\textwidth]{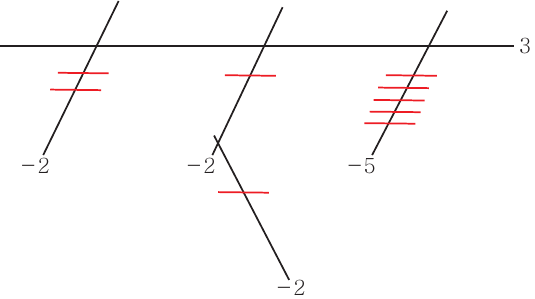}\quad
$\dim= 12$, $\mu= 7$, BO \#169
\end{minipage}

\begin{minipage}{\textwidth}
\plist{I_{30(6-2)+11}[2]}

%
\quad
\includegraphics[width=0.3\textwidth]{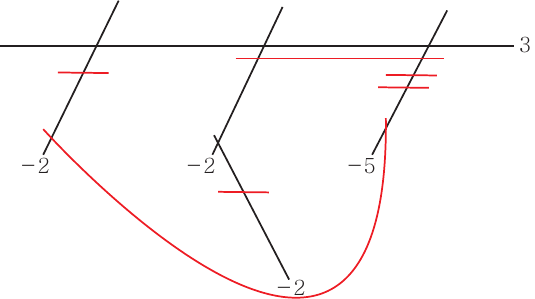}\quad
$\dim= 6$, $\mu= 4$, BO \#170
\end{minipage}

\begin{minipage}{\textwidth}
\plist{I_{30(6-2)+11}[3]}

%
\quad
\includegraphics[width=0.3\textwidth]{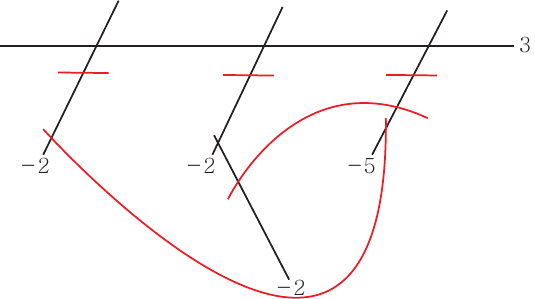}\quad
$\dim= 4$, $\mu= 3$, BO \#233
\end{minipage}

\begin{minipage}{\textwidth}
\plist{I_{30(7-2)+11}[1]}

%
\quad
\includegraphics[width=0.3\textwidth]{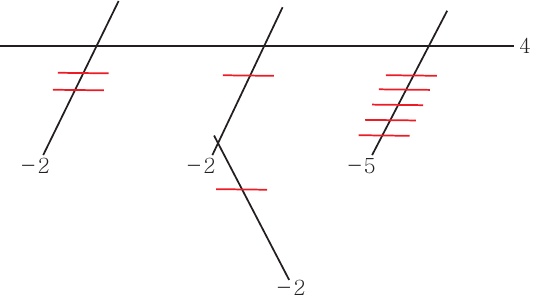}\quad
$\dim= 13$, $\mu= 7$, BO \#171
\end{minipage}

\begin{minipage}{\textwidth}
\plist{I_{30(7-2)+11}[2]}

%
\quad
\includegraphics[width=0.3\textwidth]{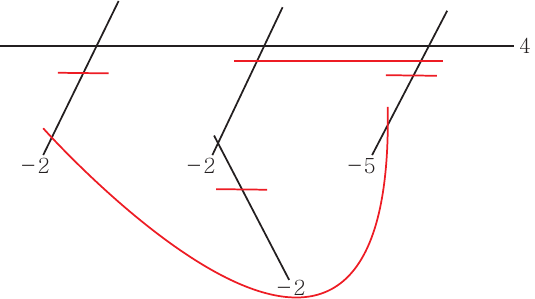}\quad
$\dim= 5$, $\mu= 3$, BO \#172
\end{minipage}

\begin{minipage}{\textwidth}
\plist{I_{30(7-2)+11}[3]}

%
\quad
\includegraphics[width=0.3\textwidth]{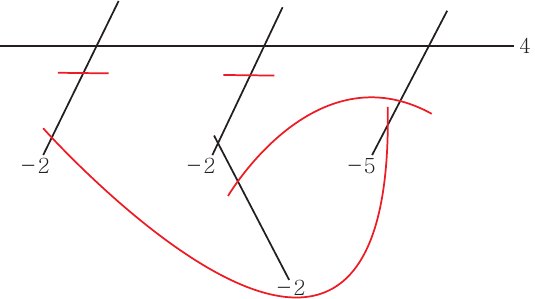}\quad
$\dim= 3$, $\mu= 2$, BO \#234
\end{minipage}

\begin{minipage}{\textwidth}
\plist{I_{30(8-2)+11}[1]}

%
\quad
\includegraphics[width=0.3\textwidth]{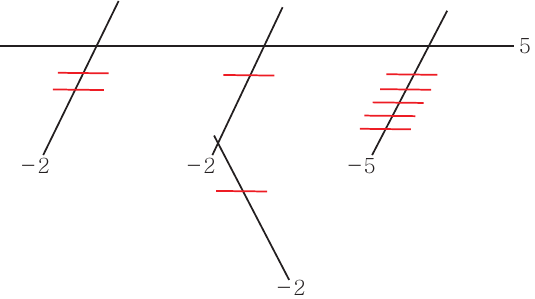}\quad
$\dim= 14$, $\mu= 7$, BO \#173
\end{minipage}

\begin{minipage}{\textwidth}
\plist{I_{30(8-2)+11}[2]}

%
\quad
\includegraphics[width=0.3\textwidth]{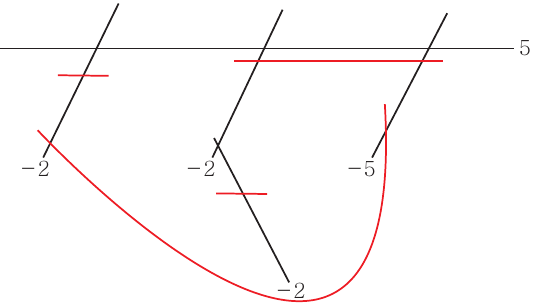}\quad
$\dim= 4$, $\mu= 2$, BO \#174
\end{minipage}

\begin{minipage}{\textwidth}
\plist{I_{30(n-2)+11}[1], n>8}

%
\quad
\includegraphics[width=0.3\textwidth]{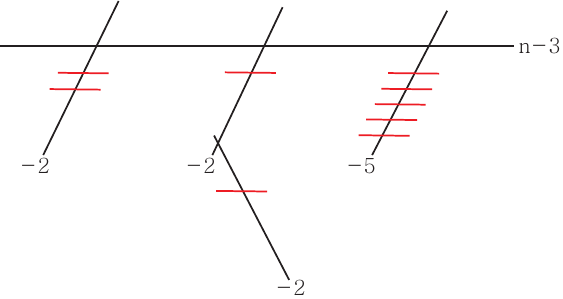}\quad
$\dim= n+6$,$\mu= 7$, BO \#175
\end{minipage}

\subsection*{Icosahedral $I_{30(n-2)+13}$}\hfill

\begin{minipage}{\textwidth}
\plist{I_{30(2-2)+13}[1]}

%
\quad
\includegraphics[width=0.3\textwidth]{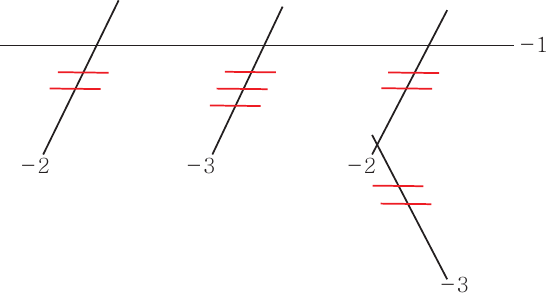}\quad
$\dim= 7$, $\mu= 6$, BO \#98
\end{minipage}

\begin{minipage}{\textwidth}
\plist{I_{30(3-2)+13}[1]}

%
\quad
\includegraphics[width=0.3\textwidth]{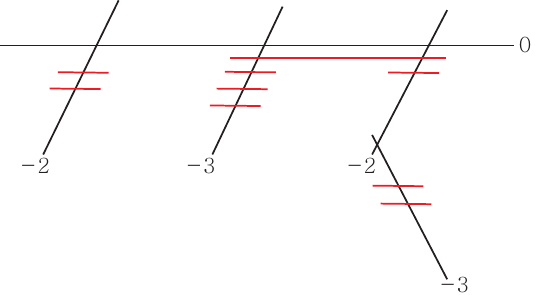}\quad
$\dim= 8$, $\mu= 6$, BO \#99
\end{minipage}

\begin{minipage}{\textwidth}
\plist{I_{30(3-2)+13}[2]}

%
\quad
\includegraphics[width=0.3\textwidth]{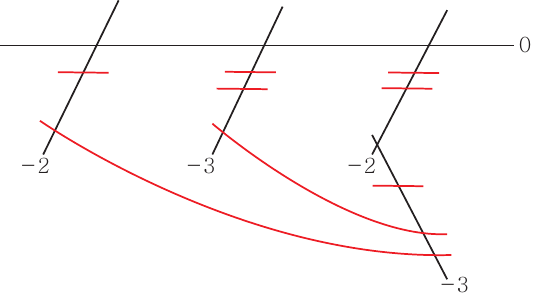}\quad
$\dim= 6$, $\mu= 5$, BO \#100
\end{minipage}

\begin{minipage}{\textwidth}
\plist{I_{30(4-2)+13}[1]}

%
\quad
\includegraphics[width=0.3\textwidth]{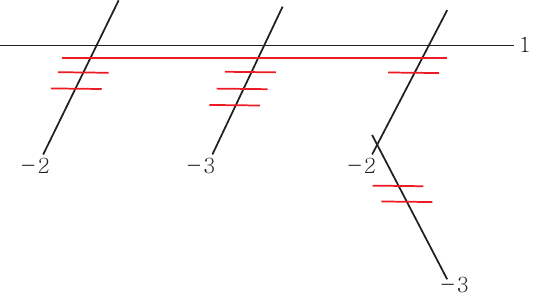}\quad
$\dim= 9$, $\mu= 6$, BO \#101
\end{minipage}

\begin{minipage}{\textwidth}
\plist{I_{30(4-2)+13}[2]}

%
\quad
\includegraphics[width=0.3\textwidth]{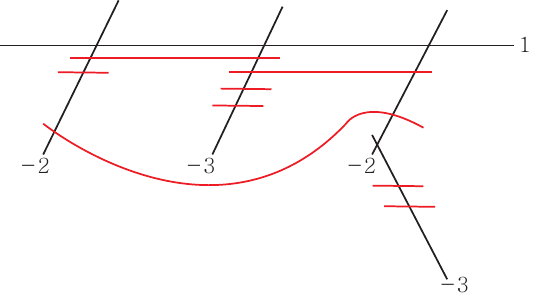}\quad
$\dim= 7$, $\mu= 5$, BO \#104
\end{minipage}

\begin{minipage}{\textwidth}
\plist{I_{30(4-2)+13}[3]}

%
\quad
\includegraphics[width=0.3\textwidth]{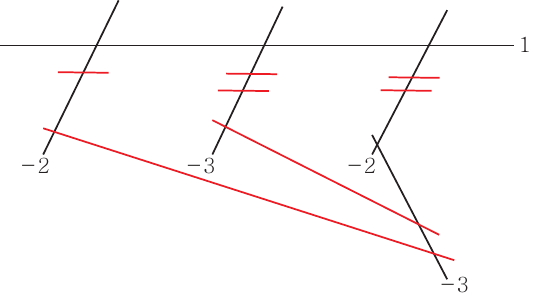}\quad
$\dim= 5$, $\mu= 4$, BO \#105
\end{minipage}

\begin{minipage}{\textwidth}
\plist{I_{30(4-2)+13}[4]}

%
\quad
\includegraphics[width=0.3\textwidth]{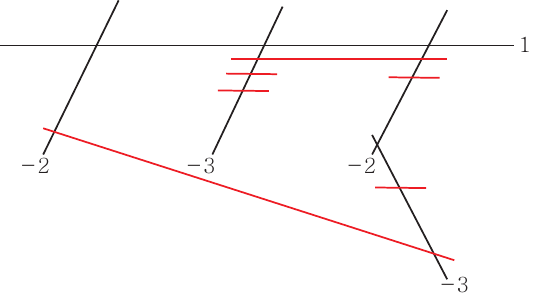}\quad
$\dim= 5$, $\mu= 4$, BO \#102
\end{minipage}

\begin{minipage}{\textwidth}
\plist{I_{30(4-2)+13}[5]}

%
\quad
\includegraphics[width=0.3\textwidth]{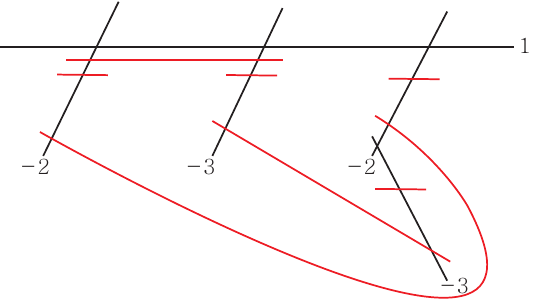}\quad
$\dim= 5$, $\mu= 4$, BO \#103
\end{minipage}

\begin{minipage}{\textwidth}
\plist{I_{30(5-2)+13}[1]}

%
\quad
\includegraphics[width=0.3\textwidth]{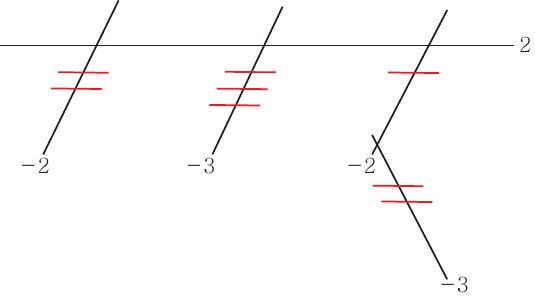}\quad
$\dim= 10$, $\mu= 6$, BO \#106
\end{minipage}

\begin{minipage}{\textwidth}
\plist{I_{30(5-2)+13}[2]}

%
\quad
\includegraphics[width=0.3\textwidth]{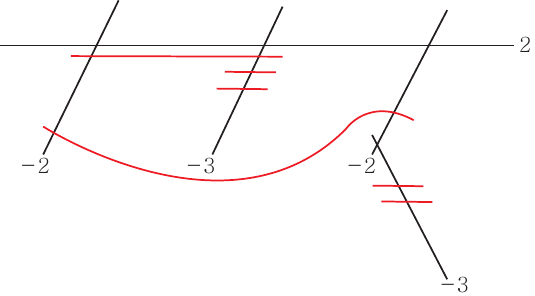}\quad
$\dim= 6$, $\mu= 4$, BO \#107
\end{minipage}

\begin{minipage}{\textwidth}
\plist{I_{30(5-2)+13}[3]}

%
\quad
\includegraphics[width=0.3\textwidth]{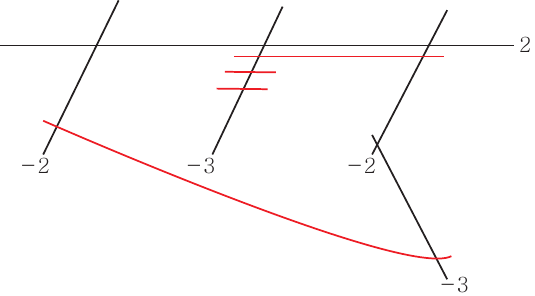}\quad
$\dim= 4$, $\mu= 3$, BO \#110
\end{minipage}

\begin{minipage}{\textwidth}
\plist{I_{30(5-2)+13}[4]}

%
\quad
\includegraphics[width=0.3\textwidth]{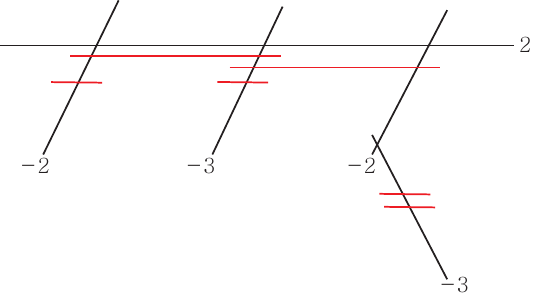}\quad
$\dim= 6$, $\mu= 4$, BO \#108
\end{minipage}

\begin{minipage}{\textwidth}
\plist{I_{30(5-2)+13}[5]}

%
\quad
\includegraphics[width=0.3\textwidth]{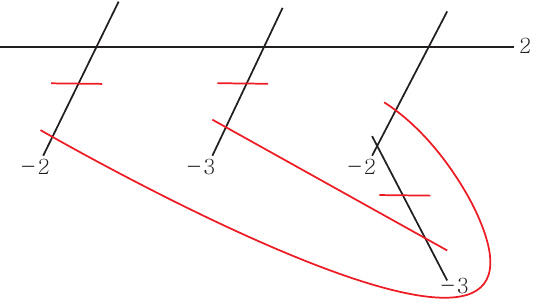}\quad
$\dim= 4$, $\mu= 3$, BO \#111
\end{minipage}

\begin{minipage}{\textwidth}
\plist{I_{30(5-2)+13}[6]}

%
\quad
\includegraphics[width=0.3\textwidth]{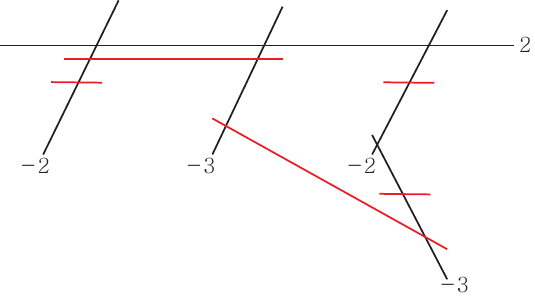}\quad
$\dim= 4$, $\mu= 3$, BO \#109
\end{minipage}

\begin{minipage}{\textwidth}
\plist{I_{30(6-2)+13}[1]}

%
\quad
\includegraphics[width=0.3\textwidth]{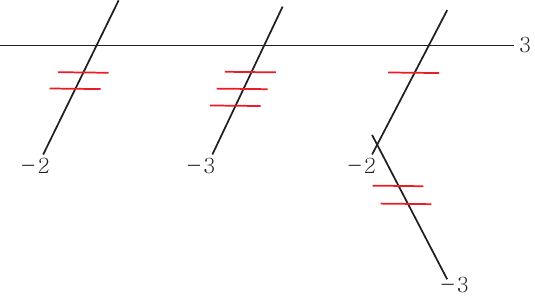}\quad
$\dim= 11$, $\mu= 6$, BO \#112
\end{minipage}

\begin{minipage}{\textwidth}
\plist{I_{30(6-2)+13}[2]}

%
\quad
\includegraphics[width=0.3\textwidth]{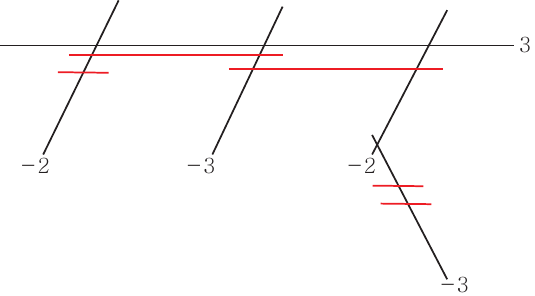}\quad
$\dim= 5$, $\mu= 3$, BO \#113
\end{minipage}

\begin{minipage}{\textwidth}
\plist{I_{30(n-2)+13}[1], n>6}

%
\quad
\includegraphics[width=0.3\textwidth]{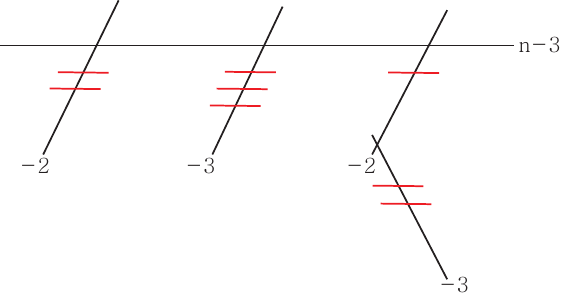}\quad
$\dim= n+5$, $\mu= 6$, BO \#114
\end{minipage}

\subsection*{Icosahedral $I_{30(n-2)+17}$}\hfill

\begin{minipage}{\textwidth}
\plist{I_{30(2-2)+17}[1]}

%
\quad
\includegraphics[width=0.3\textwidth]{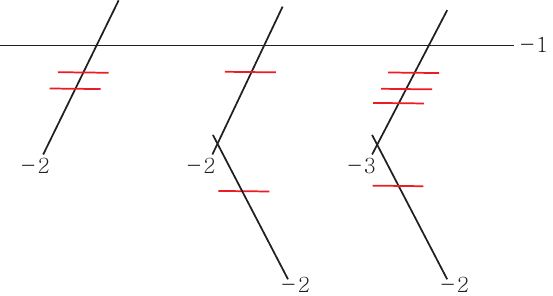}\quad
$\dim= 7$, $\mu= 5$, BO \#176
\end{minipage}

\begin{minipage}{\textwidth}
\plist{I_{30(2-2)+17}[2]}

%
\quad
\includegraphics[width=0.3\textwidth]{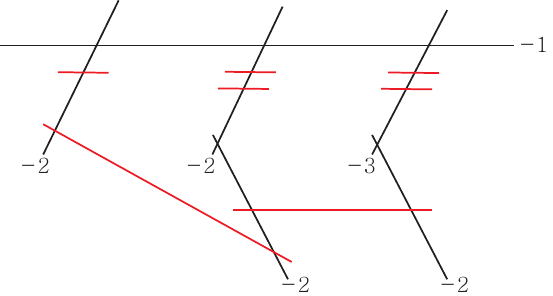}\quad
$\dim= 5$, $\mu= 4$, BO \#235
\end{minipage}

\begin{minipage}{\textwidth}
\plist{I_{30(3-2)+17}[1]}

%
\quad
\includegraphics[width=0.3\textwidth]{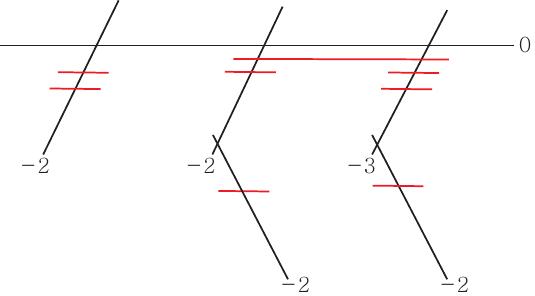}\quad
$\dim= 8$, $\mu= 5$, BO \#177
\end{minipage}

\begin{minipage}{\textwidth}
\plist{I_{30(3-2)+17}[2]}

%
\quad
\includegraphics[width=0.3\textwidth]{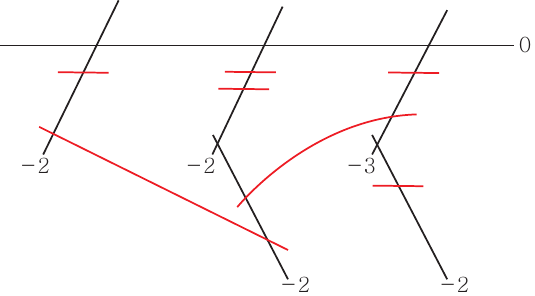}\quad
$\dim= 6$, $\mu= 4$, BO \#236
\end{minipage}

\begin{minipage}{\textwidth}
\plist{I_{30(3-2)+17}[3]}

%
\quad
\includegraphics[width=0.3\textwidth]{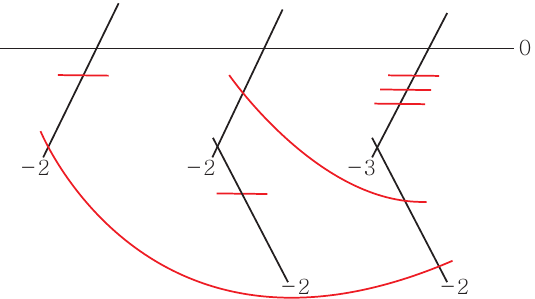}\quad
$\dim= 6$, $\mu= 4$, BO \#178
\end{minipage}

\begin{minipage}{\textwidth}
\plist{I_{30(3-2)+17}[4]}

%
\quad
\includegraphics[width=0.3\textwidth]{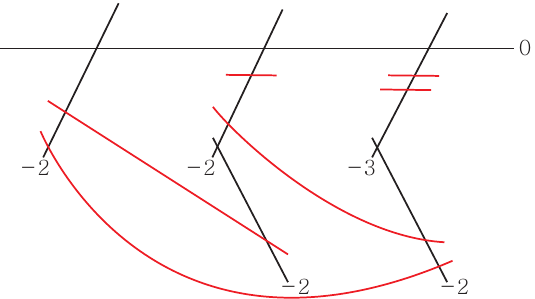}\quad
$\dim= 4$, $\mu= 3$, BO \#237
\end{minipage}

\begin{minipage}{\textwidth}
\plist{I_{30(4-2)+17}[1]}

%
\quad
\includegraphics[width=0.3\textwidth]{I-17-4-1}\quad
$\dim= 9$, $\mu= 5$, BO \#179
\end{minipage}

\begin{minipage}{\textwidth}
\plist{I_{30(4-2)+17}[2]}

%
\quad
\includegraphics[width=0.3\textwidth]{I-17-4-2}\quad
$\dim= 7$, $\mu= 4$, BO \#181
\end{minipage}

\begin{minipage}{\textwidth}
\plist{I_{30(4-2)+17}[3]}

%
\quad
\includegraphics[width=0.3\textwidth]{I-17-4-3}\quad
$\dim= 5$, $\mu= 3$, BO \#238
\end{minipage}

\begin{minipage}{\textwidth}
\plist{I_{30(4-2)+17}[4]}

%
\quad
\includegraphics[width=0.3\textwidth]{I-17-4-4}\quad
$\dim= 5$, $\mu= 3$, BO \#240
\end{minipage}

\begin{minipage}{\textwidth}
\plist{I_{30(4-2)+17}[5]}

%
\quad
\includegraphics[width=0.3\textwidth]{I-17-4-5}\quad
$\dim= 5$, $\mu= 3$, BO \#180
\end{minipage}

\begin{minipage}{\textwidth}
\plist{I_{30(4-2)+17}[6]}

%
\quad
\includegraphics[width=0.3\textwidth]{I-17-4-6}\quad
$\dim= 3$, $\mu= 2$, BO \#239
\end{minipage}

\begin{minipage}{\textwidth}
\plist{I_{30(4-2)+17}[7]}

%
\quad
\includegraphics[width=0.3\textwidth]{I-17-4-7}\quad
$\dim= 3$, $\mu= 2$, BO \#241
\end{minipage}

\begin{minipage}{\textwidth}
\plist{I_{30(5-2)+17}[1]}

%
\quad
\includegraphics[width=0.3\textwidth]{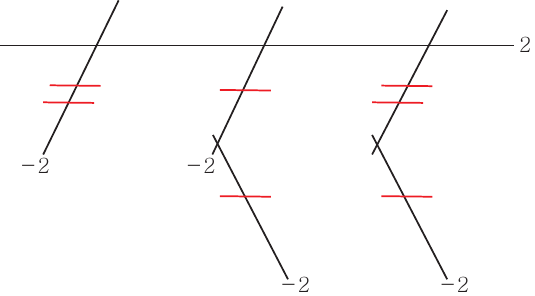}\quad
$\dim= 10$, $\mu= 5$, BO \#182
\end{minipage}

\begin{minipage}{\textwidth}
\plist{I_{30(5-2)+17}[2]}

%
\quad
\includegraphics[width=0.3\textwidth]{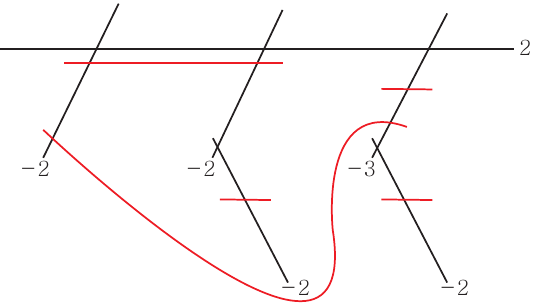}\quad
$\dim= 6$, $\mu= 3$, BO \#185
\end{minipage}

\begin{minipage}{\textwidth}
\plist{I_{30(5-2)+17}[3]}

%
\quad
\includegraphics[width=0.3\textwidth]{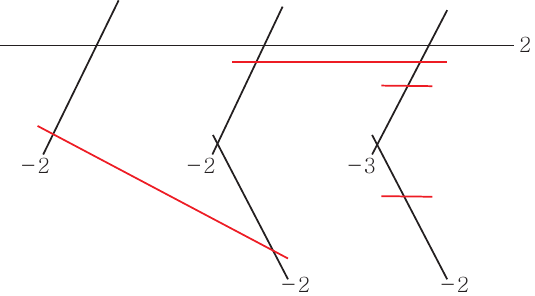}\quad
$\dim= 4$, $\mu= 2$, BO \#242
\end{minipage}

\begin{minipage}{\textwidth}
\plist{I_{30(5-2)+17}[4]}

%
\quad
\includegraphics[width=0.3\textwidth]{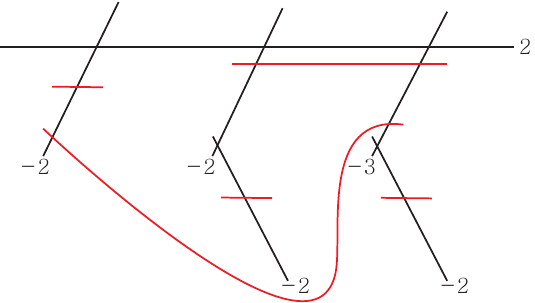}\quad
$\dim= 6$, $\mu= 3$, BO \#184
\end{minipage}

\begin{minipage}{\textwidth}
\plist{I_{30(5-2)+17}[5]}

%
\quad
\includegraphics[width=0.3\textwidth]{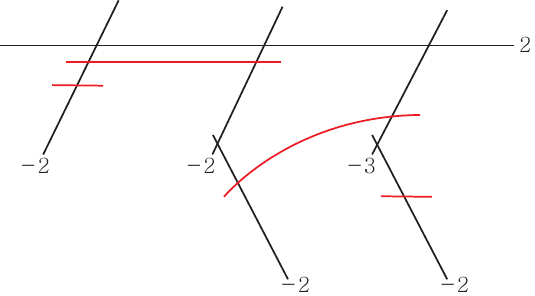}\quad
$\dim= 4$, $\mu= 2$, BO \#243
\end{minipage}

\begin{minipage}{\textwidth}
\plist{I_{30(5-2)+17}[6]}

%
\quad
\includegraphics[width=0.3\textwidth]{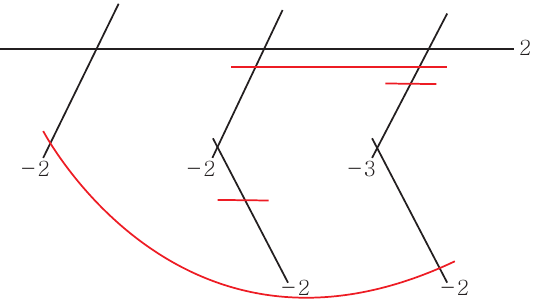}\quad
$\dim= 4$, $\mu= 2$, BO \#183
\end{minipage}

\begin{minipage}{\textwidth}
\plist{I_{30(5-2)+17}[7]}

%
\quad
\includegraphics[width=0.3\textwidth]{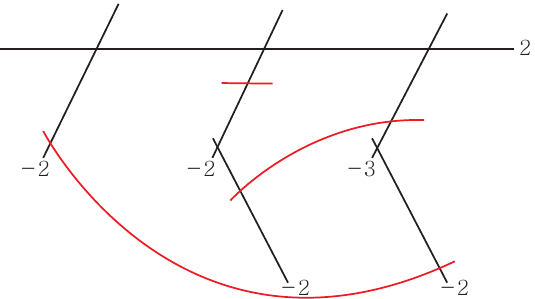}\quad
$\dim= 2$, $\mu= 1$, BO \#244
\end{minipage}

\begin{minipage}{\textwidth}
\plist{I_{30(6-2)+17}[1]}

%
\quad
\includegraphics[width=0.3\textwidth]{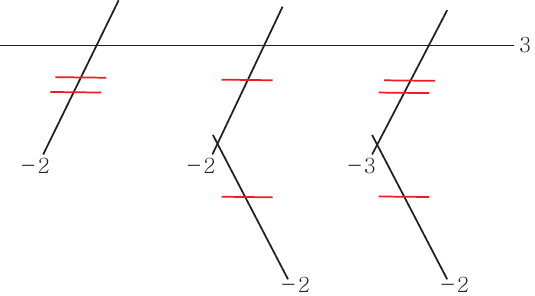}\quad
$\dim= 11$, $\mu= 5$, BO \#186
\end{minipage}

\begin{minipage}{\textwidth}
\plist{I_{30(6-2)+17}[2]}

%
\quad
\includegraphics[width=0.3\textwidth]{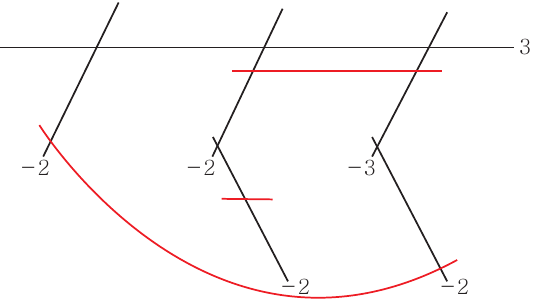}\quad
$\dim= 3$, $\mu= 1$, BO \#187
\end{minipage}

\begin{minipage}{\textwidth}
\plist{I_{30(n-2)+17}[1], n>6}

%
\quad
\includegraphics[width=0.3\textwidth]{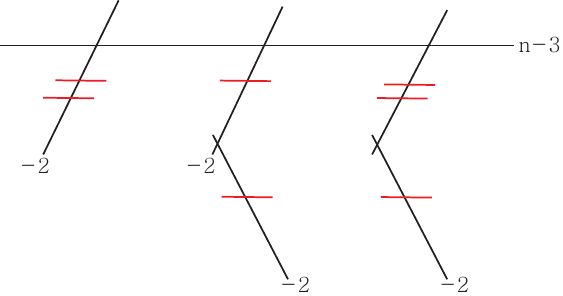}\quad
$\dim= n+5$, $\mu= 5$, BO \#188
\end{minipage}

\subsection*{Icosahedral $I_{30(n-2)+19}$}\hfill

\begin{minipage}{\textwidth}
\plist{I_{30(2-2)+19}[1]}

%
\quad
\includegraphics[width=0.3\textwidth]{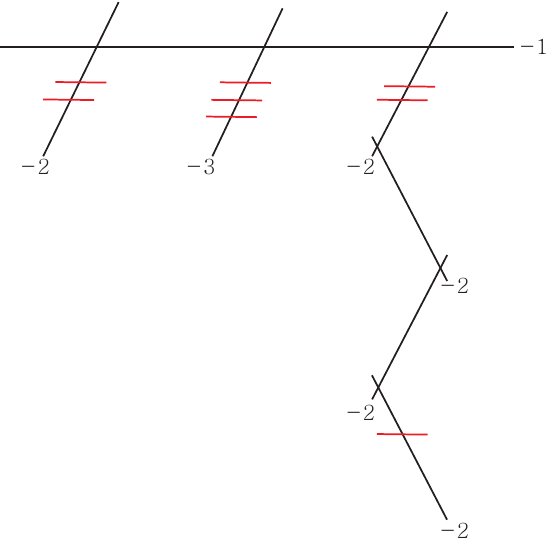}\quad
$\dim= 8$, $\mu= 5$, BO \#115
\end{minipage}

\begin{minipage}{\textwidth}
\plist{I_{30(2-2)+19}[2]}

%
\quad
\includegraphics[width=0.3\textwidth]{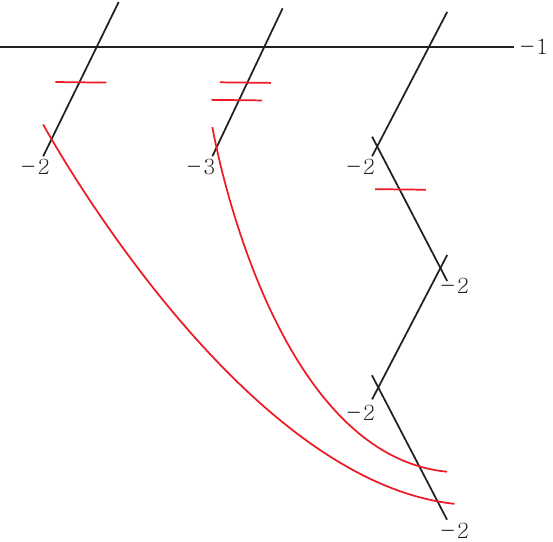}\quad
$\dim= 4$, $\mu= 3$, BO \#116
\end{minipage}

\begin{minipage}{\textwidth}
\plist{I_{30(3-2)+19}[1]}

%
\quad
\includegraphics[width=0.3\textwidth]{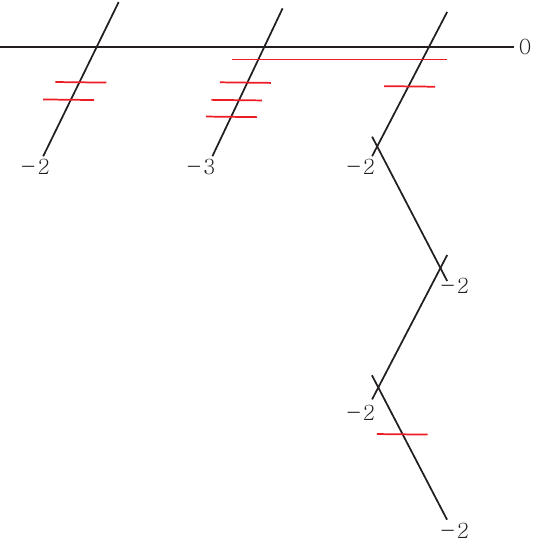}\quad
$\dim= 9$, $\mu= 5$, BO \#117
\end{minipage}

\begin{minipage}{\textwidth}
\plist{I_{30(3-2)+19}[2]}

%
\quad
\includegraphics[width=0.3\textwidth]{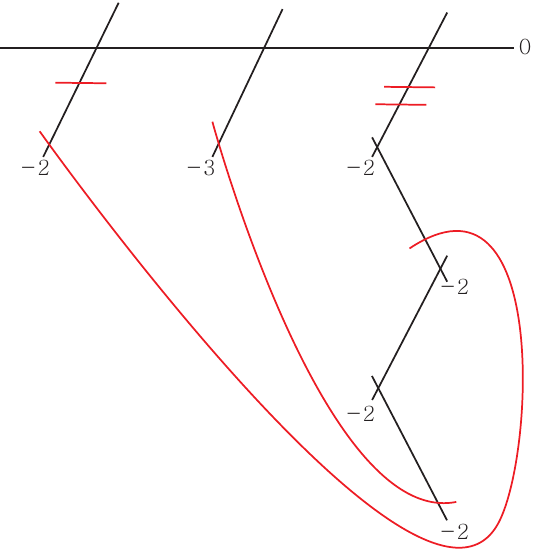}\quad
$\dim= 3$, $\mu= 2$, BO \#118
\end{minipage}

\begin{minipage}{\textwidth}
\plist{I_{30(4-2)+19}[1]}

%
\quad
\includegraphics[width=0.3\textwidth]{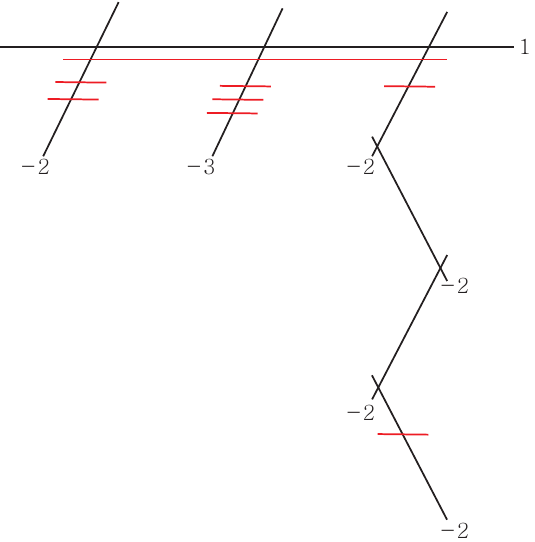}\quad
$\dim= 10$, $\mu= 5$, BO \#119
\end{minipage}

\begin{minipage}{\textwidth}
\plist{I_{30(4-2)+19}[2]}

%
\quad
\includegraphics[width=0.3\textwidth]{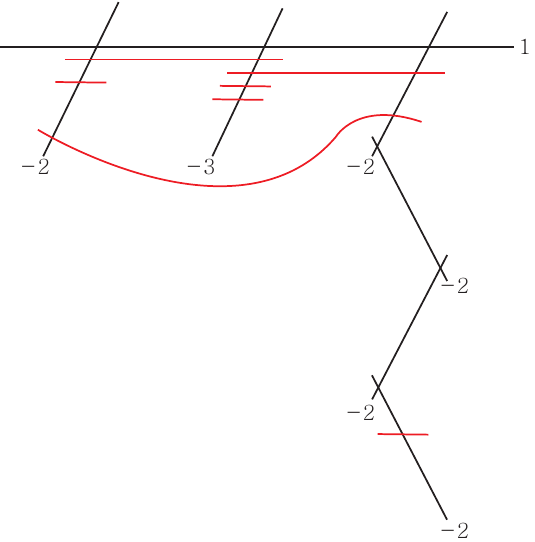}\quad
$\dim= 8$, $\mu= 4$, BO \#120
\end{minipage}

\begin{minipage}{\textwidth}
\plist{I_{30(5-2)+19}[1]}

%
\quad
\includegraphics[width=0.3\textwidth]{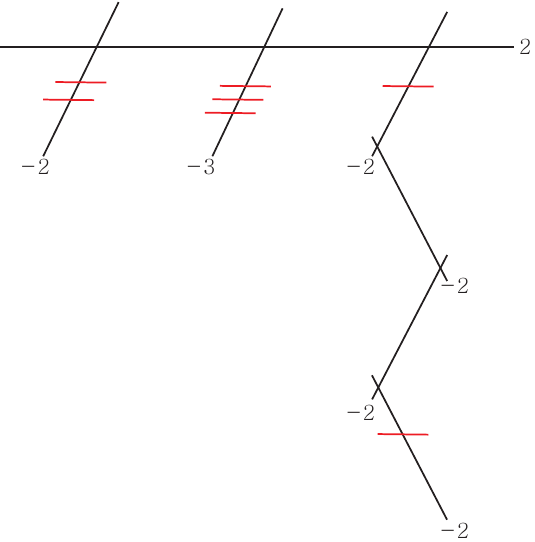}\quad
$\dim= 11$, $\mu= 5$, BO \#121
\end{minipage}

\begin{minipage}{\textwidth}
\plist{I_{30(5-2)+19}[2]}

%
\quad
\includegraphics[width=0.3\textwidth]{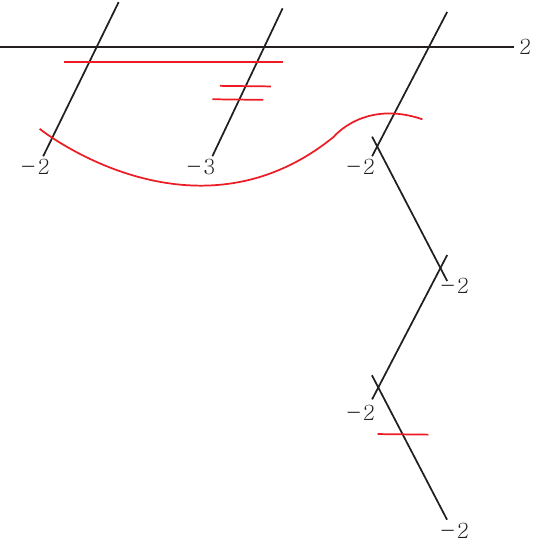}\quad
$\dim= 7$, $\mu= 3$, BO \#122
\end{minipage}

\begin{minipage}{\textwidth}
\plist{I_{30(5-2)+19}[3]}

%
\quad
\includegraphics[width=0.3\textwidth]{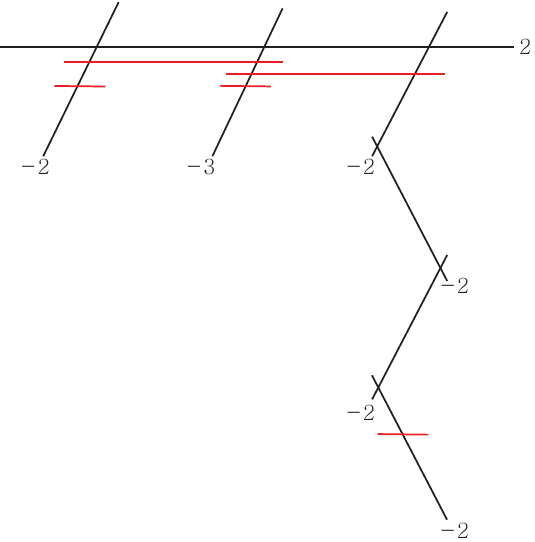}\quad
$\dim= 7$, $\mu= 3$, BO \#123
\end{minipage}

\begin{minipage}{\textwidth}
\plist{I_{30(6-2)+19}[1]}

%
\quad
\includegraphics[width=0.3\textwidth]{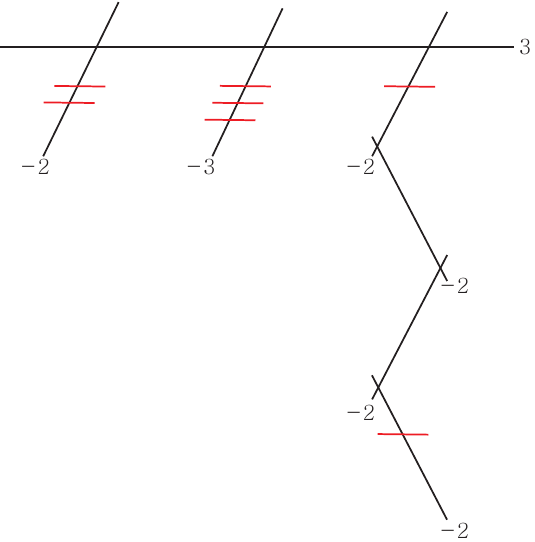}\quad
$\dim= 12$, $\mu= 5$, BO \#124
\end{minipage}

\begin{minipage}{\textwidth}
\plist{I_{30(6-2)+19}[2]}

%
\quad
\includegraphics[width=0.3\textwidth]{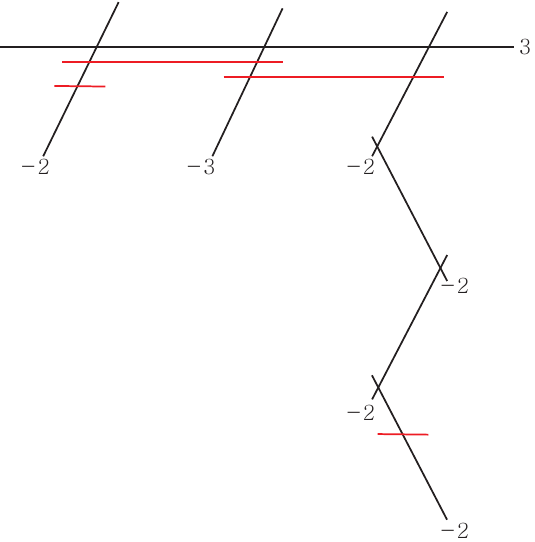}\quad
$\dim= 6$, $\mu= 2$, BO \#125
\end{minipage}

\begin{minipage}{\textwidth}
\plist{I_{30(n-2)+19}[1], n>6}

%
\quad
\includegraphics[width=0.3\textwidth]{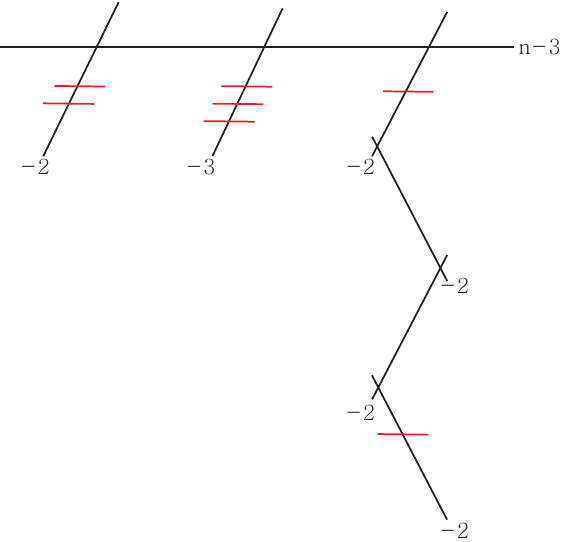}\quad
$\dim= n+6$, $\mu= 5$, BO \#126
\end{minipage}

\subsection*{Icosahedral Type $I_{30(n-2)+23}$}\hfill

\begin{minipage}{\textwidth}
\plist{I_{30(2-2)+23}[1]}

%
\quad
\includegraphics[width=0.3\textwidth]{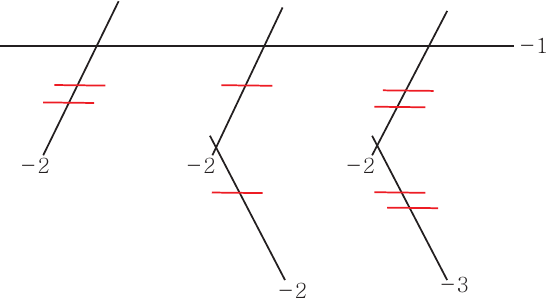}\quad
$\dim= 7$, $\mu= 5$, BO \#189
\end{minipage}

\begin{minipage}{\textwidth}
\plist{I_{30(2-2)+23}[2]}

%
\quad
\includegraphics[width=0.3\textwidth]{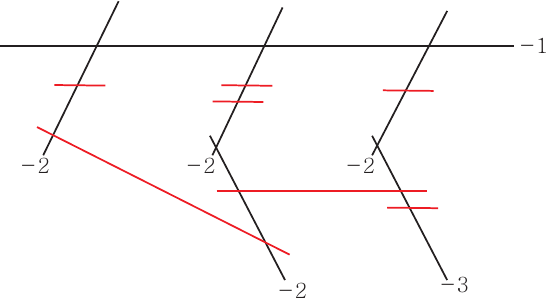}\quad
$\dim= 4$, $\mu= 4$, BO \#245
\end{minipage}

\begin{minipage}{\textwidth}
\plist{I_{30(3-2)+23}[1]}

%
\quad
\includegraphics[width=0.3\textwidth]{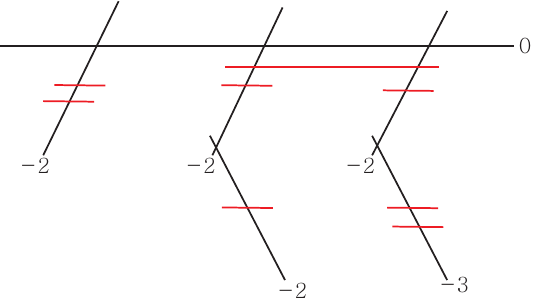}\quad
$\dim= 8$, $\mu= 5$, BO \#190
\end{minipage}

\begin{minipage}{\textwidth}
\plist{I_{30(3-2)+23}[2]}

%
\quad
\includegraphics[width=0.3\textwidth]{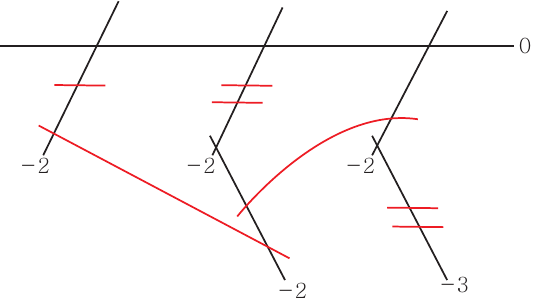}\quad
$\dim= 6$, $\mu= 4$, BO \#246
\end{minipage}

\begin{minipage}{\textwidth}
\plist{I_{30(3-2)+23}[3]}

%
\quad
\includegraphics[width=0.3\textwidth]{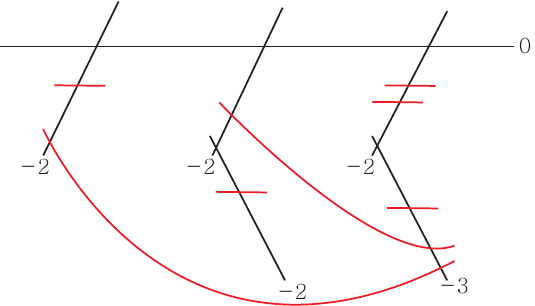}\quad
$\dim= 6$, $\mu= 4$, BO \#191
\end{minipage}

\begin{minipage}{\textwidth}
\plist{I_{30(3-2)+23}[4]}

%
\quad
\includegraphics[width=0.3\textwidth]{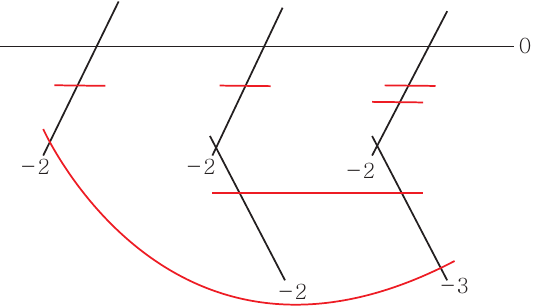}\quad
$\dim= 4$, $\mu= 3$, BO \#248
\end{minipage}

\begin{minipage}{\textwidth}
\plist{I_{30(3-2)+23}[5]}

%
\quad
\includegraphics[width=0.3\textwidth]{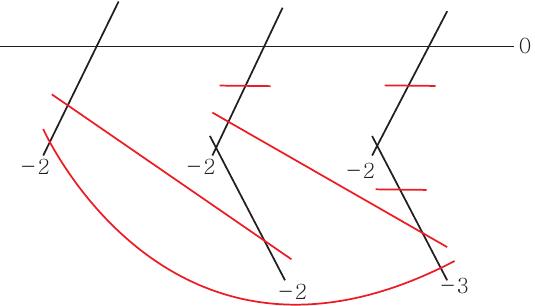}\quad
$\dim= 4$, $\mu= 3$, BO \#247
\end{minipage}

\begin{minipage}{\textwidth}
\plist{I_{30(4-2)+23}[1]}

%
\quad
\includegraphics[width=0.3\textwidth]{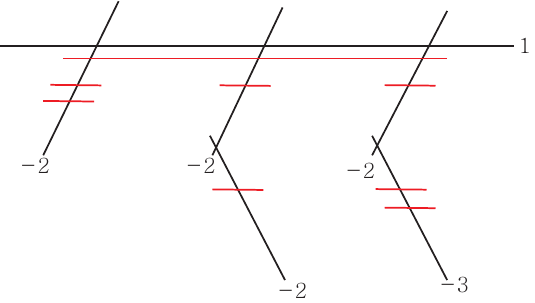}\quad
$\dim= 9$, $\mu= 5$, BO \#192
\end{minipage}

\begin{minipage}{\textwidth}
\plist{I_{30(4-2)+23}[2]}

%
\quad
\includegraphics[width=0.3\textwidth]{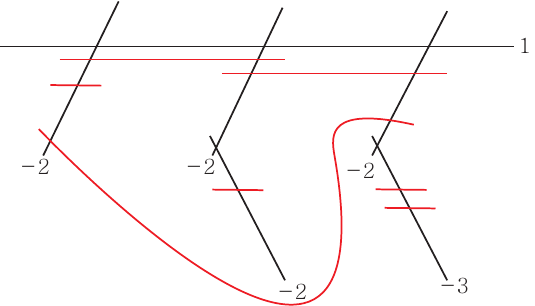}\quad
$\dim= 7$, $\mu= 4$, BO \#194
\end{minipage}

\begin{minipage}{\textwidth}
\plist{I_{30(4-2)+23}[3]}

%
\quad
\includegraphics[width=0.3\textwidth]{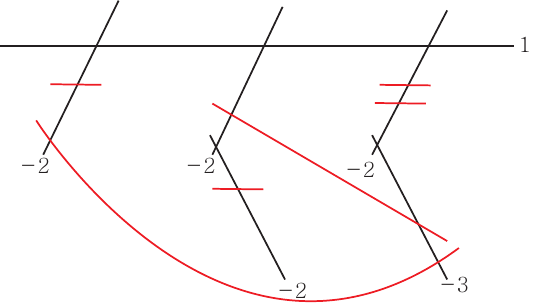}\quad
$\dim= 5$, $\mu= 3$, BO \#195
\end{minipage}

\begin{minipage}{\textwidth}
\plist{I_{30(4-2)+23}[4]}

%
\quad
\includegraphics[width=0.3\textwidth]{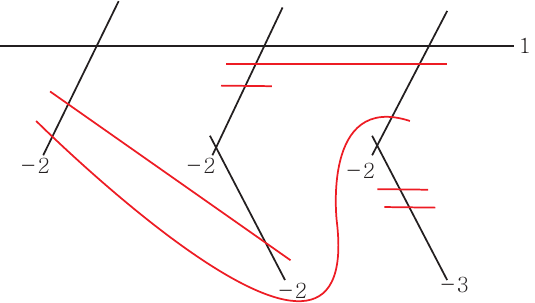}\quad
$\dim= 5$, $\mu= 3$, BO \#249
\end{minipage}

\begin{minipage}{\textwidth}
\plist{I_{30(4-2)+23}[5]}

%
\quad
\includegraphics[width=0.3\textwidth]{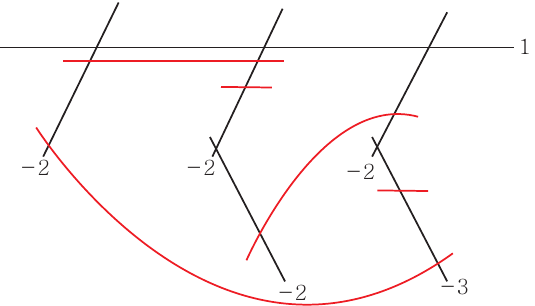}\quad
$\dim= 3$, $\mu= 2$, BO \#251
\end{minipage}

\begin{minipage}{\textwidth}
\plist{I_{30(4-2)+23}[6]}

%
\quad
\includegraphics[width=0.3\textwidth]{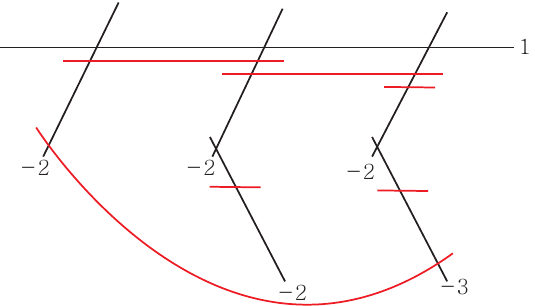}\quad
$\dim= 5$, $\mu= 3$, BO \#193
\end{minipage}

\begin{minipage}{\textwidth}
\plist{I_{30(4-2)+23}[7]}

%
\quad
\includegraphics[width=0.3\textwidth]{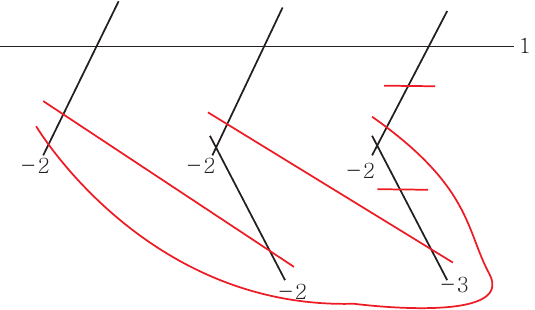}\quad
$\dim= 3$, $\mu= 2$, BO \#250
\end{minipage}

\begin{minipage}{\textwidth}
\plist{I_{30(4-2)+23}[8]}

%
\quad
\includegraphics[width=0.3\textwidth]{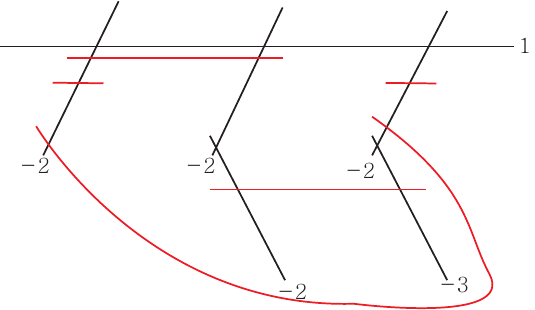}\quad
$\dim= 3$, $\mu= 2$, BO \#252
\end{minipage}

\begin{minipage}{\textwidth}
\plist{I_{30(5-2)+23}[1]}

%
\quad
\includegraphics[width=0.3\textwidth]{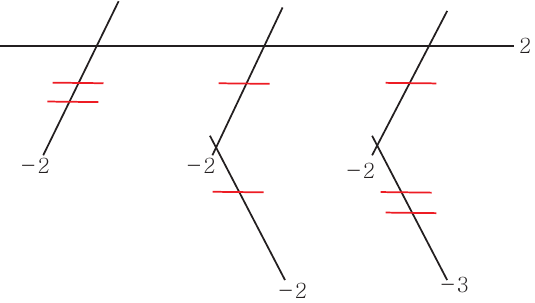}\quad
$\dim= 10$, $\mu= 5$, BO \#196
\end{minipage}

\begin{minipage}{\textwidth}
\plist{I_{30(5-2)+23}[2]}

%
\quad
\includegraphics[width=0.3\textwidth]{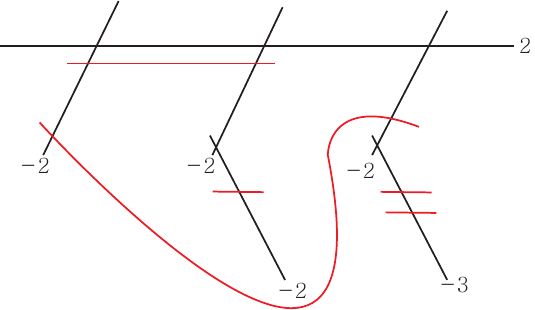}\quad
$\dim= 4$, $\mu= 3$, BO \#197
\end{minipage}

\begin{minipage}{\textwidth}
\plist{I_{30(5-2)+23}[3]}

%
\quad
\includegraphics[width=0.3\textwidth]{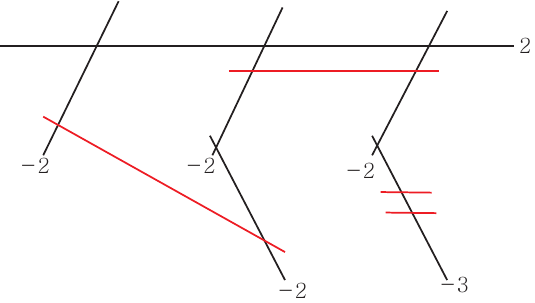}\quad
$\dim= 4$, $\mu= 2$, BO \#253
\end{minipage}

\begin{minipage}{\textwidth}
\plist{I_{30(5-2)+23}[4]}

%
\quad
\includegraphics[width=0.3\textwidth]{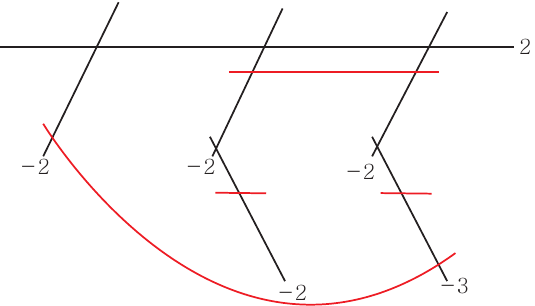}\quad
$\dim= 4$, $\mu= 2$, BO \#198
\end{minipage}

\begin{minipage}{\textwidth}
\plist{I_{30(5-2)+23}[5]}

%
\quad
\includegraphics[width=0.3\textwidth]{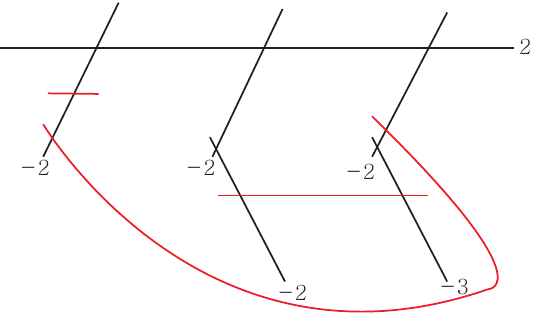}\quad
$\dim= 2$, $\mu= 1$, BO \#258
\end{minipage}

\begin{minipage}{\textwidth}
\plist{I_{30(n-2)+23}[1], n>5}

%
\quad
\includegraphics[width=0.3\textwidth]{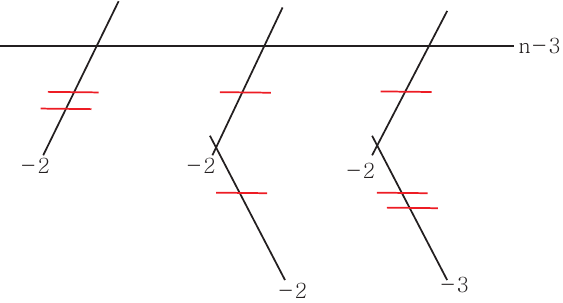}\quad
$\dim= n+5$, $\mu= 5$, BO \#199
\end{minipage}

\subsection*{Icosahedral $I_{30(n-2)+29}$}\hfill

\begin{minipage}{\textwidth}
\plist{I_{30(2-2)+29}[1]}

%
\quad
\includegraphics[width=0.3\textwidth]{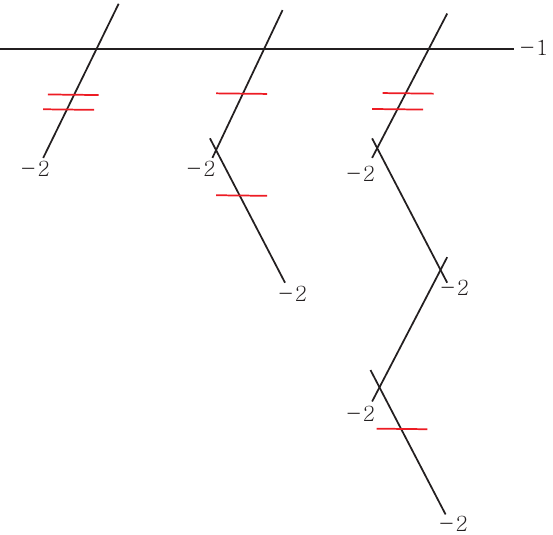}\quad
$\dim= 8$, $\mu= 4$, BO \#200
\end{minipage}

\begin{minipage}{\textwidth}
\plist{I_{30(2-2)+29}[2]}

%
\quad
\includegraphics[width=0.3\textwidth]{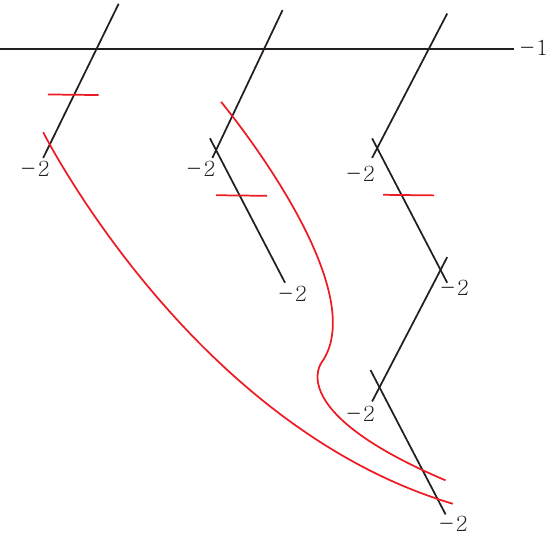}\quad
$\dim= 4$, $\mu= 2$, BO \#201
\end{minipage}

\begin{minipage}{\textwidth}
\plist{I_{30(2-2)+29}[3]}

%
\quad
\includegraphics[width=0.3\textwidth]{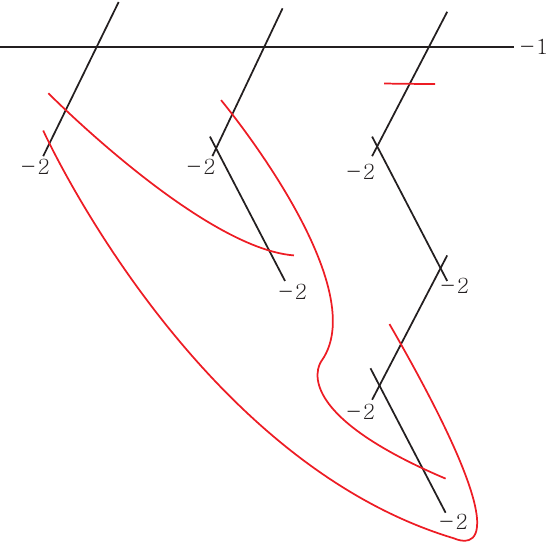}\quad
$\dim= 2$, $\mu= 1$, BO \#254
\end{minipage}

\begin{minipage}{\textwidth}
\plist{I_{30(3-2)+29}[1]}

%
\quad
\includegraphics[width=0.3\textwidth]{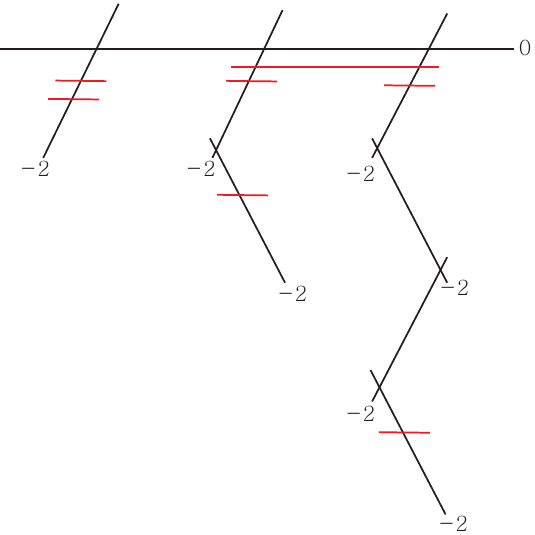}\quad
$\dim= 9$, $\mu= 4$, BO \#202
\end{minipage}

\begin{minipage}{\textwidth}
\plist{I_{30(3-2)+29}[2]}

%
\quad
\includegraphics[width=0.3\textwidth]{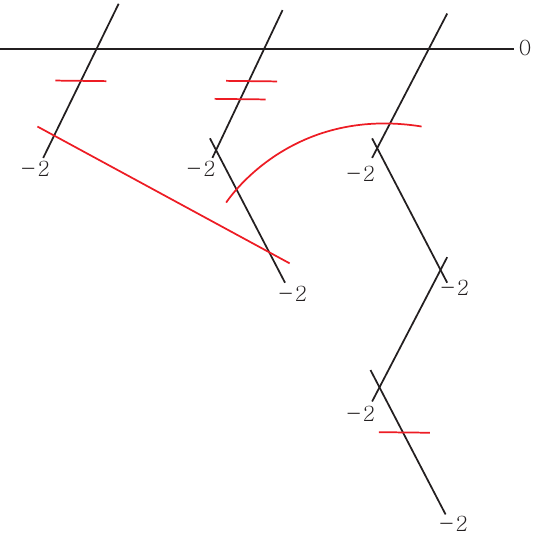}\quad
$\dim= 7$, $\mu= 3$, BO \#255
\end{minipage}

\begin{minipage}{\textwidth}
\plist{I_{30(4-2)+29}[1]}

%
\quad
\includegraphics[width=0.3\textwidth]{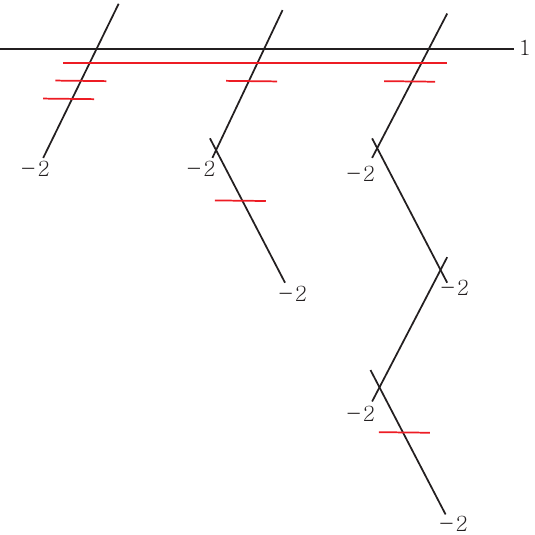}\quad
$\dim= 10$, $\mu= 4$, BO \#203
\end{minipage}

\begin{minipage}{\textwidth}
\plist{I_{30(4-2)+29}[2]}

%
\quad
\includegraphics[width=0.3\textwidth]{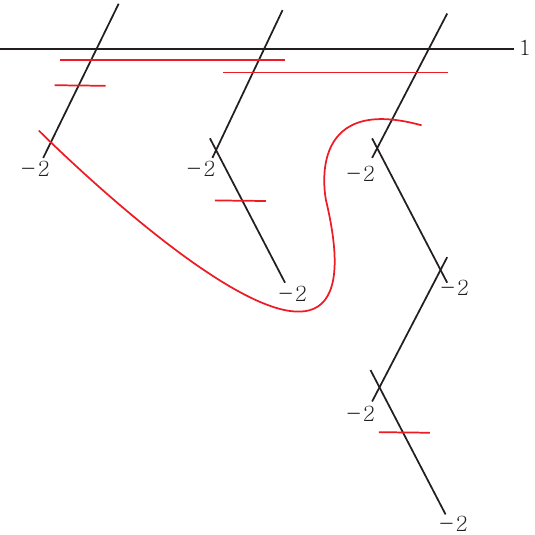}\quad
$\dim= 8$, $\mu= 3$, BO \#204
\end{minipage}

\begin{minipage}{\textwidth}
\plist{I_{30(4-2)+29}[3]}

%
\quad
\includegraphics[width=0.3\textwidth]{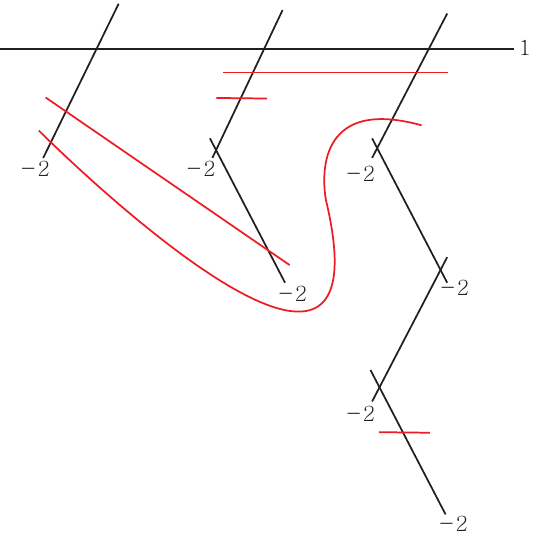}\quad
$\dim= 6$, $\mu= 2$, BO \#256
\end{minipage}

\begin{minipage}{\textwidth}
\plist{I_{30(5-2)+29}[1]}

%
\quad
\includegraphics[width=0.3\textwidth]{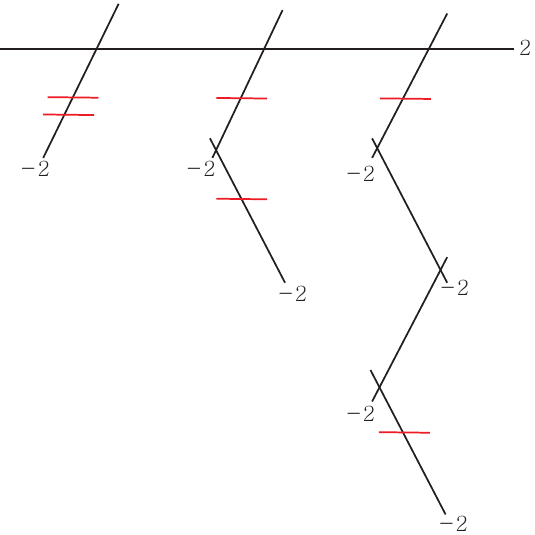}\quad
$\dim= 11$, $\mu= 4$, BO \#205
\end{minipage}

\begin{minipage}{\textwidth}
\plist{I_{30(5-2)+29}[2]}

%
\quad
\includegraphics[width=0.3\textwidth]{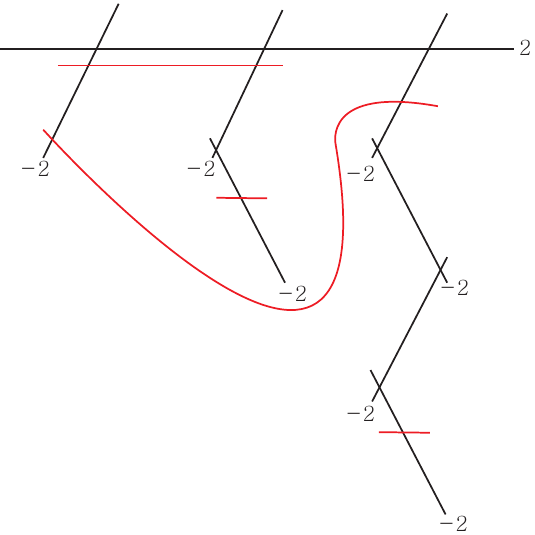}\quad
$\dim= 7$, $\mu= 2$, BO \#206
\end{minipage}

\begin{minipage}{\textwidth}
\plist{I_{30(5-2)+29}[3]}

%
\quad
\includegraphics[width=0.3\textwidth]{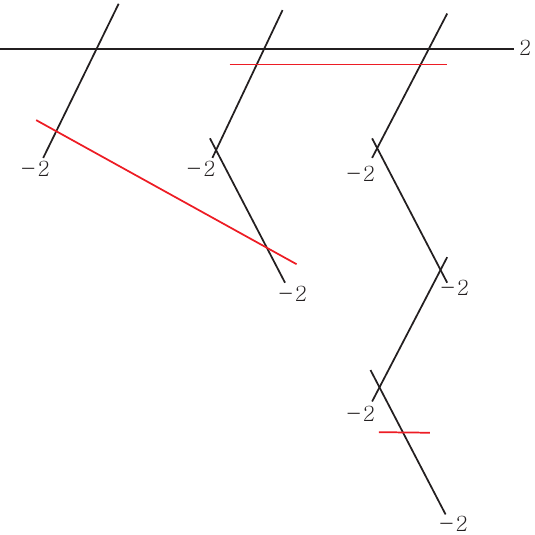}\quad
$\dim= 5$, $\mu= 1$, BO \#257
\end{minipage}

\begin{minipage}{\textwidth}
\plist{I_{30(n-2)+29}[1], n>5}

%
\quad
\includegraphics[width=0.3\textwidth]{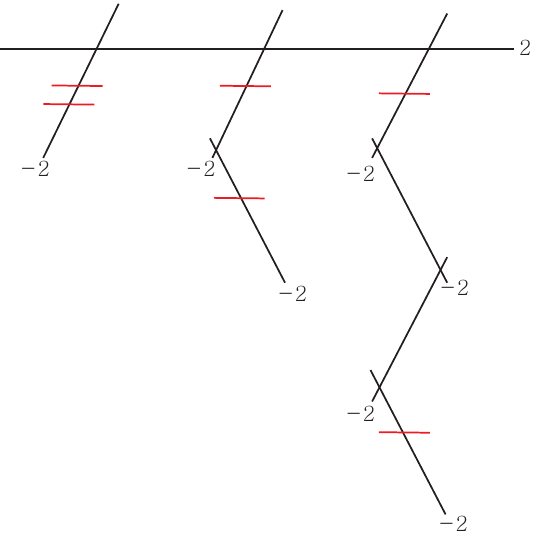}\quad
$\dim= n+6$,$\mu= 4$, BO \#207
\end{minipage}

\bibliographystyle{amsplain}

\end{document}